\documentclass[10pt]{amsart}
\allowdisplaybreaks
\usepackage{fullpage}
\usepackage{amssymb}
\usepackage[all]{xy}
\usepackage{bbm}
\usepackage{cancel}
\usepackage{graphicx}
\usepackage{mathabx}
\usepackage{tikz-cd}
\usepackage{quiver}
\usepackage{epigraph} 
\usepackage{stmaryrd}
\usepackage[inline]{enumitem}
\usepackage[mathscr]{eucal}
\usepackage{soul}
\usepackage{mathtools}
\usepackage{xfrac}
\usepackage{verbatim}

\usepackage[utf8]{inputenc}
\usepackage[T1]{fontenc}
\usepackage[backend=biber,
            isbn=false,
            doi=false,
            maxalphanames=6,
            maxnames=6,
            giveninits=true,
            style=alphabetic,
            citestyle=alphabetic]{biblatex}
\usepackage{url}
\renewbibmacro{in:}{}
\bibliography{paper.bib}

\usepackage{xcolor}
\definecolor{e-mail}{rgb}{0,.40,.80}
\definecolor{reference}{rgb}{.20,.60,.22}
\definecolor{citation}{rgb}{0,.40,.80}
\usepackage[colorlinks=true,
            linkcolor=reference,
            citecolor=citation,
            urlcolor=e-mail]{hyperref}

\usepackage{aliascnt}
\usepackage[capitalize]{cleveref}
\usepackage{adjustbox}

\theoremstyle{plain}
\newtheorem{maintheorem}{Theorem}

\newtheorem{theorem}{Theorem}[section]

\newtheorem{corollary}[theorem]{Corollary}
\newtheorem{proposition}[theorem]{Proposition}
\newtheorem{lemma}[theorem]{Lemma}
\newaliascnt{assumption}{theorem}

\aliascntresetthe{assumption}
\crefname{assumption}{assumption}{assumptions}
\theoremstyle{definition}
\newtheorem{definition}[theorem]{Definition}
\newtheorem{example}[theorem]{Example}
\newtheorem{remark}[theorem]{Remark}

\let\oldtocsection=\tocsection
\let\oldtocsubsection=\tocsubsection
\let\oldtocsubsubsection=\tocsubsubsection
\renewcommand{\tocsection}[2]{\hspace{0em}\oldtocsection{#1}{#2}}
\renewcommand{\tocsubsection}[2]{\hspace{1em}\oldtocsubsection{#1}{#2}}
\renewcommand{\tocsubsubsection}[2]{\hspace{2em}\oldtocsubsubsection{#1}{#2}}

\newcommand{\defterm}[1]{\textbf{\emph{#1}}}

\usepackage{calc}

\def\and{\quad\textup{and}\quad}

\def\cart{\ar@{}[rd]|{\Box}}

\def\counit{\mathrm{counit}}
\def\cC{\mathscr{C}}

\def\bE{\mathbb{E}}
\def\Ext{\mathrm{Ext}}

\def\bD{\mathbf{D}}
\def\Db{\mathbf{D}^{\mathrm{b}}}
\def\Dbc{\mathbf{D}^{\mathrm{b}}_{\mathrm{c}}}

\def\ex{\mathrm{ex}}
\def\Fil{\mathrm{Fil}}

\def\cF{\mathscr{F}}

\def\Hom{\mathrm{Hom}}
\def\id{\mathrm{id}}
\def\ind{\mathrm{ind}}
\def\Ind{\mathrm{Ind}}
\def\Map{\mathrm{Map}}
\def\Mod{\mathrm{Mod}}
\def\pr{\mathrm{pr}}
\def\Perf{\mathrm{Perf}}
\def\PrSt{\mathrm{Pr}^{\mathrm{St}}}

\def\bu{\mathbf{1}}
\def\unit{\mathrm{unit}}

\def\lch{\mathrm{lch}}

\def\Top{\mathrm{Top}}

\def\bD{\mathbf{D}}
\def\bbD{\mathbb{D}}

\def\Ex{\mathrm{Ex}}
\def\scF{\mathscr{F}}
\def\cF{\scF}

\def\FS{\mathrm{FS}}

\def\norm{\mathrm{norm}}
\def\Perv{\mathrm{Perv}}

\def\cP{\mathscr{P}}

\def\Shv{\mathrm{Shv}}
\def\SS{\mathrm{SS}}
\def\TS{\mathrm{TS}}

\def\rH{\mathrm{H}}
\def\HBM{\mathrm{H}^{\mathrm{BM}}}

\def\fib{\mathrm{fib}}

\def\Tot{\mathrm{Tot}}

\def\C{\mathbb{C}}

\def\Q{\mathbb{Q}}
\def\R{\mathbb{R}}
\def\Z{\mathbb{Z}}

\def\rank{\mathrm{rk}}
\renewcommand{\Re}{\operatorname{Re}}

\def\vol{\mathrm{vol}}

\def\bA{\mathbb{A}}

\def\an{\mathrm{an}}

\def\B{\mathrm{B}}
\def\Bun{\mathrm{Bun}}

\def\cl{\mathrm{cl}}
\def\et{\textnormal{\'et}}
\def\ev{\mathrm{ev}}
\def\ft{\mathrm{ft}}

\def\Ga{\mathbb{G}_{\mathrm{a}}}
\def\Gm{\mathbb{G}_{\mathrm{m}}}
\def\gr{\mathrm{gr}}

\def\bL{\mathbb{L}}
\def\cL{\mathscr{L}}
\def\ft{\mathrm{ft}}

\def\fM{\mathfrak{M}}

\def\rN{\mathrm{N}}
\def\cO{\mathscr{O}} 

\def\bP{\mathbb{P}}
\def\Pic{\mathrm{Pic}}
\def\pt{\mathrm{pt}}
\def\QCoh{\mathrm{QCoh}}

\def\red{\mathrm{red}}
\def\Res{\mathrm{Res}}
\def\Sch{\mathrm{Sch}}
\def\sep{\mathrm{sep}}

\def\Spec{\operatorname{Spec}}

\def\supp{\operatorname{supp}}

\def\Nilp{\mathrm{Nilp}}
\def\dim{\mathop{\mathrm{dim}}}

\def\ch{\mathrm{ch}}
\def\bK{\mathbf{K}}
\def\rK{\mathrm{K}}
\def\bKGL{\mathbf{KGL}}
\def\bMap{\mathbf{Map}}
\def\bMZ{\mathbf{MZ}}

\def\SH{\mathcal{SH}}
\def\Th{\mathrm{Th}}

\def\cA{\mathscr{A}}
\def\can{\mathrm{can}}
\def\Crit{\mathrm{Crit}}

\def\der{\mathrm{der}}

\def\ex{\mathrm{ex}}

\def\rN{\mathrm{N}}
\def\ori{\mathrm{or}}
\def\cS{\mathscr{S}}
\def\Sing{\mathrm{Sing}}
\def\stab{\mathsf{stab}}
\def\symp{\mathrm{symp}}
\def\T{\mathrm{T}}
\def\bT{\mathbb{T}}
\def\N{\mathrm{N}}

\def\g{\mathfrak{g}}

\def\cN{\mathcal{N}}
\def\sfP{\mathsf{P}}

\makeatletter
 
  \@addtoreset{equation}{section}
 \makeatother

\tikzcdset{scale cd/.style={every label/.append style={scale=#1}, cells={nodes={scale=#1}}}}

\begin{document}

\title{Period sheaves via perverse pullbacks}

\address{Institute of Mathematics, Academia Sinica, Taipei, Taiwan}
\address{National Center for Theoretical Sciences, National Taiwan University, Taipei, Taiwan}
\email{adeelkhan@as.edu.tw}
\author{Adeel A. Khan}
\address{Research Institute for Mathematical Sciences, Kyoto University, Kyoto 606-8502, Japan}
\email{tkinjo@kurims.kyoto-u.ac.jp}
\author{Tasuki Kinjo}
\address{June E Huh Center for Mathematical Challenges, Korea Institute for Advanced Study, Seoul, Republic of Korea}
\email{hyeonjunpark@kias.re.kr}
\author{Hyeonjun Park}
\address{School of Mathematics and Maxwell Institute for Mathematical Sciences, University of Edinburgh, Edinburgh, UK}
\email{p.safronov@ed.ac.uk}
\author{Pavel Safronov}

\begin{abstract}
We construct period sheaves for Hamiltonian spaces, as conjectured in the work of Ben-Zvi, Sakellaridis and Venkatesh, using the perverse pullback functors introduced in the authors' previous work. We prove a dimensional reduction isomorphism (generalizing the results of Davison and Kinjo in cohomological Donaldson--Thomas theory) which implies that perverse pullbacks refine the ordinary pullback functors in constructible sheaf theory, and relate perverse pullbacks to microstalk functors. These results imply that our period sheaves recover the known constructions in the cotangent and Whittaker cases.
\end{abstract}


\maketitle

\setcounter{tocdepth}{1}
\tableofcontents

\section*{Introduction}

\subsection*{Perverse pullbacks}

For a morphism $f\colon X\rightarrow B$ of complex algebraic varieties, the corresponding pullback functors $f^*, f^!\colon \Dbc(B)\rightarrow \Dbc(X)$ between bounded derived categories of constructible sheaves are, in general, not exact for the perverse $t$-structure. If $f$ is smooth, we have an isomorphism $f^*[\dim(X/B)]\cong f^![-\dim(X/B)]$ and this functor \emph{is} perverse $t$-exact. In the previous paper \cite{KKPS1} we have explained how to define perverse pullback functors beyond the smooth case by considering suitable derived enhancements. Namely, for a morphism $f\colon X\rightarrow B$ of derived Artin stacks equipped with \begin{enumerate*}
    \item a relative exact $(-1)$-shifted symplectic structure in the sense of \cite{PTVV} and
    \item an orientation datum (a choice of a square root line bundle $K_{X/B}^{1/2}$),
\end{enumerate*}
we have defined a perverse $t$-exact functor
\[f^\varphi\colon \Dbc(B)\longrightarrow \Dbc(X)\]
called the \emph{perverse pullback}. This functor interpolates between the functor of vanishing cycles (for $X=B$) and the so-called \emph{DT sheaf} $\varphi_X=f^\varphi \Q_B$ (for $B=\pt$) constructed in \cite{BBDJS, BBBBJ,KiemLi}.

The DT sheaf $\varphi_X$ is the central object of cohomological DT theory allowing one to define, for instance, Gopakumar--Vafa invariants of Calabi-Yau threefolds \cite{MaulikToda}, BPS cohomology for stacks \cite{bu2025cohomology} and microlocal homology of lci schemes \cite{Schefers}. Similarly, the perverse pullback functor is a central object for family versions of cohomological DT invariants.

\subsection*{Perverse pullbacks and microlocalization}

Microlocal sheaf theory \cite{KashiwaraSchapira} views sheaves on a manifold $M$ as local objects (microsheaves) on the cotangent bundle $\T^* M$. For instance, for a sheaf $\cF$ one has its microsupport $\SS(\cF)\subset \T^* M$ and for a submanifold $Z\subset M$ one can define its microlocalization $\mu_{Z/M}(\cF)$ which localizes the microsheaf onto the conormal bundle $\rN^*(Z/M)$.

By \cite{ParkSymplectic} the data of a relative exact $(-1)$-shifted symplectic structure on $f\colon X\rightarrow B$ is the same as the data of an exact $0$-shifted Lagrangian structure on the \emph{moment map} $\mu\colon X\rightarrow \T^* B$ and we have suggested (see \cite[Conjecture B]{KKPS1}) to view the perverse pullback functor $f^\varphi$ as a kind of a Lagrangian microlocalization functor $\mu^{\mathrm{Lag}}_{X/\T^*B}$ localizing microsheaves on $\T^* B$ onto $X$. The first result we prove in this direction is a support estimate suggested by this interpretation.

\begin{maintheorem}[\cref{prop:microsupportestimate}]\label{mainthm:microsupportestimate}
Let $\pi\colon X\rightarrow B$ be a morphism of derived Artin stacks locally of finite presentation equipped with an oriented exact relative $(-1)$-shifted symplectic structure. Let $\mu\colon X\rightarrow \T^* B$ be the corresponding moment map. Suppose $\cF\in\Dbc(B)$ is microsupported on $\Lambda\subset \T^* B$. Then $\pi^\varphi(\cF)$ is supported on $\mu^{-1}(\Lambda)$.
\end{maintheorem}

This support estimate generalizes the usual support estimate \cite[Corollary 5.4.10]{KashiwaraSchapira} for microlocalization functors $\mu_{Z/M}$ along submanifolds.

\subsection*{Dimensional reduction}

The second result we prove for perverse pullbacks is that they refine both $*$- and $!$-pullbacks. For this, recall from \cite{CalaqueCotangent} that for every geometric morphism of derived stacks $Y\rightarrow B$ locally of finite presentation the relative $(-1)$-shifted cotangent bundle $\T^*[-1](Y/B)$ (equivalently, the conormal bundle $\rN^*(Y/B)$ of $Y\rightarrow B$) carries a relative exact $(-1)$-shifted symplectic structure over $B$. Moreover, it is canonically oriented by \cite[Section 6.1]{KKPS1}. In particular, we may consider the perverse pullback along the morphism $\pi \colon \T^*[-1](Y/B)\rightarrow B$.

\begin{maintheorem}[\cref{thm:dimensionalreduction}]\label{mainthm:dimensionalreduction}
Let $g\colon Y\rightarrow B$ be a morphism of derived Artin stacks locally of finite presentation. Let $\tau\colon \T^*[-1](Y/B)\rightarrow Y$ be the natural projection. Then there are natural isomorphisms
\[\tau_! \pi^\varphi\cong g^*[\dim(Y/B)],\qquad \tau_* \pi^\varphi\cong g^![-\dim(Y/B)]\]
of functors $\Dbc(B)\rightarrow \Dbc(Y)$.
\end{maintheorem}

This theorem generalizes the dimensional reduction isomorphism for the perverse sheaf $\varphi_X$ constructed in \cite{KinjoDimred} in three ways. First, from $B=\pt$ to the relative situation. Second, from $g$ being quasi-smooth (i.e. the cotangent complex $\bL_{Y/B}$ having Tor-amplitude $\geq -1$) to a general morphism locally of finite presentation. Finally, from 1-Artin stacks to higher Artin stacks. 
In contrast with the proof in \cite{KinjoDimred}, our argument does not rely on gluing of local models. Instead, we use deformation to the normal cone to reduce the problem to the linear case where $g$ is the zero section of a perfect complex.
The generalized dimensional reduction theorem has the following applications:
\begin{enumerate}
    \item It encodes a microlocal nature of the Fourier transform in Borel--Moore homology. Namely, consider a morphism of higher Artin stacks $B\rightarrow S$, locally of finite type over $\C$, and a perfect complex $E$ over $B$. Taking $Y=E$ in \cref{mainthm:dimensionalreduction} and using the obvious isomorphism $\T^*[-1](E/B)\cong \T^*[-1](E^\vee[-1]/B)$ we obtain the \emph{Fourier--Kashiwara isomorphism} (\cref{thm-Kashiwara-Fourier})
    \[\HBM_\bullet(E/S)\cong \HBM_{\bullet-2\rank(E)}(E^\vee[-1]/S),\]
    generalizing a construction from \cite{KashiwaraCharacter}.
    Here $\rank(E)$ indicates the (virtual) rank of $E$.
    \item Let $S$ be a smooth variety and $\fM_{\Perf(S)}$ the moduli stack of compactly supported perfect complexes on $S$, which is a higher Artin stack. If $S$ is a surface and $\Tot(K_S)\rightarrow S$ the total space of its canonical bundle, then using \cite{IkedaQiu,BozecCalaqueScherotzke} we get $\T^*[-1]\fM_{\Perf(S)}\cong \fM_{\Perf(\Tot(K_S))}$. In particular, using \cref{mainthm:dimensionalreduction} we get
    \[\rH^\bullet(\fM_{\Perf(\Tot(K_S))}, \varphi_{\fM_{\Perf(\Tot(K_S))}})\cong \HBM_{\dim(\fM_{\Perf(S)})-\bullet}(\fM_{\Perf(S)})\]
    which allows one to relate statements about the threefold $\Tot(K_S)$ to statements about the surface $S$.
    \item Nadler has suggested a definition of microlocal homology of singular varieties interpolating between the usual cohomology and Borel--Moore homology. The definition of Schefers \cite{Schefers} is as follows: for a quasi-smooth derived scheme $S$ (e.g. a classical lci scheme) and a closed conic subset $\Lambda\subset \T^*[-1] S$ (equivalently, a closed conic subset of the scheme of singularities $\Sing(S)$ from \cite{ArinkinGaitsgory}) let
    \[\rH_\bullet^\Lambda(S) = \rH^{-\bullet}_\Lambda(\T^*[-1] S, \varphi_{\T^*[-1] S}),\]
    be the cohomology of the DT sheaf with support on $\Lambda$. The dimensional reduction isomorphism from \cite{KinjoDimred} implies that for $\Lambda=\{0\}$ the zero section we recover the ordinary cohomology and for $\Lambda=\T^*[-1] S$ we recover shifted Borel--Moore homology. The generalization we prove above, \cref{mainthm:dimensionalreduction}, allows one to extend microlocal homology beyond the quasi-smooth setting to arbitrary derived Artin stacks locally of finite presentation.
\end{enumerate}

\subsection*{Period sheaves}

In the recent work \cite{BZSV} Ben-Zvi, Sakellaridis and Venkatesh have studied an extension of the geometric Langlands correspondence. Beyond the equivalence of ``automorphic'' and ``spectral'' categories associated to Langlands dual groups $G, \check{G}$, it further matches specific objects on the two sides: the period sheaf associated to a Hamiltonian $G$-space $M$ and the $L$-sheaf associated to the dual Hamiltonian $\check{G}$-space $\check{M}$. Moreover, they have constructed the (normalized) period sheaf $\cP^{\norm}_X$ for a $G$-space $X$ and conjectured that it only depends on the Hamiltonian $G$-space $M=\T^* X$ with the appropriate ``anomaly trivialization'' data. As an application of the construction of perverse pullbacks, we realize this vision in full generality. Let $\bD(\Bun_G(\Sigma))$ be the $\infty$-category of ind-constructible complexes on the moduli stack of $G$-bundles on an algebraic curve $\Sigma$ and $\Nilp\subset \T^*\Bun_G(\Sigma)$ the global nilpotent cone. Let $\bD_{\Nilp}(\Bun_G(\Sigma))\subset \bD(\Bun_G(\Sigma))$ be the full subcategory of ind-constructible complexes microsupported on the global nilpotent cone. The following statement combines \cref{thm:periodanomalyfree} and \cref{def:induction}.

\begin{maintheorem}\label{mainthm:period}
Let $G=G_1\times G_2$ be a connected reductive group over $\C$, $(M, \omega)$ a quasi-separated smooth complex symplectic scheme equipped with a $G\times\Gm$-action, such that the $G$-action preserves $\omega$ and the $\Gm$-action scales it with weight $2$ and $\Sigma$ a connected smooth complex projective algebraic curve equipped with a choice of $K_\Sigma^{1/2}$ (i.e. a spin structure). Choose an anomaly trivialization, e.g. a trivialization of the \'etale second Chern class $c^{G\times\Gm}_2(\T_M\otimes\cO(1))\in|\R\Gamma_{\et}([M/G\times\Gm], \mu_2^{\otimes 2})[4]|$. Then there is an induction functor
\[\ind^{G_1\rightarrow G_2}_M\colon \bD_{\Nilp}(\Bun_{G_1}(\Sigma))\longrightarrow \bD_{\Nilp}(\Bun_{G_2}(\Sigma)).\]
\end{maintheorem}

Moreover, we relate the induction functor to the known constructions of period sheaves:
\begin{itemize}
    \item Using the dimensional reduction isomorphism, \cref{mainthm:dimensionalreduction}, we prove that for a $G$-space $X$ and $M=\T^* X$ it recovers the normalized period sheaf $\cP^{\norm}_X$ defined in \cite{BZSV} (see \cref{prop:BZSVcomparison}). For instance, if $U\subset G$ is the unipotent radical of the Borel subgroup, the induction functor for $\T^*(G/U)$ recovers the functors of constant term and Eisenstein series (see \cref{ex:Eisensteinperiod}).
    \item Using the support estimate for perverse pullbacks, \cref{mainthm:microsupportestimate}, we prove that for the Hamiltonian space $M=G\times \cS$, where $\cS\subset \g^*$ is the Kostant slice, the relevant perverse pullback functor coincides with a microstalk. Using the relationship between the Whittaker functors and microstalks established in \cite{NadlerTaylor} we get that the induction functor in this case recovers the Whittaker functional and the Whittaker sheaf from \cite{NadlerYun} (see \cref{thm:Whittakerperiod}).
\end{itemize}

Using a self-duality of $\bD_{\Nilp}(\Bun_{G_2}(\Sigma))$ (see \cite[Theorem 3.2.2]{AGKRRV2}), the induction functor $\ind_M$ is determined by a functor
\[\ind^{G_1\rightarrow G_2}_M\colon \bD_{\Nilp}(\Bun_G(\Sigma))\rightarrow \bD(\pt)\]
which we define to be the composite of the perverse pullback along the projection $\Bun_G^M(\Sigma)\rightarrow \Bun_G(\Sigma)$, where $\Bun_G^M(\Sigma)$ parametrizes twisted maps $\Sigma\rightarrow [M/G]$ (see \cref{sect:periodsetup} for details), and the $!$-pushforward along $\Bun_G^M(\Sigma)\rightarrow \pt$.

The symplectic nature of the period sheaf exhibited in \cref{mainthm:period} encodes natural symmetries invisible in the definition of the normalized period sheaf, such as the Fourier isomorphism (categorification of the functional equation) identifying period sheaves $\cP^{\norm}_X\cong \cP^{\norm}_{X^\vee}$ associated to dual $G$-representations $X$ and $X^\vee$ (see \cref{prop:categorifiedfunctionalequation} which generalizes \cite[Proposition 6.5.1]{FengWang}).

It is instructive to analyze the above construction in the case $G=\pt$, so that $\ind_M$ reduces to a chain complex. In this case $\Bun^M_G(\Sigma)=\Map^{\mathrm{twisted}}(\Sigma, M)$ is a twisted version of the space of algebraic maps $\Sigma\rightarrow M$, which carries a natural d-critical structure. Then
\[\ind^{\pt\rightarrow \pt}_M(\bu) = \R\Gamma_c(\Map^{\mathrm{twisted}}(\Sigma, M), \varphi)\]
is the usual cohomological DT invariant; this is consistent with the proposals for the state spaces in the 3d A-model in \cite{NakajimaCoulomb,SafronovWilliams}.

The notion of period sheaves is also closely related to the notion of Coulomb branches of 3d $\cN=4$ gauge theories \cite{NakajimaCoulomb,BFN}, where $M$ is a complex symplectic $G$-representation. In addition, formally replace the curve $\Sigma$ by the space
\[\Spec \C[\![t]\!]\coprod_{\Spec \C(\!(t)\!)} \Spec \C[\![t]\!],\]
so that $\bD(\Bun_G(\Sigma))$ is replaced by the derived Satake category $\bD(G[\![t]\!]\backslash G(\!(t)\!) / G[\![t]\!])$, which carries a natural $\bE_3$ structure \cite{Nocera}. In this case \cite{BDFRT} have defined a certain commutative ring object $\cA_{G, M}\in \bD(G[\![t]\!]\backslash G(\!(t)\!) / G[\![t]\!])$ in the corresponding homotopy category (or its twisted version if the anomaly cancellation condition does not hold), so that the algebra of functions on the Coulomb branch is
\[\rH^\bullet(G[\![t]\!]\backslash G(\!(t)\!) / G[\![t]\!], \cA_{G, M}).\]
See also \cite{TelemanCoulomb} for an alternative construction.

The theory we develop here does not immediately apply to the situation of Coulomb branches as the stack $G[\![t]\!]\backslash G(\!(t)\!) / G[\![t]\!]$ is not locally of finite presentation. It would be interesting to extend the theory of shifted symplectic structures to the semi-infinite setting to incorporate this example.

\subsection*{Conventions}

Throughout the paper we work with schemes over the field $k=\C$ of complex numbers. We also fix a field of coefficients $R$ for our sheaves.

\subsection*{Acknowledgements}

We would like to thank David Ben-Zvi and David Nadler for useful comments. AAK acknowledges support from the grants AS-CDA-112-M01 and NSTC 112-2628-M-001-0062030. 
TK was supported by JSPS KAKENHI Grant Number 25K17229.
HP was supported by Korea Institute for Advanced Study (SG089201).

\section{Background}

\subsection{Shifted symplectic structures}

Throughout the paper we work in the setting of derived algebraic geometry. We refer to \cite[Sections 1.1 and 1.3]{KKPS1} for our conventions. In particular, for a derived stack $X$ we have its classical truncation $X^{\cl}$ which is an ordinary (higher) stack. If $X$ is a derived Artin stack, then $X^{\cl}$ is a higher Artin stack. Finally, if $f\colon X\rightarrow Y$ is a morphism locally of finite presentation (lfp for short), then $f^{\cl}\colon X^{\cl}\rightarrow Y^{\cl}$ is locally of finite type (lft for short). For a derived stack $X$ we have the $\infty$-category $\QCoh(X)$ of quasi-coherent complexes which has the structure sheaf $\cO_X\in\QCoh(X)$. We denote $\Hom(\cO_X, -)=\R\Gamma(X, -)$ the functor of global sections. For a geometric morphism of derived stacks $\pi\colon X\rightarrow B$ we have the cotangent complex $\bL_{X/B}\in\QCoh(X)$ which is perfect if $\pi$ is locally of finite presentation; in this case we denote by $\bT_{X/B}$ its dual.

\begin{remark}
If $X$ is a stack, we can also regard it as a derived stack. However, even if $X$ is lft as a classical stack, it is not necessarily lfp as a derived stack. For instance, if $X$ is a finite type non-lci scheme, then it is not lfp as a derived stack.
\end{remark}

We will use the language of shifted symplectic structures on derived stacks as developed in \cite{PTVV}. Let us briefly recall the relevant definitions. For a chain complex $V$ we denote by $|V|$ the corresponding $\infty$-groupoid obtained via the Dold--Kan correspondence.  We have the space $\cA^{2, \ex}(X/B, n)$ of exact two-forms of degree $n$ which fits into a fiber sequence
\[
|\R\Gamma(X, \bL_{X/B})[n]|\longrightarrow \cA^{2, \ex}(X/B, n)\longrightarrow |\R\Gamma(X, \cO_X)[n+1]|
\]
of pointed $\infty$-groupoids. Concretely, the space $\cA^{2, \ex}(X/B, n)$ parametrizes pairs of a function $f$ on $X$ of degree $n$ and a nullhomotopy $d_B f\sim 0\in\bL_{X/B}$ of its de Rham differential. A \defterm{relative exact $n$-shifted symplectic structure} on an lfp geometric morphism $\pi\colon X\rightarrow B$ of derived stacks is a relative exact two-form $\omega\in\cA^{2, \ex}(X/B, n)$ whose underlying two-form induces an equivalence $\bT_{X/B}\cong \bL_{X/B}[n]$. Given a morphism of derived stacks $X\rightarrow B$ equipped with a relative exact $n$-shifted symplectic structure we denote by $\overline{X}$ the same morphism equipped with the opposite relative exact $n$-shifted symplectic structure. For a Cartesian square
\[
\xymatrix{
X' \ar[r] \ar[d] & X \ar[d] \\
B' \ar[r] & B
}
\]
and a relative exact $n$-shifted symplectic structure on $X\rightarrow B$ there is a natural pullback relative exact $n$-shifted symplectic structure on $X'\rightarrow B'$.

Given an lfp geometric morphism $f\colon L\rightarrow X$ of derived stacks over a derived stack $B$ and a relative exact $n$-shifted symplectic structure $\omega$ on $X\rightarrow B$, a \defterm{relative exact $n$-shifted Lagrangian structure} on $f$ is a nullhomotopy of $f^*\omega\in\cA^{2, \ex}(L/B, n)$ whose underlying two-form induces an equivalence $\bT_{L/X}\cong \bL_{L/B}[n-1]$. A \defterm{relative exact $n$-shifted Lagrangian correspondence} is a diagram
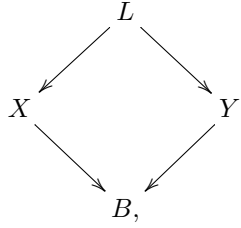
\[
\xymatrix{
& L \ar[dl] \ar[dr] & \\
X \ar[dr] && Y \ar[dl] \\
& B,
}
\]
where $X\rightarrow B$ and $Y\rightarrow B$ are equipped with relative exact $n$-shifted symplectic structures and $L\rightarrow \overline{X}\times_B Y$ is equipped with a relative exact $n$-shifted Lagrangian structure (where $\overline{X}$ denotes $X$ equipped with the opposite relative exact $n$‑shifted symplectic structure).

\begin{example}\label{ex:Lagrangianintersection}
Given a relative exact $n$-shifted symplectic structure on $X\rightarrow B$ and morphisms $f_1\colon L_1\rightarrow X$ and $f_2\colon L_2\rightarrow X$ equipped with relative exact $n$-shifted Lagrangian structures, the derived intersection $L_1\times_X L_2\rightarrow B$ carries a natural relative exact $(n-1)$-shifted symplectic structure (see \cite[Theorem 2.9]{PTVV}).
\end{example}

\begin{example}\label{ex:cotangentsymplectic}
Given an lfp geometric morphism of derived stacks $X\rightarrow B$ we have the relative $n$-shifted cotangent bundle $\T^*[n](X/B)$ which carries an exact $n$-shifted symplectic structure given by the Liouville one-form \cite{CalaqueCotangent}. Moreover, for any function $f$ of degree $n$ on $X$, the graph $\Gamma_{d_B f}\colon X\rightarrow \T^*[n](X/B)$ of its differential carries a natural relative exact $n$-shifted Lagrangian structure.
\end{example}

\begin{example}
Given an lfp geometric morphism of derived stacks $X\rightarrow B$ together with a function $f$ of degree $n$ on $X$, the \defterm{derived critical locus} $\R\Crit_{X/B}(f)$ is the derived intersection of the graph of the differential $\Gamma_{d_B f}\colon X\rightarrow \T^*[n](X/B)$ and the zero section $\Gamma_0\colon X\rightarrow \T^*[n](X/B)$. In particular, by \cref{ex:Lagrangianintersection,ex:cotangentsymplectic} $\R\Crit_{X/B}(f)\rightarrow B$ carries a natural relative exact $(n-1)$-shifted symplectic structure. We denote by $\Crit_{X/B}(f)$ the underlying classical stack.
\end{example}

We will mostly work with $(-1)$-shifted symplectic structures, for which there is the following Darboux-type result \cite{BBJ, BBBBJ, ParkSymplectic} which says that locally any relative exact $(-1)$-shifted symplectic stack is a derived critical locus.

\begin{proposition}\label{prop:Darboux}
Let $X\rightarrow B$ be an lfp geometric morphism of derived stacks equipped with a relative exact $(-1)$-shifted symplectic structure, where $B$ is a derived affine scheme. Then there is smooth morphism $U\rightarrow B$ of derived schemes equipped with a function $f\colon U\rightarrow \bA^1$ and a relative exact $(-1)$-shifted Lagrangian correspondence
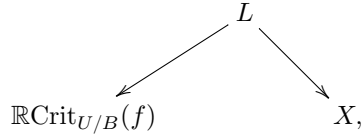
\[
\xymatrix{
& L \ar[dl] \ar[dr] & \\
\R\Crit_{U/B}(f) && X,
}
\]
where $L\rightarrow \R\Crit_{U/B}(f)$ is an equivalence on classical truncations and $L\rightarrow X$ is smooth surjective.
\end{proposition}

We will also use the following relationship between exact (relative) shifted Lagrangian structures in cotangent bundles and exact relative shifted symplectic structures. Consider lfp geometric morphisms $X\xrightarrow{\pi} B\xrightarrow{p} S$ of derived stacks, so that we obtain a fiber sequence
\[
\pi^*\bL_{B/S}\longrightarrow \bL_{X/S}\longrightarrow \bL_{X/B}.
\]
Given a relative exact $n$-shifted symplectic structure $\omega$ on $X\rightarrow B$ specified by a function $f$ of degree $n+1$ and a nullhomotopy $d_B f\sim 0\in\bL_{X/B}$ the above fiber sequence implies that $d_S f$ lifts to a section of $\pi^*\bL_{B/S}$. So, $\omega$ gives rise to a morphism
\[\mu\colon X\rightarrow \T^*[n+1] (B/S)\]
which we call the \defterm{moment map} associated to $\omega$. In this case we define the \defterm{symplectic pushforward} $p_* X\rightarrow S$ to be the Lagrangian intersection
\[
\xymatrix{
p_*X \ar[r] \ar[d] & X \ar^{\mu}[d] \\
B \ar^-{\Gamma_0}[r] &  \T^*[n+1] (B/S)
}
\]

\begin{proposition}{\cite[Corollary 3.1.3]{ParkSymplectic}}\label{prop:momentmap}
Fix lfp geometric morphisms $X\rightarrow B\rightarrow S$ of derived stacks. The above assignment $(X\rightarrow B, \omega)\mapsto (\mu\colon X\rightarrow \T^*[n+1](B/S))$ establishes an equivalence between the space of $B$-relative exact $n$-shifted symplectic structures on $X$ and the space of morphisms $\mu\colon X\rightarrow \T^*[n+1](B/S)$ lifting $X\rightarrow B$ equipped with an $S$-relative exact $(n+1)$-shifted Lagrangian structure.
\end{proposition}

\begin{example}
Consider a smooth scheme $X$ equipped with an exact symplectic structure $\omega=d \alpha$. For an algebraic group $G$ with an action on $X$ which preserves $\alpha$ the morphism $[X/G]\rightarrow \B G$ carries a relative exact $0$-shifted symplectic structure. Denoting by $\g$ the Lie algebra of $G$, the above moment map $\mu\colon [X/G]\rightarrow [\g^*/G]$ is given by the natural moment map $X\rightarrow \g^*$ associated to this situation given by $\iota_{a(-)} \alpha$ for $a\colon \g\rightarrow \Gamma(X, \T_X)$ the action map. In this case the symplectic pushforward of $[X/G]\rightarrow \B G$ along $p\colon \B G\rightarrow \pt$ is equivalent to the Hamiltonian reduction $X/\!/G$.
\end{example}

For a morphism $X\rightarrow B$ of derived stacks equipped with a $(-1)$-shifted symplectic structure $\omega$ an \defterm{orientation} is the choice of a $\Z/2$-graded line bundle $\cL$ together with an isomorphism $\cL^{\otimes 2}\cong K_{X/B}\coloneqq \det(\bL_{X/B})$. Orientations of $(X\rightarrow B, \omega)$ naturally form a groupoid and can be pulled back along base changes $B'\rightarrow B$.

\begin{example}\label{ex:ocan}
For an lfp geometric morphism $U\rightarrow B$ of derived stacks equipped with a function $f\colon U\rightarrow \bA^1$ the derived critical locus $\R\Crit_{U/B}(f)\rightarrow B$ carries the \emph{canonical orientation} $o^{\can}_{\R\Crit_{U/B}(f)/B}$ as in \cite[Section 6.1]{KKPS1}. In particular, for $f=0$ we get the canonical orientation $o^{\can}_{\T^*[-1](U/B)/B}$ of $\T^*[-1](U/B)\rightarrow B$.
\end{example}

For a relative exact $(-1)$-shifted Lagrangian correspondence $Y\xleftarrow{q_Y} L\xrightarrow{q_X} X$ we have a natural isomorphism
\[\Upsilon^{\der}_{(q_X, q_Y)}\colon K_{X/B}|_L\otimes K_{L/X}^{\otimes 2}\xrightarrow{\sim}K_{Y/B}|_L\]
as in \cite[(3.5)]{KPS}. In particular, given such a correspondence with $q_X$ an isomorphism on classical truncations, we can pullback orientations of $X\rightarrow B$ to orientations of $Y\rightarrow B$.

\subsection{Ind-constructible sheaves}

Let $\PrSt_R$ be the $\infty$-category of stable presentable $R$-linear $\infty$-categories, where $R$ is a field. For a locally compact Hausdorff topological space $X\in\Top^{\lch}$ we denote by $\Shv(X; R)\in\PrSt_R$ the $\infty$-category of sheaves of complexes of $R$-modules, which admits a 6-functor formalism \cite{KashiwaraSchapira, Volpe, KhanWeavelisse}. We denote by $\Db(X)\subset \Shv(X; R)$ the full subcategory consisting of complexes with bounded cohomology groups.

For separated scheme of finite type $S\in\Sch^{\sep\ft}$ (over the field of complex numbers)  we denote by $S^{\an}$ the underlying topological space considered with respect to analytic topology. We denote by $\Dbc(S)\subset \Shv(S^{\an}; R)$ the full subcategory of algebraically constructible complexes and let
\[\bD(S) = \Ind(\Dbc(S))\]
be the $\infty$-category of ind-constructible sheaves. As explained in \cite[Section 2.2]{KKPS1} the functor $\bD(-)$ admits an extension to lft higher Artin stacks which comes with a 6-functor formalism. For an lft higher Artin stack an object $\cF\in\bD(X)$ is constructible if for every scheme $S\in\Sch^{\sep\ft}$ together with a morphism $f\colon S\rightarrow X$ the pullback $f^*\cF\in\bD(S)$ is constructible. We denote by $\Dbc(X)\subset \bD(X)$ the full subcategory of constructible sheaves. For a derived stack $X$ with $X^{\cl}$ an lft higher Artin stack we denote $\bD(X) = \bD(X^{\cl})$.

For a smooth morphism $f\colon X\rightarrow Y$ of lft higher Artin stacks we denote
\[f^\dag\coloneqq f^*[\dim(X/Y)]\cong f^![-\dim(X/Y)]\colon \bD(Y)\longrightarrow \bD(X).\]
This functor is perverse $t$-exact. We sometimes refer to the isomorphism above, defined in \cite[Proposition 2.14]{KKPS1}, as the \emph{purity} isomorphism.

For a scheme $S\in\Sch^{\sep\ft}$ and a function $f\colon S\rightarrow \bA^1$ we denote by
\[\varphi_f\colon \Shv(S^{\an}; R)\longrightarrow \Shv(f^{-1}(0)^{\an}; R)\]
the functor of vanishing cycles normalized so that it preserves perverse sheaves. We denote by
\[\phi_f = \bigoplus_{c\in\C} (f^{-1}(c)\rightarrow S)_*\varphi_{f-c}\colon \Dbc(S)\longrightarrow \Dbc(S)\]
its extension to other values.

We are ready to summarize the main construction of \cite{KKPS1}.

\begin{theorem}\label{thm:perversepullback}
Consider an lfp geometric morphism $\pi\colon X\rightarrow B$ of derived stacks equipped with an oriented exact relative $(-1)$-shifted symplectic structure such that $B^{\cl}$ is an lft higher Artin stack. Then there is a colimit-preserving \defterm{perverse pullback} functor
\[\pi^\varphi\colon \bD(B)\longrightarrow \bD(X)\]
depending functorially on the choice of the orientation. It comes with the following natural isomorphisms:
\begin{enumerate}
    \item If $B\in\Sch^{\sep\ft}$, $g\colon U\rightarrow B$ is a smooth morphism and $X=\R\Crit_{U/B}(f)$ with its canonical orientation, then we have a natural isomorphism
    \[\pi^\varphi(-)\cong (\phi_f\circ g^\dag)(-)|_X.\]
    \item For a Cartesian square
    \[
    \xymatrix{
    X' \ar^{\overline{p}}[r] \ar^{\pi'}[d] & X \ar^{\pi}[d] \\
    B' \ar^{p}[r] & B
    }
    \]
    with $p\colon B'\rightarrow B$ smooth and considering the pullback oriented relative exact $(-1)$-shifted symplectic structure on $X'\rightarrow B'$ we have a natural isomorphism
    \begin{equation}\label{eq:alphabasechange}
    (\pi')^\varphi p^\dag\cong \overline{p}^\dag\pi^\varphi.
    \end{equation}
    \item For a relative exact $(-1)$-shifted Lagrangian correspondence $Y\xleftarrow{q_Y} L\xrightarrow{q_X} X$ with $q_Y$ an isomorphism on classical truncations and $q_X$ smooth and considering the pullback oriented relative exact $(-1)$-shifted symplectic structure on $\pi'\colon Y\rightarrow B$ we have a natural isomorphism
    \begin{equation}\label{eq:alphapullback}
    (\pi')^\varphi\cong (q_Y)_* (q_X)^\dag \pi^\varphi.
    \end{equation}
    \item For a Cartesian square
    \[
    \xymatrix{
    \tilde{X} \ar^{\tilde{c}}[r] \ar^{\tilde{\pi}}[d] & X \ar^{\pi}[d] \\
    \tilde{B} \ar^{c}[r] & B
    }
    \]
    with $c\colon \tilde{B}\rightarrow B$ finite and considering the pullback oriented relative exact $(-1)$-shifted symplectic structure on $\tilde{X}\rightarrow \tilde{B}$ we have a natural isomorphism
    \begin{equation}\label{eq-perverse-pullback-symplectic-finite-base-change-compatibility}
    \beta_{c} \colon \pi^{\varphi} c_* \xrightarrow[]{\sim} \tilde{c}_* \tilde{\pi}^{\varphi}.
    \end{equation}
    \item For an lfp geometric morphism $p\colon B\rightarrow S$ such that $S^{\cl}$ is an lft higher Artin stack, consider the symplectic pushforward which fits into a commutative diagram
    \[\begin{tikzcd}
    	{p_*(X, \omega)} & X \\
    	S & {B.}
    	\arrow["i", from=1-1, to=1-2]
    	\arrow["{\bar{\pi}}", from=1-1, to=2-1]
    	\arrow["\pi", from=1-2, to=2-2]
    	\arrow["p", from=2-2, to=2-1]
    \end{tikzcd}\]
    If $p$ is smooth, there exists a natural isomorphism
    \begin{equation}\label{eq-symp-push-perverse-pullback-smooth}
    \gamma_{p} \colon \pi^{\varphi} p^{\dagger} \xrightarrow{\sim} i_* \bar{\pi}^{\varphi}
    \end{equation}
    and if $p$ is a closed immersion, there exists a natural isomorphism
    \begin{equation}\label{eq-symp-push-perverse-pullback}
    \varepsilon_p\colon \overline{\pi}^\varphi p_* \xrightarrow{\sim} i^\dag \pi^\varphi.
    \end{equation}
    \item For two morphisms $\pi_1\colon X_1\rightarrow B_1$ and $\pi_2\colon X_2\rightarrow B_2$ as above, equip $\pi_1\times \pi_2\colon X_1\times X_2\rightarrow B_1\times B_2$ with the product oriented relative exact $(-1)$-shifted symplectic structure. Then there is a natural isomorphism
    \begin{equation}\label{eq:TS}
    \TS \colon \pi_1^{\varphi} (-) \boxtimes \pi_2^{\varphi} (-) \xrightarrow{\sim} (\pi_1 \times \pi_2)^{\varphi}(- \boxtimes -).    
    \end{equation}
\end{enumerate}
\end{theorem}

\begin{remark}
The Darboux theorem (\cref{prop:Darboux}) implies that perverse pullback is uniquely determined by properties (1)-(3) in \cref{thm:perversepullback}.
\end{remark}

It will be convenient to combine the isomorphisms \eqref{eq:alphabasechange} and \eqref{eq:alphapullback} from properties (2) and (3) in \cref{thm:perversepullback} as follows (see \cite[Proposition 6.9]{KKPS1}). Consider the following data:
\begin{itemize}
    \item $B$ is an lft higher Artin stack and $p\colon B'\rightarrow B$ is a smooth geometric morphism.
    \item $\pi\colon X\rightarrow B$ and $\pi'\colon X'\rightarrow B'$ are lfp geometric morphisms of derived stacks equipped with relative exact $(-1)$-shifted symplectic structures.
    \item $X'\xleftarrow{q'} L\xrightarrow{\tilde{q}} X\times_B B'$ is a Lagrangian correspondence over $B'$ with $\tilde{q}$ smooth, and let $q\coloneqq \pr_1\circ \tilde{q}\colon L\rightarrow X$ be the composite.
    \item An orientation on $X\rightarrow B$ which by pullback induces an orientation on $X'\rightarrow B'$.
\end{itemize}

Then there is a natural isomorphism
\begin{equation}\label{eq-Lagrangian functoriality of perverse pullback}
\alpha_{p, (q, q')} \colon (\pi')^\varphi p^{\dagger}
\xrightarrow{\sim} q'_{*} q^\dag \pi^\varphi
\end{equation}

\subsection{Nearby cycles}

In this section we recall various functors of nearby cycles, following \cite[\S 1.1.1]{Schuermann} (see also \cite[\S 2.3]{KinjoVirtual} for a review). For a holomorphic function $t \colon X \to \C$ on a complex analytic space $X$, denote by $X_0$ the fiber over $\{0\}$.
The \defterm{nearby cycles functor} $\psi_t \colon \Shv(X; R)\rightarrow \Shv(X_0; R)$ is defined by
\[\psi_t = (X_0 \hookrightarrow X)^*(X_{\Re>0}\hookrightarrow X)_*(X_{\Re>0}\hookrightarrow X)^*(-),\]
where $X_{\Re>0} = \{x\in X\mid \Re(t(x))> 0\}$.
This sits in a canonical fiber sequence
\begin{equation}\label{eq:psi phi triangle}
    i^* \to \psi_t \to \varphi_t[1],
\end{equation}
where $i \colon X_0 \hookrightarrow X$ is the inclusion (see \cite[(1.7)]{Schuermann}).
See also \cite[Corollary 1.1.1]{Schuermann} for the comparison with the more standard definition of $\psi_t$.
In particular,  $\psi_t$ preserves constructible objects since $\varphi_t$ and $i^*$ do.
For a field extension $R \to R'$ there is a natural morphism
\begin{equation}\label{eq-nearby-scalar}
\psi_{t}(\cF) \otimes_{R} R' \to \psi_{t}(\cF \otimes_{R} R'),
\end{equation}
which by \cite[Proposition 4.2.5]{Achar} is invertible as long as $\cF \in \Dbc(X)$.

We also recall the \emph{unipotent nearby cycles functor} following \cite{BeilinsonGlue}, and \cite[\S 3.1]{Campbell} (see also \cite[\S 3.1]{ChenNearby} or \cite[\S 5.2.2]{bu2025cohomology}).
Denote by $j \colon X_{\neq 0} \hookrightarrow X$ the complementary open immersion to $i$.
The functor $i^* j_* \colon \Shv(X_{\neq 0}; R) \to \Shv(X_{0}; R)$ naturally upgrades to a functor
\[
    (i^* j_*)^{\mathrm{enh}} \colon \Shv(X_{\neq 0}; R) \to \Lambda\textnormal{-Mod} (\Shv(X_{0}; R))
\]
valued in modules over the cochains algebra $\Lambda \coloneqq \mathrm{C}^\bullet(\mathbb{G}_{\mathrm{m}}; R)$.
Restriction along the unit section $1\colon \mathrm{pt} \hookrightarrow \mathbb{G}_{\mathrm{m}}$ determines a canonical augmentation $\Lambda \to R$, and the \defterm{unipotent nearby cycles functor} $\psi_{t}^{\mathrm{uni}} \colon \Shv(X; R) \to \Shv(X_0; R)$ is the composite
\[
    \psi_{t}^{\mathrm{uni}} \colon
    \Shv(X; R) \xrightarrow{j^*}
    \Shv(X_{\neq 0}; R) \xrightarrow{(i^* j_*)^{\mathrm{enh}}}
    \Lambda\textnormal{-Mod} (\Shv(X_{0}; R)) \xrightarrow{(-)\otimes_{\Lambda} R[-1]}
    \Shv(X_0; R).
\]

We have (see \cite[Remark after Corollary 3.2]{BeilinsonGlue}):

\begin{proposition}\label{prop:psi from psiuni}
    Suppose that $R$ is an algebraically closed field.
    For every $c \in R^\times$, denote by $\mathscr{L}_c$ the local system of rank one on $\mathbb{G}_{\mathrm{m}}$ whose monodromy operator is multiplication by $c$.
    Then for every $\cF \in \Dbc(X)$, there is a natural isomorphism
    \[
        \psi_{t} (\cF) \cong \bigoplus_{c \in R^\times} \psi_t^{\mathrm{uni}} (\cF \otimes t^*\mathscr{L}_c).
    \]
\end{proposition}
\begin{proof}
    Let $T$ denote the monodromy operator of $\psi_t(\cF)$.
    Since $R$ is algebraically closed, $\psi_t(\cF)$ admits a functorial decomposition into subsheaves $\psi_t(\cF)^{(c)}$ where $T-c$ acts nilpotently (for $c \in R^\times$), such that $\psi_t(\cF)_1 \cong \psi_t^{\mathrm{uni}}(\cF)$; see e.g. \cite[Proposition 1.1]{MorelNearbyGluing}.
    Since the monodromy operator of $\psi_t(\cF \otimes t^*\cL_c)$ is $c \cdot T$, this decomposition can be written
    \begin{equation*}
        \psi_t(\cF) \cong
        \bigoplus_{c \in R^\times} \psi_t(\cF)^{(c^{-1})} \cong
        \bigoplus_{c \in R^\times} \psi_t^{\mathrm{uni}}(\cF \otimes t^*\cL_c)
    \end{equation*}
    as claimed.
\end{proof}

\begin{remark}
    When the complex analytic space $X$ is a scheme, the unipotent nearby cycles functor (in contrast to the ``full'' nearby cycles functor $\psi_t$) only relies on the algebraic structure of the complex line $\C$.
    In particular, $\psi_t^{\mathrm{uni}}$ can be defined in more general sheaf theories, such as those of $\ell$-adic and motivic sheaves (cf. \cite{cass2024central}).
\end{remark}

\section{Microsupport and perverse pullbacks}

In this section we prove an estimate for the support of perverse pullbacks in terms of microsupport.

\subsection{Microsupport}

Let $X$ be a complex manifold. For $\cF\in\Db(X)$ a complex of sheaves on $X$ recall the notion of its \defterm{microsupport} $\SS(\cF)\subset \T^* X$ from \cite[Chapter V]{KashiwaraSchapira}. For a conic subset $\Lambda\subset \T^* X$ let $\Db_\Lambda(X)\subset \Db(X)$ be the full subcategory of sheaves microsupported on $\Lambda$.

Now let $S$ be a smooth affine scheme. If $\cF\in\Dbc(S)\subset \Shv(S^{\an}; R)$, its microsupport is a conic Lagrangian subset of $\T^* S$, see \cite[Theorem 8.5.5]{KashiwaraSchapira}. For a conic subset $\Lambda\subset \T^* S$ we denote by $\Perv_\Lambda(S)\subset \Perv(S)$ the full subcategory consisting of objects with microsupport contained in $\Lambda$. We denote by $\bD_\Lambda(S)\subset \bD(S)$ the full subcategory of objects whose perverse cohomology sheaves lie in $\Ind(\Perv_\Lambda(S))$.

For a smooth morphism $f\colon T\rightarrow S$ of smooth affine schemes consider the correspondence
\[
\xymatrix{
\T^* T & T\times_S \T^* S \ar_-{f^\ast}[l] \ar^-{f_\pi}[r] & \T^* S.
}
\]
The following is \cite[Proposition 5.4.5(i)]{KashiwaraSchapira}.

\begin{proposition}\label{prop:SSpullback}
For a smooth morphism $f\colon T\rightarrow S$ of smooth affine schemes and $\cF\in\Dbc(S)$ we have $\SS(f^*\cF) = f^\ast(f^{-1}_\pi(\SS(\cF)))$.
\end{proposition}

This allows us to make the following definition following \cite[Section F.6]{AGKRRV1}.

\begin{definition}
Let $X$ be a smooth higher Artin stack and $\Lambda\subset \T^* X$ a conic subset. $\bD_\Lambda(X)\subset \bD(X)$ is the full subcategory consisting of objects $\cF\in\bD(X)$ such that for every smooth morphism $f\colon S\rightarrow X$ for $S$ an affine scheme and $\Lambda_S = f^\ast(f^{-1}_\pi(\Lambda))\subset \T^* S$ we have $f^*\cF\in\bD_{\Lambda_S}(S)$.
\end{definition}

Let us now explain how to extend the notion of microsupport to singular varieties following \cite[Appendix E.6]{AGKRRV1}. Suppose $S$ is a derived affine scheme of finite presentation and $S^{\cl}$ is its underlying classical scheme. In this case $\T^* S$ is an lfp derived Artin stack and we can talk about closed subsets of $\T^* S$. Explicitly: writing a resolution of $\bL_S|_{S^{\cl}}$ by vector bundles $E_0\leftarrow E_{-1}\leftarrow \dots$, a closed subset of $\T^* S$ is the same as a closed subset of $E_0$ which is invariant under the action of $E_{-1}$. It is shown in \cite[E.6.4]{AGKRRV1} that for $\cF\in\Dbc(S)$ one can attach its microsupport $\SS(\cF)\subset \T^* S$ uniquely characterized by the following properties:
\begin{enumerate}
    \item The formation of microsupport is Zariski local on $S$.
    \item For any closed immersion $i\colon S\rightarrow S'$, where $S'$ is a smooth affine scheme, we have $i_\pi((i^*)^{-1}(\SS(\cF))) = \SS(i_*\cF)$.
\end{enumerate}

The above notion of microsupport is compatible with smooth morphisms as in \cref{prop:SSpullback} and therefore extends to lfp derived Artin stacks.

In \cite[Corollary 5.4.10]{KashiwaraSchapira} Kashiwara and Schapira give an estimate for the support of microlocalization along a complex submanifold in terms of microsupport of the original sheaf. We will now prove an analogous estimate for perverse pullbacks.

\begin{proposition}\label{prop:microsupportestimate}
Let $\pi\colon X\rightarrow B$ be a morphism of lfp derived Artin stacks equipped with an oriented exact relative $(-1)$-shifted symplectic structure. Let $\mu\colon X\rightarrow \T^* B$ be the corresponding moment map. Consider a conic subset $\Lambda\subset \T^* B$ and $\cF\in\bD_\Lambda(B)$. Then
\[\supp \pi^\varphi(\cF)\subset \mu^{-1}(\Lambda).\]
\end{proposition}
\begin{proof}
The claim is smooth-local on $B$, so we may assume that $B$ is a derived affine scheme of finite presentation. Since $\bD(B) = \Ind(\Dbc(B))$, we may assume that $\cF\in\Dbc(B)$ with $\SS(\cF)\subset \Lambda$. Applying the Darboux theorem (\cref{prop:Darboux}) and \cref{thm:perversepullback}(3) we may reduce to the following case: there is a smooth morphism $p\colon U\rightarrow B$ together with a function $f\colon U\rightarrow \bA^1$ with $X=\R\Crit_{U/B}(f)$. By \cite[Proposition 1.15]{KKPS1} we may locally find a diagram
\[
\xymatrix{
U \ar^{g}[dr] \ar^{p}[ddr] && \\
& \overline{U} \ar^{\overline{i}}[r] \ar^{\overline{p}}[d] & U' \ar^{p'}[d] \\
& B \ar^{i}[r] & B',
}
\]
with the square Cartesian, $U'$ and $B'$ smooth affine schemes, $i$ and $U\rightarrow U'$ closed immersions and $g$ \'etale. Moreover, we may find a function $f'\colon U'\rightarrow \bA^1$ with $f'|_U=f$. Denote $\overline{f} = f'|_{\overline{U}}$. Let $\pi'\colon X'=\R\Crit_{U'/B'}(f')\rightarrow B'$, $\overline{\pi}\colon \overline{X}=\R\Crit_{\overline{U}/B}(\overline{f})\rightarrow B$ and $\mu'\colon X'\rightarrow \T^* B'$ the corresponding moment map, so that we have a commutative diagram
\[
\xymatrix{
\T^* B'\times_{B'} B \ar^{i^*}[r] & \T^* B \\
X'\times_{B'} B\cong \overline{X} \ar^{\mu'}[u] & X. \ar^{\mu}[u] \ar[l]
}
\]

Applying proper base change we have $(\pi')^\varphi i_*\cF\cong \overline{i}_*\overline{\pi}^\varphi \cF$ and applying smooth base change we get $g^*\overline{\pi}^\varphi \cF\cong \pi^\varphi\cF$. Therefore,
\[\supp(\pi^\varphi\cF)\subset g^{-1}(\supp(\overline{\pi}^\varphi\cF))\subset g^{-1}(\supp((\pi')^\varphi i_*\cF) \cap \overline{U}).\]
Therefore, the support estimate for $X\rightarrow B$ follows from the corresponding support estimate for $X'\rightarrow B'$.

Thus, it is enough to assume $B$ is smooth and hence $U$ is also smooth. By \cite[Equation (8.6.12)]{KashiwaraSchapira} we have
\[\supp(\phi_f p^*\cF)\subset \{x\in U\mid (x, df(x))\in\SS(p^*\cF)\}.\]
By \cref{prop:SSpullback} we have
\[\SS(p^*\cF) = \{(x, \alpha\in\T^*_{p(x)} B)\mid (p(x), \alpha)\in\SS(\cF)\}.\]
The moment map $\mu\colon X\rightarrow \T^* B$ sends $x\in X$ to the pair $\mu(x) = (p(x), \alpha\in \T^*_{p(x)} B)$, where $p^*\alpha = df(x)$. The claim follows.
\end{proof}

\subsection{Microstalks}

For a real manifold $X$ and a smooth function $t\colon X\rightarrow \R$ one can define the vanishing cycles functor $\varphi_t\colon \Shv(X; R)\rightarrow \Shv(t^{-1}(0); R)$ by
\[\varphi_t = (t^{-1}(0)\rightarrow X_{\leq 0})^*(X_{\leq 0}\rightarrow X)^!,\]
where $X_{\leq 0} = \{x\in X\mid t(x)\leq 0\}$. The definition is compatible with the functor of vanishing cycles of holomorphic functions, so that for a holomorphic function $t\colon X\rightarrow \C$ on a complex manifold we have a natural isomorphism
\[\varphi_t(-)\cong \varphi_{\Re t}(-)|_{t^{-1}(0)}.\]

Now suppose $X$ is a complex manifold, $\Lambda$ is a (complex) conic Lagrangian submanifold of $\T^* X$, $p\in\Lambda$ and $t\colon X\rightarrow \R$ a smooth function. Let $\Gamma_{dt}$ be the graph of $dt$, $\pi\colon \T^* X\rightarrow X$ the projection and $\tau_t(p)\in \Z$ is the \defterm{triple Maslov index} of the (real) Lagrangian subspaces $\T_p \pi^{-1}\pi(p)$, $\T_p \Lambda$ and $\T_p \Gamma_{dt}$ of $\T_p \T^* X$ as in \cite[Equation (7.5.3)]{KashiwaraSchapira}. If $e$ is the dimension of the intersection $\T_p\Lambda$ and $\T_p \Gamma_{dt}$, then $\tau_t(p) + e$ is even.

\begin{proposition}\label{prop:microstalk}
Let $X$ be a complex manifold, $\Lambda\subset \T^* X$ a (complex) conic Lagrangian subset and $p\in\Lambda$ a smooth point with $\pi(p) = x\in X$. For a smooth function $t\colon X\rightarrow \R$ with $t(x) = 0$ and such that $\Gamma_{dt}$ and $\Lambda$ intersect cleanly at $p$ with excess $e\in\Z$, the functor
\[\varphi_t(-)|_x[-(\tau_t(p)+e)/2]\colon \Db_\Lambda(X)\longrightarrow \Mod_R\]
is independent of the choice of the function $t$, up to a (non-canonical) natural isomorphism.
\end{proposition}
\begin{proof}
The proof is similar to the proof of \cite[Proposition 7.5.3]{KashiwaraSchapira} which treats the case $e=0$ (i.e. when $\Gamma_{dt}$ and $\Lambda$ intersect transversely at $p$). Fix $\cF\in\Db_\Lambda(X)$.

First, by \cite[Corollary A.2.7 and Lemma 7.5.2]{KashiwaraSchapira} we may apply a quantized contact transformation to assume that $\Lambda$, in a neighborhood of $p$, is the conormal bundle of a complex submanifold $i\colon M\hookrightarrow X$. By \cite[Proposition 6.6.1]{KashiwaraSchapira} working microlocally near $p\in\T^* X$ we may find $L\in\Mod_R$ such that $\varphi_t(\cF)|_x\cong \varphi_t(i_* L_M)|_x$, where $L_M$ is the constant sheaf on $M$ with stalk $L$. By proper base change we have $\varphi_t(i_* L_M)|_x\cong \varphi_{t|_M}(L_M)|_x$.

The assumption on the clean intersection of $\Gamma_{dt}$ and $\Lambda$ implies that $t|_M\colon M\rightarrow \R$ is Morse--Bott with critical locus a submanifold $S\subset M$ of dimension $e$ at $x$. Let $Q$ be the Hessian of $t|_M$ at $x$ viewed as a quadratic form on $V=\T_x M$. Let $n_+, n_-$ be the number of positive and negative eigenvalues of $Q$. By the Morse--Bott lemma (see \cite[Proposition 2.5.2]{Nicolaescu}), locally near $x$ we may find a diffeomorphism $\Phi\colon V\rightarrow M$ sending $0$ to $x$ under which $\Phi^\ast (t|_M) = Q$. Then $\varphi_{t|_M}(L_M)|_x\cong \varphi_Q(L_V)|_0$. The standard computation of vanishing cycles of a quadratic form shows that $\varphi_Q(L_V)|_0\cong L[-n_+]$. By \cite[Proposition A.3.6]{KashiwaraSchapira} we have
\[\tau_t(p) = n_- - n_+.\]
Therefore, in total we get
\[\varphi_t(\cF)|_x[-(\tau_t(p)+e)/2]\cong L[-\dim M/2]\]
which is independent of the choice of $t$.
\end{proof}

The functor
\[\mu_p\colon \Db_\Lambda(X)\rightarrow \Mod_R\]
defined in \cref{prop:microstalk} is the \defterm{microstalk} functor at $p\in\Lambda$. We have a smooth functoriality of microstalk functors as follows.

\begin{proposition}\label{prop:microstalkpullback}
Let $f\colon Y\rightarrow X$ be a smooth morphism of complex manifolds and $\Lambda\subset \T^* X$ a (complex) conic Lagrangian subset. Let $(y, p)\in Y\times_X \T^* X$ be a point such that $p\in\Lambda$ is smooth and let $p' = f^\ast(y, p)$. Then there is a (non-canonical) natural isomorphism
\[\mu_{p'} f^\dag\cong \mu_p\colon \Db_\Lambda(X)\longrightarrow \Mod_R.\]
\end{proposition}
\begin{proof}
Let $x$ be the image of $p$ under $\T^* X\rightarrow X$ and $t\colon X\rightarrow \R$ a smooth function such that $\Gamma_{dt}$ and $\Lambda$ intersect cleanly at $p$ with excess $e$. Let $t' = f^\ast t$ and $\Lambda_Y = f^\ast(f^{-1}_\pi(\Lambda))\subset \T^* Y$. Then $\Gamma_{dt'}$ and $\Lambda_Y$ intersect cleanly at $p'$ with excess $e'=e + 2\dim(Y/X)_y$. Using the additivity property of the Maslov index from \cite[Theorem A.3.2(vii)]{KashiwaraSchapira} we get that $\tau_{t'}(p') = \tau_t(p)$. Therefore,
\begin{align*}
\mu_p(-)&\cong \varphi_t(-)|_x[-(\tau_t(p)+e)/2] \\
&\cong (f^\dag\varphi_t(-))|_y[-(\tau_t(p)+e)/2 - \dim(Y/X)_y] \\
&\cong (\varphi_{t'} f^\dag(-))|_y[-(\tau_{t'}(p')+e')/2] \\
&\cong \mu_{p'} f^\dag(-),
\end{align*}
where we have used smooth base change in the third line.
\end{proof}

\section{Fourier transforms via perverse pullbacks}\label{sect:Fourier}

In this section we relate perverse pullbacks to Fourier transforms for sheaves.

\subsection{Conic and monodromic sheaves}

In this section we recall the notions of conic and monodromic sheaves and state the contraction principle. It is convenient to phrase the definitions in the setting of topological stacks, i.e. presheaves of $\infty$-groupoids on $\Top^{\lch}$ satisfying \v{C}ech descent along surjective topological submersions. Analogously to the algebraic situation we may also define topological higher Artin stacks; we refer to \cite[Section 8.3]{KhanWeavelisse} for more details. The 6-functor formalism $\Shv(-; R)$ of sheaves of complexes of $R$-modules on $\Top^{\lch}$ extends to topological higher Artin stacks.

Consider the group $\R_{>0}$ of positive real numbers under multiplication.

\begin{definition}
Let $X$ be a topological higher Artin stack equipped with an $\R_{>0}$-action. A sheaf $\cF\in\Shv(X; R)$ is \defterm{conic} if it is $\R_{>0}$-equivariant.
\end{definition}

We denote by $\Shv_{\R_{>0}}(X; R)\subset \Shv(X; R)$ the full subcategory of conic sheaves.

\begin{remark}
Since $\R_{>0}$ is contractible, pullback along $X\rightarrow [X/\R_{>0}]$ is fully faithful and identifies $\Shv([X/\R_{>0}]; R)$ with the full subcategory of $\Shv(X; R)$ consisting of conic sheaves.
\end{remark}

\begin{remark}
If $X\in\Top^{\lch}$ is a topological space, a sheaf $\cF$ is conic if, and only if, its restriction to $\R_{>0}$-orbits is locally constant.
\end{remark}

The key result about conic sheaves is the \emph{contraction principle}. To introduce it, consider the topological monoid $\R_{\geq 0}$ of nonnegative real numbers. It contains the subgroup $\R_{>0}\subset \R_{\geq 0}$ as the group of units.

\begin{definition}\label{def:contracting}
Let $X$ be a topological stack equipped with an $\R_{>0}$-action. The $\R_{>0}$-action is \defterm{contracting} if it extends to an action of the monoid $\R_{\geq 0}$.
\end{definition}

For a topological stack $X$ equipped with a contracting $\R_{>0}$-action let $X^0$ be the topological stack of $\R_{\geq 0}$-equivariant maps $\{0\}\rightarrow X$. We have natural morphisms
\[i\colon X^0\longrightarrow X,\qquad r\colon X\longrightarrow X^0\]
such that $r\circ i=\id_{X^0}$. The proof of the following \emph{stacky contraction principle} can be given as in \cite[Appendix C]{Drinfeld_compact_generation} by replacing algebraic stacks with topological higher Artin stacks and $D$-modules with $\Shv(-; R)$.

\begin{proposition}\label{prop:contraction}
Let $X$ be a topological higher Artin stack with a contracting $\R_{>0}$-action such that the contracting locus $X^0$ is still a topological higher Artin stack. Then the morphisms
\[i^!\cong r_!i_! i^!\xrightarrow{\counit} r_!,\qquad r_*\xrightarrow{\unit} r_* i_*i^*\cong i^*\]
are isomorphisms on $\Shv_{\R_{>0}}(X; R)$.
\end{proposition}

Let us now switch to the setting of complex algebraic stacks. Given a higher Artin stack $X$, we have the corresponding topological higher Artin stack $X^{\an}$. Moreover, a $\Gm$-action on $X$ gives rise to a $\C^\times$-action on $X^{\an}$. In particular, $X^{\an}$ carries an $\R_{>0}$-action. Therefore, we may transfer the definition of conic sheaves to the algebraic setting as follows.

\begin{definition}
Let $X$ be a higher Artin stack equipped with a $\Gm$-action. A sheaf $\cF\in\Dbc(X)$ is \defterm{monodromic} if the underlying sheaf on $X^{\an}$ is conic with the respect to the $\R_{>0}$-action.
\end{definition}

\begin{remark}
For $X\in\Sch^{\sep\ft}$ a sheaf $\cF\in\Dbc(X)$ is monodromic in the above sense if, and only if, its restriction to $\Gm$-orbits is locally constant.
\end{remark}

Consider the monoid scheme $\bA^1$ considered with respect to multiplication, which contains $\Gm$ as an open group subscheme of its units.

\begin{definition}
Let $X$ be a stack. A $\Gm$-action on $X$ is \defterm{contracting} if the $\Gm$-action extends to an action of the monoid $\bA^1$.
\end{definition}

Clearly, given a contracting $\Gm$-action on a stack $X$, the underlying $\R_{>0}$-action on $X^{\an}$ is contracting. In particular, the contraction principle applies to monodromic sheaves on higher Artin stacks with contracting $\Gm$-actions.

\subsection{Exponential sheaves}

In this section we study exponential sheaves and their relationship to vanishing cycle functors. Consider the scaling action of $\R_{>0}$ on $\R$ and let $[\R/\R_{>0}]$ be the corresponding quotient stack. The real \defterm{exponential sheaf} is the $\R$-constructible sheaf
\[\exp_\R = (\R_{\geq 0}\rightarrow [\R/\R_{>0}])_!\bu_{\R_{\geq 0}}\in\Shv([\R/\R_{>0}]; R).\]
We also denote by $\exp_\R\in\Shv_{\R_{>0}}(\R; R)$ the corresponding conic sheaf on $\R$.

There is a complex version of the exponential sheaf defined as follows. For the $\Gm$-action on $\bA^1$ of weight $n>0$ we denote by $[\bA^1/_n\Gm]$ the corresponding quotient stack. Consider the maps
\[i_1\colon \pt\xrightarrow{p_1} \B\mu_n\xrightarrow{j}[\bA^1_n/\Gm],\]
where $j$ is the complement of the origin and $i_1$ is the inclusion of $1\in\C$. The complex \defterm{exponential sheaf} is
\[\exp_n = (i_1)_* \bu[-1]\in\Dbc([\bA^1/_n\Gm]).\]

\begin{remark}
To explain the analogy, note that if $i_1\colon \pt\rightarrow [\R/\R_{>0}]$ is the inclusion of $1$, then one can show that $\exp_\R\cong (i_1)_* \bu[-1]$ using the localization sequence.
\end{remark}

\begin{proposition}\label{prop:expproperties}
The exponential sheaves satisfy the following properties:
\begin{enumerate}
    \item Consider the correspondence of topological stacks
    \[
    \xymatrix{
    [\R/\R_{>0}] & [\C/\R_{>0}] \ar_-{f_1}[l] \ar^-{f_2}[r] & [\C/_n\C^\times],
    }
    \]
    where $f_1$ is induced by the real part map $\Re\colon \C\rightarrow \R$ and $f_2$ is induced by the $n$-th power map $\R_{>0}\rightarrow \C^\times$. Then there is an isomorphism
    \[(f_2)_! f_1^* \exp_\R\cong \exp_n.\]
    \item Let
    \[s_d\colon [\bA^1/_{nd} \Gm]\longrightarrow [\bA^1/_n\Gm]\]
    be the map given by the identity on $\bA^1$ and the $d$-th power map on $\Gm$. Then there is an isomorphism
    \begin{equation}\label{eq:exppushforward}
    (s_d)_! \exp_{nd}\cong \exp_n.
    \end{equation}
    \item Assume $R$ has characteristic not $2$ and let $\bu^-_{\B\mu_2}$ be the sign representation of $\B\mu_2$ viewed as a one-dimensional local system on $\B\mu_2$. Then
    \[\exp_2\cong (j_!\bu^-_{\B\mu_2}\oplus j_*\bu_{\B\mu_2})[-1].\]
\end{enumerate}
\end{proposition}
\begin{proof}$ $
\begin{enumerate}
    \item Let $\C_{\Re\geq 0}\subset \C$ be the subset of complex numbers with nonnegative real part. Using base change we have to show that
    \[\exp'_n = (\C_{\Re\geq 0}\rightarrow [\C/_n\C^\times])_! \bu\]
    is isomorphic to $\exp_n$. Consider the inclusion of the origin $i_0\colon \B\C^\times\rightarrow [\C/_n\C^\times]$. Let us compute the $!$-pullbacks of $\exp'_n$ along $i_0$ and its complement $j$:
    \begin{enumerate}
        \item We prove that $i_0^! \exp'_n = 0$. The pullback of $\exp'_n$ along $\C\rightarrow [\C/\C^\times]$ is conic. Applying the contraction principle (\cref{prop:contraction}) on this cover the claim is equivalent to showing that $([\C/\C^\times]\rightarrow \B\C^\times)_! \exp'_n = 0$. But this follows from the vanishing of the compactly supported cohomology $\rH^\bullet_c(\C_{\geq 0})$.
        \item We prove that $j^*\exp'\cong (p_1)_!\bu[-1]$. By base change we get $j^*\exp'\cong (p_1)_! \R\Gamma_c(\C_{\geq 0}\setminus \{0\}, \bu)$. Using the localization fiber sequence and the vanishing of the compactly supported cohomology of $\C_{\geq 0}$ we get $\R\Gamma_c(\C_{\geq 0}\setminus \{0\}, \bu)\cong \bu[-1]$.
    \end{enumerate}
    The above claims imply that $\exp'_n\cong j_* p_!\bu[-1]\cong j_*p_*\bu[-1] = \exp_n$.
    \item Since $s_d$ is finite, the claim is equivalent to the statement that $(s_d)_* \exp_{nd}\cong \exp_n$ which follows from the compatibility of $i_1$ with the maps $s_d$.
    \item We may factor $i_1$ as $\pt\xrightarrow{p_1} \B\mu_2\xrightarrow{j} [\bA^1/_2\Gm]$. Therefore,
    \[\exp_2\cong (j_*\bu^-_{\B\mu_2}\oplus j_*\bu_{\B\mu_2})[-1].\]
    Next we will show that the natural morphism $j_!\bu^-_{\B\mu_2}\rightarrow j_*\bu^-_{\B\mu_2}$ is an isomorphism. This is equivalent to showing that the $!$-pullback of $j_!\bu^-_{\B\mu_2}$ along the inclusion of the origin $\B\Gm\rightarrow [\bA^1/_2\Gm]$ is zero. By the stacky contraction principle \cite[Appendix C]{Drinfeld_compact_generation} the $!$-pullback along $\B\Gm\rightarrow [\bA^1/_2\Gm]$ is isomorphic to the $!$-pushforward along $[\bA^1/_2\Gm]\rightarrow \B\Gm$. But the $!$-pushforward of $\bu^{-1}_{\B\mu_2}$ along $\B\mu_2\rightarrow \B\Gm$ is zero.
\end{enumerate}
\end{proof}

\begin{remark}
\Cref{prop:expproperties}(3) identifies the above definition of the exponential sheaf with the one considered in \cite[Definition 10.5.1(c)]{BZSV} in the case $n=2$.
\end{remark}

\begin{remark}
The tower $\{\exp_n\in\bD([\bA^1/_n\Gm])\}$ of exponential sheaves together with isomorphisms \eqref{eq:exppushforward} may be thought of as an approximation to a (non-existent) exponential sheaf on $\bA^1$. This is motivated by the following two results:
\begin{enumerate}
    \item Let $R=\C$ and consider the exponential $D$-module on $\bA^1$. Then its $!$-pushforward along $\bA^1\rightarrow [\bA^1/_n\Gm]$ is isomorphic to $\exp_n$ via the Riemann--Hilbert correspondence.
    \item Consider the setting of schemes over a field of characteristic $p>0$ and let $R=\Q_{\ell}$ for a prime $\ell\neq p$. Choose an additive character $\psi\colon \mathbb{F}_p\rightarrow R^\times$ and let $\cL_\psi$ be the associated Artin--Schreier $\ell$-adic sheaf on $\bA^1$. Then its $!$-pushforward along $\bA^1\rightarrow [\bA^1/_n\Gm]$ is isomorphic to $\exp_n$ in the setting of $\ell$-adic sheaves; see \cite[Lemme 2.3]{Laumon} for the case $n=1$.
\end{enumerate}
\end{remark}

\begin{proposition}\label{prop:exponentialvanishing}
Let $X\in\Sch^{\sep\ft}$ be a scheme equipped with a contracting $\Gm$-action and a function $t\colon X\rightarrow \bA^1$ of weight $n>0$, so that it descends to a morphism $\overline{t}\colon [X/\Gm]\rightarrow [\bA^1/_n\Gm]$. Let $X^0\subset X$ be the fixed point locus. Then there are the following natural isomorphisms:
\begin{enumerate}
    \item For a monodromic $\cF\in\Dbc(X)$ we have
    \[(t^{-1}(0)\rightarrow X^0)_! \varphi_t(\cF)\cong (X\rightarrow X^0)_!(\cF\otimes (\Re t)^*\exp_\R).\]
    \item For $\cF\in\Dbc([X/\Gm])$ we have
    \[(t^{-1}(0)\rightarrow X^0)_! \varphi_t(X\rightarrow [X/\Gm])^*\cF\cong ([X/\Gm]\rightarrow X^0)_!(\cF\otimes \overline{t}^*\exp_n).\]
\end{enumerate}
\end{proposition}
\begin{proof}
Let
\[X_{\Re\geq 0}=\{x\in X^{\an}\mid \Re t(x)\geq 0\}.\]
For a monodromic $\cF\in\Dbc(X)$ let $\overline{\cF}\in\Shv([X^{\an}/\R_{>0}]; R)$ be the corresponding sheaf on the quotient stack. We have a sequence of natural isomorphisms
\begin{align*}
(t^{-1}(0)\rightarrow X^0)_! \varphi_t\cF &\cong (t^{-1}(0)\rightarrow X^0)_!(t^{-1}(0)\rightarrow X_{\Re\geq 0})^! (X_{\Re\geq 0}\rightarrow X^{\an})^* (X^{\an}\rightarrow [X^{\an}/\R_{>0}])^*\overline{\cF} \\
&\cong (X_{\Re\geq 0}\rightarrow X^0)_! (X_{\Re\geq 0}\rightarrow [X^{\an}/\R_{>0}])^*\overline{\cF} \\
&\cong ([X^{\an}/\R_{>0}]\rightarrow X^0)_! (\overline{\cF}\otimes t^*\exp_\R),
\end{align*}
where the first isomorphism uses Verdier self-duality of vanishing cycles, the second isomorphism uses the contraction principle (\cref{prop:contraction}) and the third isomorphism uses base change.

The second statement follows from the first statement and the relationship between $\exp_\R$ and $\exp_n$ established in \cref{prop:expproperties}(1).
\end{proof}

\subsection{Fourier transforms for sheaves}

Let $B\in\Top^{\lch}$ be a topological space and $E\rightarrow B$ a real vector bundle. Consider the natural pairing
\[\ev^\R\colon [E/\R_{>0}]\times_B [E^\vee/\R_{>0}]\longrightarrow [\R/\R_{>0}]\]
and denote by
\[\pr_1^\R\colon [E/\R_{>0}]\times_B [E^{\vee}/\R_{>0}]\longrightarrow [E/\R_{>0}],\qquad \pr_2^\R\colon [E/\R_{>0}]\times_B [E^{\vee}/\R_{>0}]\longrightarrow [E^{\vee}/\R_{>0}]\]
the projections. We refer to \cite[Chapter 3.7]{KashiwaraSchapira} and \cite[Section 2.4]{KinjoVirtual} for more details on the following notion.

\begin{definition}
The \defterm{Fourier--Sato transform} $\FS\colon \Shv([E/\R_{>0}]; R)\rightarrow \Shv([E^\vee/\R_{>0}]; R)$ is
\[\FS(\cF)= \pr^\R_{2, !}(\pr_1^{\R, *}(\cF)\otimes \ev^{\R, *}\exp_\R).\]
\end{definition}

Let us now relate the Fourier--Sato transform to perverse pullbacks. For this, consider an lfp derived Artin stack $B$ and a vector bundle $E\rightarrow B$. Let
\[\ev\colon E\times_B E^\vee\longrightarrow \bA^1,\qquad \pr_1\colon E\times_B E^\vee\longrightarrow E,\qquad \pr_2\colon E\times_B E^\vee\longrightarrow E^\vee\]
be the corresponding morphisms.

The projection $\pr_2$ is smooth and
\[\R\Crit_{E\times_B E^\vee/E}(\ev)\cong E^\vee.\]
In particular, the morphism
\[u_{E^\vee/E}\colon E^\vee\xrightarrow{p_{E^\vee}} B\xrightarrow{0_E} E\]
carries a natural oriented relative exact $(-1)$-shifted symplectic structure.

\begin{proposition}\label{prop:FSperversepullback}
Let $B$ be an lft higher Artin stack and $E\rightarrow B$ a vector bundle. For a monodromic complex $\cF\in\Dbc(E)$ there is a natural isomorphism
\[u_{E^\vee/E}^\varphi(\cF)\cong \FS(\cF)[\rank E].\]
\end{proposition}
\begin{proof}
Using base change for the Fourier--Sato transform and perverse pullbacks we reduce to the case when $B$ is an affine scheme of finite type. By \cref{thm:perversepullback}(1) we have
\[u_{E^\vee/E}^\varphi(\cF)\cong (\phi_{\ev}\pr_1^\dag\cF)|_{\{0\}\times E^\vee}.\]
Since $(\phi_{\ev}\pr_1^\dag\cF)$ is supported on $\{0\}\times E^\vee$ where $\ev$ takes the zero value, we have
\[u_{E^\vee/E}^\varphi(\cF)\cong (\ev^{-1}(0)\rightarrow E^\vee)_!\varphi_{\ev} \pr_1^\dag\cF.\]
Consider the $\Gm$-action on $E\times_B E^\vee$ given by the weight $1$ action on $E$ and the weight $0$ action on $E^\vee$. The fixed point locus of the action is given by $E^\vee$ and with respect to this $\Gm$-action the function $\ev\colon E\times_B E^\vee\rightarrow \bA^1$ has weight $1$. The claim then follows from \cref{prop:exponentialvanishing}(1).
\end{proof}

\begin{remark}
A closely related statement to \cref{prop:FSperversepullback} is shown in \cite[Lemma 5.2]{Schefers}.
\end{remark}

\section{Dimensional reduction for perverse pullbacks}

In this section we analyze perverse pullbacks for relative $(-1)$-shifted cotangent bundles and prove the following relative dimensional reduction theorem.

\begin{theorem}\label{thm:dimensionalreduction}
    Let $g \colon Y\rightarrow B$ be an lfp geometric morphism of derived stacks and suppose that $B^{\cl}$ is an lft higher Artin stack. Consider the morphism $\pi \colon \T^*[-1](Y/B)\to B$ equipped with its canonical oriented relative exact $(-1)$-shifted symplectic structure. Then there is a canonical isomorphism
    \begin{equation}\label{eq-dimrensional-reudction}
    \red_{Y/B}\colon \tau_* \pi^\varphi\cong g^![-\dim(Y/B)]
    \end{equation}
    In particular, applying Verdier duality we get an isomorphism
    \[
    \tau_! \pi^\varphi\cong g^*[\dim(Y/B)].
    \]
\end{theorem}

Under the tautological identification $\bD(B) = \bD(B^{\cl})$, the statement for a derived stack $B$ is equivalent to the statement for its classical truncation, so from now on we assume that $B$ is classical.

\subsection{Dimensional reduction for vanishing cycle functors}

We first prove the dimensional reduction theorem for local models of relative $(-1)$-shifted cotangent stacks.
This case is due to Davison \cite[Theorem A.1]{DavisonCritical}, and here we give an alternative proof which has the advantage of applying in the context of $\ell$-adic sheaves.
We start with the following lemma, which may be regarded as a slight strengthening of the dimensional reduction theorem:

\begin{lemma}\label{lem-pre-local-dim-red}
    Let $X \in \Sch^{\sep\ft}$ be a separated scheme with a contracting $\Gm$-action, so that $X^+ = X$ and the inclusion of the fixed locus $X^0$ admits a retract $r \colon X \to X^0$ (see \cref{def:contracting}).
    Let $t \colon X \to \mathbb{A}^1$ be a regular function of positive weight $m$ and $\iota \colon Z \hookrightarrow X^{0}$ a closed immersion such that $t |_{r^{-1}(Z)} = 0$.
    Then for any $\mathbb{G}_m$-equivariant constructible complex $\cF \in \Dbc(X)$, we have $\iota^! r_* \psi_t(\cF) = 0$.
\end{lemma}

\begin{proof}
    We first prove the analogous statement for the unipotent nearby part $\psi_t^{\mathrm{uni}}(\cF)$.
    Let $i \colon X_0 \hookrightarrow X$ and $j \colon X_{\neq 0} \hookrightarrow X$ be the inclusions.
    By the definition of $\psi_t^{\mathrm{uni}}$, it is enough to prove the vanishing
    \[
      \iota^! r_* i_* i^* j_* j^* \cF = 0.
    \]
    Denoting by $s \colon X^{0} \hookrightarrow X$ the inclusion of the fixed locus, using \cref{prop:contraction} twice yields natural isomorphisms
    \begin{equation*}
      r_* i_* i^* \cong 
      s^* i_* i^* \cong 
      s^*  \cong 
      r_*
    \end{equation*}
    on $\Gm$-equivariant complexes, since $X^0 \subset X_0$.
    We thus deduce the isomorphism
    \[
        \iota^! r_* i_* i^* j_* j^* \cF \cong
        \iota^! r_*  j_* j^* \cF \cong
        (r |_{p^{-1}(Z)})_*  (\tilde{\iota})^!  j_* j^* \cF,
    \]
    where $\tilde{\iota} \colon p^{-1}(Z) \hookrightarrow X$ is the base change of $\iota \colon Z \hookrightarrow X^0$.
    But since $r^{-1}(Z) \subset X$ is disjoint to $X_{\neq 0}$, the last term vanishes as desired.
    
    We now consider the full nearby cycles functor $\psi_t$.
    By extending scalars and using the invertibility of \eqref{eq-nearby-scalar}, we may assume that the coefficient ring $R$ is an algebraically closed field.
    Then by \cref{prop:psi from psiuni}, it is enough to prove the vanishing
    \[
     \iota^! r_* i_* i^* j_* j^* (\cF \otimes t^* \mathscr{L}_c) = 0
    \]
    for every $c \in R^\times$, where $\mathscr{L}_c$ denotes the rank one local system on $\mathbb{G}_{\mathrm{m}}$ with monodromy operator $c$.
    For this we may repeat the argument in the first paragraph.
\end{proof}

\begin{corollary}[Local dimensional reduction]\label{cor-local-dim-red}
    Let $X \in \Sch^{\sep\ft}$ be a scheme and $E$ a vector bundle on $X$.
    We let $s \in \Gamma(X, E)$ be a section and $s^{\vee} \colon E^{\vee} \to \mathbb{A}^1$ the corresponding cosection.
    We denote by $p \colon E^{\vee} \to X$ the projection and by $\iota \colon Z \hookrightarrow X$ the inclusion of the zero locus of $s$.
    Then for $\mathcal{F} \in \Dbc(X)$, the maps
    \[
    p_! \varphi_{s^{\vee}} p^! \mathcal{F} \xrightarrow{\unit}
    p_! \varphi_{s^{\vee}} p^! \iota_* \iota^* \mathcal{F}
    \cong \iota_* \iota^* \mathcal{F}
    , \qquad
    \iota_! \iota^!\mathcal{F} \cong p_* \varphi_{s^{\vee}} p^* \iota_! \iota^!\mathcal{F} \xrightarrow{\counit} p_* \varphi_{s^{\vee}} p^* \mathcal{F}
    \]
    are invertible, where the isomorphisms follow from proper base change.
\end{corollary}

\begin{proof}
    Since these maps of constructible complexes are exchanged by Verdier duality, it is enough to show the second one is invertible. By \cite[Proposition 2.19(1)]{KKPS1} the support of the complex $p_* \varphi_{s^{\vee}} p^* \mathcal{F}$ lies in $Z$.
    Thus it is enough to prove that the map
    \[
    \iota_! \iota^!p_* \varphi_{s^{\vee}} p^* \iota_! \iota^!\mathcal{F} \xrightarrow{\counit} \iota_! \iota^! p_* \varphi_{s^{\vee}} p^* \mathcal{F}
    \]
    is invertible.
    Denoting by $K$ the zero locus of $s^\vee$ and by $i_K \colon K \hookrightarrow E^\vee$ the inclusion, the counit $\iota_!\iota^! \to \id$ induces a map of fiber sequences \eqref{eq:psi phi triangle}
    \[\begin{tikzcd}
	{\iota_! \iota^!p_* i_{K,*}i_K^* p^* \iota_! \iota^!\mathcal{F}} & {\iota_! \iota^!p_* \psi_{s^{\vee}} p^* \iota_! \iota^!\mathcal{F}} & {\iota_! \iota^!p_* \varphi_{s^{\vee}} p^* \iota_! \iota^!\mathcal{F}[1]} \\
	{\iota_! \iota^! p_* i_{K,*} i_K^* p^* \mathcal{F}} & {\iota_! \iota^! p_* \psi_{s^{\vee}} p^* \mathcal{F}} & {\iota_! \iota^! p_* \varphi_{s^{\vee}} p^* \mathcal{F}[1].}
	\arrow[from=1-1, to=1-2]
	\arrow[from=1-1, to=2-1]
	\arrow[from=1-2, to=1-3]
	\arrow[from=1-2, to=2-2]
	\arrow[from=1-3, to=2-3]
	\arrow[from=2-1, to=2-2]
	\arrow[from=2-2, to=2-3]
    \end{tikzcd}\]
    By \cref{lem-pre-local-dim-red}, the complexes in the middle are zero, so it is enough to show the invertibility of the left-hand vertical map.
    Let $0 \colon X \hookrightarrow E^\vee$ denote the zero section.
    Then \cref{prop:contraction} implies an isomorphism
    \[
    p_* i_{K,*}i_K^* p^*  \cong 0^* i_{K,*}i_K^* p^* \cong 0^* p^* \cong \id.
    \]
    In particular, the left-hand vertical map is identified with the counit map $\iota_! \iota^! \iota_! \iota^! \to \iota_! \iota^!$, which is clearly invertible.
    Hence we conclude that the right-hand vertical map is invertible as desired.
\end{proof}

\subsection{Dimensional reduction for zero loci}

We now reinterpret the local dimensional reduction isomorphism (\cref{cor-local-dim-red}) in the language of the usual dimensional reduction statement \cref{thm:dimensionalreduction}.
This relies on the following lemma:

\begin{lemma}\label{lem-crit-locus-vs-shifted-cotangent}
    Let $f \colon Y \to B$ be an lfp geometric morphism of derived stacks such that $B^{\cl}$ is an lft higher Artin stack.
    Let $E$ be a perfect complex over $Y$ and $s \in \Gamma(Y, E)$ a section.
    Denote by $Z(s)$ the zero locus and by $s^{\vee} \colon E^{\vee} \to \mathbb{A}^1$ the corresponding cosection.
    Then the projection $p \colon E^\vee \to B$ restricts to $\R\Crit_{E^\vee/B}(s^\vee) \to Z(s)$ and exhibits its source as the relative $(-1)$-shifted cotangent bundle of $Z(s)$ over $B$.
    Moreover, this isomorphism of derived stacks over $Z(s)$
    \[
        \T^*[-1](Z(s) / B) \cong \R\Crit_{E^{\vee} / B} (s^{\vee})
    \]
    is compatible with relative $(-1)$-shifted symplectic structures and there is a canonical isomorphism between canonical orientations.
\end{lemma}

\begin{proof}
    This is a $(-1)$-shifted symplectic version of \cite[Corollary 5.3]{kiem2025cosection}. Though the proof is similar,
    we repeat the proof for reader's convenience.
    Consider a $1$-shifted symmetric complex $E \oplus E^{\vee}[1]$ and take a section $(s, 0)$ equipped with the canonical isotropic structure $s \cdot 0 \to 0$.
    By definition, the symplectic zero locus $Z_{Y/Y}^{\mathrm{symp}}(E \oplus E^{\vee}[1],(s, 0))$ is the zero locus $Z(E \oplus E^{\vee}[1], (s, 0))$ equipped with an exact $(-1)$-shifted symplectic structure over $U$ introduced in \cite[Definition 3.2.3]{ParkSymplectic}.
    Using \cite[Eq. (21)]{ParkSymplectic} twice, we obtain equivalences of exact $(-1)$-shifted symplectic stacks
    \[
    Z_{U/U}^{\mathrm{symp}}(E \oplus E^{\vee}[1],(s, 0)) \cong \T^*[-1](Z(s) / Y), \quad  Z_{U/U}^{\mathrm{symp}}(E \oplus E^{\vee}[1], (s, 0)) \cong \R\Crit_{E^{\vee} / Y} (s^{\vee}).
    \]
    In particular, we have
    \begin{equation}\label{eq-equiv-symplectic-over-Y}
    \T^*[-1](Z(s) / Y) \cong \R\Crit_{E^{\vee} / Y} (s^{\vee}).
    \end{equation}
    By considering the symplectic pushforward along $f$, we obtain the desired equivalence of exact $(-1)$-shifted symplectic stacks over $B$.

We now compare the orientations. 
To see this, it is enough to show that the isomorphism \cref{eq-equiv-symplectic-over-Y}
 preserves the canonical orientation.    
 Consider the following Cartesian diagram
\[\begin{tikzcd}
	{\T^*[-1](Z(s) / Y) } & {\R\Crit_{E^{\vee} / Y} (s^{\vee})} & {E^{\vee}} \\
	{Z(s)} && {Y.}
	\arrow["\cong", from=1-1, to=1-2]
	\arrow[from=1-1, to=2-1]
	\arrow[from=1-2, to=1-3]
	\arrow[from=1-3, to=2-3]
	\arrow[from=2-1, to=2-3]
\end{tikzcd}\]
 Set $X \coloneqq \T^*[-1](Z(s) / Y) \cong \R\Crit_{E^{\vee} / Y} (s^{\vee})$.
 The above Cartesian diagram implies an isomorphism
 \[
 \mathbb{L}_{X / Y} \cong E|_X \oplus E^{\vee}[1] |_X.
 \]
 By definition, the canonical orientation $o^{\can}_{\T^*[-1](Z(s) / Y)/Y}$ is given by the composite
 \begin{align*}
 \left(K_{Z(s) / Y} |_X\right)^{\otimes 2} 
 &\cong \det(E^{\vee}[1] |_X)^{\otimes 2} \xrightarrow[\cong]{\id \otimes \theta_{E}} \det(E^{\vee}[1] |_X) \otimes \det(E^{\vee} |_X)^{\vee} 
  \\
 &\xrightarrow[\cong]{\id \otimes \theta_{E^{\vee}}} \det(E^{\vee}[1] |_X) \otimes \det(E |_X) \cong \det(E|_X \oplus E^{\vee}[1] |_X) \cong \det(\bL_{X/Y}) = K_{X/Y}.
 \end{align*}
 On the other hand, the canonical orientation
 $o^{\can}_{\R\Crit_{E^{\vee} / Y} (s^{\vee})/Y}$ is given by the composite 
 \begin{align*}
    \left( K_{E^{\vee} / Y} |_X \right)^{\otimes 2} &\cong \det(E|_X)^{\otimes 2} \cong 
    \det(E|_X) \otimes \det((E^{\vee}[1])^{\vee}[1]|_X)
    \xrightarrow[\cong]{\id \otimes \theta_{(E^{\vee}[1]|_X)^{\vee}}} \det(E|_X) \otimes \det((E^{\vee}[1]|_X)^{\vee})^{\vee} \\
    &\xrightarrow[\cong]{\id \otimes \iota_{E^{\vee}[1]}} \det(E|_X) \otimes \det(E^{\vee}[1]|_X) \cong \det(E^{\vee}[1]|_X \oplus E|_X) \cong \det(\bL_{X/Y}) = K_{X/Y}.
 \end{align*}
  We claim that these maps are identified under the isomorphism
  \[
  K_{Z(s) / Y} |_X \cong \det(E^{\vee}[1]|_X) \xrightarrow[\cong]{\theta_{E^{\vee}}} \det(E^{\vee} |_X)^{\vee} \xrightarrow[\cong]{\iota_{E}} \det(E|_X) \cong K_{E^{\vee} / Y} |_X.
  \]
 This is the consequence of the following two facts:
 \begin{itemize}
     \item The swapping isomorphism
     \[
     \det(E) \otimes \det(E) \xrightarrow[\cong]{\mathrm{swap}} \det(E) \otimes \det(E)
     \]
     is given by the multiplication by $(-1)^{\rank (E)}$.
     \item The following diagram commutes:
    \[\begin{tikzcd}
	{\det((E^{\vee}[1])^{\vee}[1])} & {\det((E^{\vee}[1])^{\vee})^{\vee}} & {\det(E^{\vee}[1])} & {\det(E[1])^{\vee}} \\
	{\det(E)} &&& {\det(E)}
	\arrow["{\theta_{(E[1]^{\vee})^{\vee}}}", from=1-1, to=1-2]
	\arrow[from=1-1, to=2-1]
	\arrow["{\iota_{E^{\vee}[1]}}", from=1-2, to=1-3]
	\arrow["{\theta_{E[1]}}", from=1-3, to=1-4]
	\arrow["{\iota_{E^{\vee}}}", from=1-4, to=2-4]
	\arrow["{(-1)^{\rank (E)} \cdot \id}"', from=2-1, to=2-4]
    \end{tikzcd}\]
    This is a consequence of the commutativity of the diagrams \cite[(2.16), (2.19)]{KPS}.
 \end{itemize}


\end{proof}

\begin{corollary}\label{cor:dim red for crit}
    Let $f \colon Y \to B$ be a smooth morphism between separated schemes of finite type, $E$ a vector bundle over $Y$, and $s \in \Gamma(Y, E)$ a section.
    Denote by $g \colon  Z(s) \to B$ the natural map, by $\tau \colon \T^*[-1](Z(s)/B) \to Z(s)$ the relative $(-1)$-shifted cotangent bundle, and by $\pi = g \circ \tau \colon\T^*[-1](Z(s)/B) \to B$ the composite, equipped with its relative $(-1)$-shifted symplectic structure and the canonical orientation.
    Then there is a canonical isomorphism
    \begin{equation}\label{eq-dim-red-local-model}
        \tau_* \pi^{\varphi} \cong g^! [-\dim(Y/B)].
    \end{equation}
\end{corollary}
\begin{proof}
    By \cref{lem-crit-locus-vs-shifted-cotangent} and \cref{thm:perversepullback}(1) we have a canonical isomorphism $\pi^\varphi \cong (\phi_{s^\vee} p_{E^\vee}^\dag f^{\dagger}(-))|_{\Crit_{E^\vee/B}(s^\vee)}$.
    We let $\iota \colon Z(s) \hookrightarrow Y$ denote the canonical inclusion. Pushing forward along $\tau$ yields
    \begin{align*}
        \tau_* \pi^{\varphi}
        &\cong \tau_* (\varphi_{s^\vee} p_{E^\vee}^\dag f^{\dagger}(-)|_{\Crit_{E^\vee/B}(s^\vee)})\\
        &\cong \iota^*p_{E^{\vee},*} \varphi_{s^{\vee}} p_{E^{\vee}}^{\dagger} f^{\dagger}\\
        &\cong \iota^*p_{E^{\vee},*} \varphi_{s^{\vee}} p_{E^{\vee}}^* f^{!} [-\dim(Z(s)/B)]\\
        &\cong \iota^* \iota_! \iota^! f^{!} [-\dim(Z(s)/B)] \cong g^! [-\dim(Z(s)/B)].
    \end{align*}
    For the second isomorphism, we have used that $\tau_* (-|_{\Crit_{E^\vee/B}(s^\vee)}) \cong g^* p_{E^\vee,*}$ on objects supported in $\Crit_{E^\vee/B}(s^\vee) \subset E^\vee$.
    The third isomorphism is the purity isomorphism.
    The fourth isomorphism is local dimensional reduction (\cref{cor-local-dim-red}).
\end{proof}

\subsection{Lagrangian correspondences}

Before proceeding it will be useful to record the following functoriality statements.
Given a commutative square of derived stacks
\[\begin{tikzcd}
    Y' \ar{r}{g'}\ar{d}{p_Y}
    & B' \ar{d}{p_B}
    \\
    Y \ar{r}{g}
    & B,
\end{tikzcd}\]
where $g$ and $g'$ are lfp  and geometric, consider the relative $(-1)$-shifted cotangent bundles
\[
    \tau \coloneqq \tau_{Y/B} \colon \T^*[-1](Y/B) \to Y,
    \qquad
    \tau' \coloneqq \tau_{Y'/B'} \colon \T^*[-1](Y'/B') \to Y'.
\]

\begin{lemma}
    There exists a canonical relative oriented exact $(-1)$-shifted Lagrangian correspondence
    \begin{equation}\label{eq:Lagrangian corr}
    \begin{tikzcd}
        &  \T^*[-1](Y/B) \times_{Y} Y' \ar[swap]{ld}{\partial p}\ar{rd}{p_\tau} & 
        \\
        \T^*[-1](Y'/B') & &  \T^*[-1](Y/B) \times_{B} B',
    \end{tikzcd}
\end{equation}
where $p_\tau$ is induced by $(g, g')$ and $\partial p$ is induced by the map $d(p_Y) \colon p_Y^* \bL_{Y/B} \to \bL_{Y'/B'}$.
\end{lemma}

\begin{proof}
It will be useful to phrase the proof using the $\infty$-category $\mathrm{Symp}^{\mathrm{ex}}_{B,-1}$ (for $B$ a derived stack) from \cite[Section 2]{ParkSymplectic} which has the following description:
\begin{itemize}
    \item Its objects are lfp geometric morphisms $\pi \colon X \to B$ equipped with a relative exact $(-1)$-shifted symplectic structure $\omega$.
    \item Its morphisms $(X_1, \omega_1) \dashrightarrow (X_2, \omega_2)$ are relative exact $(-1)$-shifted Lagrangian correspondences $X_2 \gets L \to X_1$.
\end{itemize}

For a morphism $p\colon B_1\rightarrow B_2$ we have the base change functor $p^*\colon \mathrm{Symp}^{\mathrm{ex}}_{B_2,-1}\rightarrow \mathrm{Symp}^{\mathrm{ex}}_{B_1,-1}$ which by \cite[Theorem A]{ParkSymplectic} admits a right adjoint $p_*\colon \mathrm{Symp}^{\mathrm{ex}}_{B_1,-1}\rightarrow \mathrm{Symp}^{\mathrm{ex}}_{B_2,-1}$ given by the shifted symplectic pushforward.

    Consider the following diagram:
\[\begin{tikzcd}
	{Y'} \\
	& {Y \times _B B'} & {B'} \\
	& Y & B.
	\arrow["\eta", from=1-1, to=2-2]
	\arrow["{g'}", curve={height=-12pt}, from=1-1, to=2-3]
	\arrow["{p_Y}"', curve={height=12pt}, from=1-1, to=3-2]
	\arrow["{\tilde{g}}", from=2-2, to=2-3]
	\arrow["{\tilde{p}_B}"', from=2-2, to=3-2]
	\arrow["{p_B}"', from=2-3, to=3-3]
	\arrow["g", from=3-2, to=3-3]
\end{tikzcd}\]
We show that the desired exact $(-1)$-shifted Lagrangian correspondence is given by the Beck--Chevalley map
\begin{equation}\label{eq:unit Lag corr gtilde}
    \T^*[-1](Y'/B')
    \cong \tilde{g}_* \eta_*  (Y', 0) \dashleftarrow \tilde{g}_*  (Y\times_B B', 0)
    \cong \T^*[-1](Y\times_B B'/B')
    \cong \T^*[-1](Y/B) \times_{B} B'.
\end{equation}
in $\mathrm{Symp}^{\mathrm{ex}}_{B', -1}$, where the middle dashed arrow is induced by the unit map $\eta_* (Y', 0) \dashleftarrow Y\times_B B'$ in $\mathrm{Symp}^{\mathrm{ex}}_{Y\times_B B', -1}$.
Indeed, the latter is given by the exact Lagrangian correspondence
\begin{equation}\label{eq-unit-Lag-corresp}
\T^*[-1](Y' / Y \times_B B') \xleftarrow{0} Y' \xrightarrow{\eta} Y \times_B B'.
\end{equation}
Since $\T^*(Y \times_B B' / B') \cong  \T^*(Y/B) \times_{Y} (Y \times_{B} B')$, it follows from \cite[Proposition 2.3.1]{ParkSymplectic} that $\tilde{g}_*(Y')$ as a morphism in $\mathrm{Symp}^{\mathrm{ex}}_{Y\times_B B', -1}$ is computed as the zero locus of the map $\mu \colon Y' \xrightarrow{p_Y} Y \xrightarrow{0} \T^*(Y/B)$, whence a canonical identification $\tilde{g}_*(Y') \cong  \T^*[-1](Y/B) \times_{Y} Y'$.
Thus we find that the morphism \eqref{eq:unit Lag corr gtilde} is given by the desired correspondence \eqref{eq:Lagrangian corr}.
Since \eqref{eq-unit-Lag-corresp} is naturally oriented, we conclude that \eqref{eq:Lagrangian corr} is naturally oriented as well.
\end{proof}


In certain cases, we can use \eqref{eq:Lagrangian corr} to relate the perverse pullbacks along $\pi \colon \T^*[-1](Y/B) \to B$ and $\pi' \colon \T^*[-1](Y'/B') \to B'$.

\begin{lemma}\label{lem:Lagrangian funct}
    With notation as above, we have:
    \begin{enumerate}
        \item\label{item:Lagrangian funct/1}
        If the morphism $\eta \colon Y' \to Y\times_B B'$ is an isomorphism on classical truncations and is lfp with cotangent complex of Tor-amplitude $\le -2$, then in \eqref{eq:Lagrangian corr} the morphism $\partial p$ is smooth, $p_\tau$ is an isomorphism on classical truncations.
        Moreover, if we set $n \coloneqq \dim(Y'/Y\times_B B')$, there is a canonical isomorphism
        \[ (\tau \times_{B} \id_{B'})_* (\pi \times_{B} \id_{B'})^\varphi \simeq \eta_*  \tau'_* \pi'^\varphi [n]. \]

        \item\label{item:Lagrangian funct/2}
        If the morphism $\eta \colon Y' \to Y\times_B B'$ is smooth, then in \eqref{eq:Lagrangian corr} the morphism $\partial p$ is an isomorphism on classical truncations and $p_\tau$ is smooth.
        Moreover, there is a canonical isomorphism
        \[ \tau'_* \pi'^\varphi \cong \eta^\dag (\tau \times_{B} \id_{B'})_* (\pi \times_{B} \id_{B'})^\varphi. \]
    \end{enumerate}
\end{lemma}
\begin{proof}
    (\ref{item:Lagrangian funct/1})
    Note that there is a Cartesian square
    \[\begin{tikzcd}
        \T^*[-1](Y/B) \times_{Y} Y' \ar{r}{p_\tau}\ar{d}{\tau \times_{Y}  \id_{Y'}}
        & \T^*[-1](Y/B) \times_{B} B' \ar{d}{\tau \times_{B}  \id_{B'}}
        \\
        Y' \ar{r}{\eta}
        & Y \times_B B'.
    \end{tikzcd}\]
    Since $\eta$ is an isomorphism on classical truncations with the contangent complex having Tor-amplitude $\le -2$, so is $p_\tau$. 
    We have a Cartesian square
    \[
    \xymatrix{
    p^*_Y\bL_{Y/B}[-1] \ar[r] \ar[d] & \bL_{Y'/B'}[-1] \ar[d] \\
    0 \ar[r] & \bL_{Y'/Y\times_B B'}[-1]
    }
    \]
    of perfect complexes on $Y'$. Taking total spaces, we obtain a Cartesian square
    \begin{equation}\label{eq:cotangentbasechangeCartesian}
    \xymatrix{
    \T^*[-1](Y/B) \times_{Y} Y' \ar^{\partial p}[r] \ar[d] & \T^*[-1](Y'/B') \ar[d] \\
    Y' \ar[r] & \T^*[-1](Y'/Y\times_B B').
    }
    \end{equation}
    Since $\bL_{Y'/Y\times_B B'}[-1]$ is of Tor-amplitude $\le -1$, the zero section $Y'\rightarrow \T^*[-1](Y'/Y\times_B B')$ is smooth and hence so is $\partial p$.
    We may thus apply \cite[Proposition 6.9]{KKPS1} and deduce the natural isomorphisms
    \begin{align*}
       (\tau \times_{B} \id_{B'})_* (\pi \times_{B} \id_{B'})^\varphi \xrightarrow{\sim} (\tau \times_{B} \id_{B'})_* p_{\tau,*} (\partial p)^\dag \pi'^\varphi \xrightarrow{\sim} \eta_* \tau'_*(\partial p)_*  (\partial p)^\dag \pi'^\varphi \xrightarrow{\sim} \eta_* \tau'_* \pi'^\varphi [n], 
    \end{align*}
    where we used purity and homotopy invariance in the last step (note that $\partial p$ is a torsor under a vector bundle stack of rank $n$, so this follows from \cite[Proposition A.10]{KhanVirtual}).

    (\ref{item:Lagrangian funct/2})
    The assumption that $\eta\colon Y' \to Y\times_B B'$ is smooth implies that its base change $p_\tau$ is also smooth, and that the zero section $Y'\rightarrow \T^*[-1](Y'/Y\times_B B')$ is an isomorphism on classical truncations. Therefore, using the Cartesian square \eqref{eq:cotangentbasechangeCartesian} we get that $\partial p$ is an isomorphism on classical truncations. We may thus apply \cite[Proposition 6.9]{KKPS1} and deduce natural isomorphisms
    \[
    \tau'_* \pi'^\varphi \xrightarrow[\sim]{}  \tau'_* (\partial p)_* p_\tau^\dag (\pi \times_{B} \id_{B'})^\varphi  \xleftarrow[\sim]{\Ex^!_*} \eta^{\dagger} (\tau \times_{B} \id_{B'})_*(\pi \times_{B} \id_{B'})^\varphi.
    \]
\end{proof}

\subsection{Dimensional reduction for closed immersions}

We prove \cref{thm:dimensionalreduction} when the morphism $g \colon Y \to B$ is a closed immersion.
We begin with a direct proof of the following consequence of \cref{thm:dimensionalreduction}.

\begin{lemma}\label{lem-pre-dim-red-cl-imm}
    Assume that $g \colon Y \to B$ is a closed immersion.
    Then the maps
    \[
     \tau_! \pi^{\varphi}  \xrightarrow{\unit} \tau_! \pi^{\varphi} g_* g^*,
     \qquad \tau_* \pi^{\varphi}  g_! g^! \xrightarrow{\counit} \tau_* \pi^{\varphi}
    \]
    are invertible.
\end{lemma}

\begin{proof}
We consider the second map; the proof for the first is similar.
Thanks to \cref{thm:perversepullback}(2), the statement is smooth-local on $B$.
Therefore we may assume that $B$ is a scheme of finite type, and that $g \colon Y \to B$ factors as 
\[
g \colon Y \xrightarrow{i} Y' \xrightarrow{g'} B
\]
where $i$ is an isomorphism on classical truncations and $\mathbb{L}_{Y / Y'}$ is of Tor-amplitude $\leq -2$, and $g'$ is the inclusion of the zero locus of a section $s \in \Gamma(B, E)$ of a vector bundle $E$ on $B$.
(This follows from the proof of \cite[Proposition 3.2.18]{LurieHA}).
We let $\tau' \colon \mathrm{T}^*[-1](Y' / B) \to Y'$ be the projection and set $\pi' \coloneqq g' \circ \tau'$.
Note that \cref{cor:dim red for crit} yields $\tau'_* \pi'^{\varphi} \cong g'^![-\dim(Y' / B)]$.
By \cref{lem:Lagrangian funct}(\ref{item:Lagrangian funct/1}) we have $\tau'_* \pi'^\varphi \cong i_* \tau_* \pi^\varphi [\dim(Y/Y')]$, or equivalently $\tau_* \pi^\varphi \simeq i^! \tau'_* \pi'^\varphi [-\dim(Y/Y')]$ (since $i_*$ is an equivalence with inverse $i^* \simeq i^!$).
Combining this with the isomorphism $\tau'_* \pi'^{\varphi} \cong g'^![-\dim(Y' / B)]$ yields
\[
    \tau_* \pi^\varphi \simeq i^! g'^![-\dim(Y / B)] \simeq g^![-\dim(Y/B)].
\]
Since $g^!$ inverts the counit map $g_! g^! \to \id$ (as $g$ is a closed immersion), the claim now follows.
\end{proof}

The above lemma implies the dimensional reduction theorem for closed immersions:

\begin{corollary}\label{cor-dim-red-cl-imm}
    Assume that $g \colon Y \to B$ is a closed immersion and the coefficient ring $R$ is a field.
    Then we have natural isomorphisms
    \begin{equation}\label{eq-dim-red-cl-imm}
    \tau_! \pi^\varphi\cong g^*[\dim(Y/B)],\qquad \tau_* \pi^\varphi\cong g^![-\dim(Y/B)].
    \end{equation}
\end{corollary}

\begin{proof}
From \eqref{eq-symp-push-perverse-pullback} we obtain a canonical isomorphism $\varepsilon_g\colon \pi^\phi g_! \cong \pi^\phi g_* \xrightarrow{\sim} \tau^\dag$. By \cref{lem-pre-dim-red-cl-imm} we thus have
    \begin{equation*}
        \tau_* \pi^\varphi
        \xleftarrow[\sim]{\counit} \tau_* \pi^\varphi g_! g^!
        \xrightarrow[\sim]{\varepsilon_g} \tau_* \tau^\dag g^!
        \xrightarrow[\sim]{} g^![-\dim(Y/B)],
    \end{equation*}
    where $\tau_* \tau^\dag \cong \tau_* \tau^* [r] \cong (-)[r]$ by homotopy invariance since $\tau$ is a vector bundle of rank $r=-\dim(Y/B)$.
    Similarly, $\tau_! \pi^\varphi \cong \tau_! \tau^\dag g^* \cong g^*[\dim(Y/B)]$.
\end{proof}

\subsection{Dimensional reduction for the Fourier--Sato transform}\label{sect:FourierSato}

We now turn our attention to another special case of the dimensional reduction theorem.
Given a perfect complex $E \in \Perf(B)$, we denote the total space $\Tot_B(E)$ also by $E$.
If we take the morphism $g$ in \cref{thm:dimensionalreduction} to be the zero section $0_E \colon B \to E$, observe that there is a canonical isomorphism $\T^*[-1](B/E) \cong E^{\vee}$ and the morphism $\pi \colon \T^*[-1](B/E) \to E$ is identified with the composite
\begin{equation}\label{eq:uEveeE}
    u_{E^\vee/E} \colon E^{\vee} \xrightarrow{p_{E^\vee}} B \xrightarrow{0_E} E.
\end{equation}
In particular, we obtain a canonical relative exact $(-1)$-shifted symplectic structure and orientation on $u_{E^\vee/E}$, whence a perverse pullback functor
\[
    u_{E^\vee/E}^{\varphi} \colon \Dbc(E) \to \Dbc(E^{\vee}).
\]

In this situation, the dimensional reduction theorem (\cref{thm:dimensionalreduction}) implies in particular:

\begin{proposition}\label{prop:dim red Fourier-Sato}
    Let $B$ be an lft higher Artin stack. Then,
    for every perfect complex $E \in \Perf(B)$ of rank $r$, there is a canonical isomorphism
    \begin{equation}\label{eq-dimred-FS}
        \red_E \colon p_{E^{\vee},*} u_{E^\vee/E}^\varphi p_{E}^!
        \cong (-)[r].
    \end{equation}
\end{proposition}

We will give a direct proof of \cref{prop:dim red Fourier-Sato} on our route towards the proof of \cref{thm:dimensionalreduction}.

\begin{remark}\label{rem:Fourier-Sato}
    The functor $u_{E^\vee/E}^\varphi$ preserves monodromic sheaves and coincides with (the restriction to perverse sheaves of) the Fourier–Sato transform $\Shv(E;R) \to \Shv(E^\vee;R)$ defined in \cite[\S2]{KhanKinjo}. When $E$ is a vector bundle, this is established in \cref{prop:FSperversepullback}, and the general case reduces to this situation. As this fact is not needed in the present work, we do not include the detail here.
\end{remark}

In view of \cref{rem:Fourier-Sato}, we abbreviate
\[ \mathfrak{F}_E \coloneqq u_{E^\vee/E}^{\varphi} \colon \Dbc(E) \to \Dbc(E^\vee). \]
This functor enjoys the following properties:

\begin{lemma}
    Let $g \colon B' \to B$ be a morphism between lft higher Artin stacks and $E$ be a perfect complex over $B$. Denote by $g_E \colon E |_ {B'} \to E$ and $g_{E^{\vee}} \colon E^{\vee} |_ {B'} \to E^{\vee}$ the base change of $g$.
    \begin{enumerate}
        \item If $g$ is a smooth morphism, there is a natural isomorphism
        \begin{equation}\label{eq-FS-base-pullback}
             \mathfrak{F}_{E |_ {B'}} g_{E}^{\dagger} \cong g_{E^{\vee}}^{\dagger}  \mathfrak{F}_{E}.
        \end{equation}
         \item If $g$ is a finite morphism, there is a natural isomorphism
        \begin{equation}\label{eq-FS-base-pushforward}
            \mathfrak{F}_{E} g_{E, *} \cong g_{E^{\vee}, *} \mathfrak{F}_{E |_ {B'}}.
        \end{equation}
        \item Assume that we are further given a morphism $h \colon \tilde{B} \to B$ and form the following Cartesian square:
        \[\begin{tikzcd}
        	{\tilde{B'}} & {\tilde{B}} \\
        	{B'} & {B.}
        	\arrow["{\tilde{g}}", from=1-1, to=1-2]
        	\arrow["{h'}", from=1-1, to=2-1]
        	\arrow["h"', from=1-2, to=2-2]
        	\arrow["g", from=2-1, to=2-2]
        \end{tikzcd}\]       
        Assume further that $g$ is finite and $h$ is smooth.
        We define $h_E$, $h'_E$, $\tilde{g}_E$, $h_{E^{\vee}}$, $h'_{E^{\vee}}$, $\tilde{g}_{E^{\vee}}$ in similar manners to $g_E$ and $g_{E^{\vee}}$. 
        Then the following diagram commutes:
        \begin{equation}\label{eq-FS-base-base-smooth-finite}
        \begin{tikzcd}
        	{ \mathfrak{F}_{E |_{\tilde{B}}}  \tilde{g}_{E,*}   h_{E}'^{\dagger}} & { \mathfrak{F}_{E |_{\tilde{B}}}    h_{E}^{\dagger}  {g}_{E,*}} & {h_{E^{\vee}}^{\dagger} \mathfrak{F}_{E } g_{E, *} } \\
        	{ \tilde{g}_{E^{\vee},*}  \mathfrak{F}_{E |_{\tilde{B}'}}  h_{E}'^{\dagger}} & { \tilde{g}_{E^{\vee},*} h_{E^{\vee}}'^{\dagger} \mathfrak{F}_{E |_{B'}} } & {h_{E^{\vee}}^{\dagger} g_{E^{\vee},*} \mathfrak{F}_{E |_{B'}}.}
        	\arrow["\textnormal{{(\ref{eq-FS-base-pushforward})}}"', from=1-1, to=2-1]
        	\arrow["{{\Ex^!_*}}"', "\sim", from=1-2, to=1-1]
        	\arrow["\textnormal{{(\ref{eq-FS-base-pullback})}}", from=1-2, to=1-3]
        	\arrow["\textnormal{{(\ref{eq-FS-base-pushforward})}}", from=1-3, to=2-3]
        	\arrow["\textnormal{{(\ref{eq-FS-base-pullback})}}"', from=2-1, to=2-2]
        	\arrow["{{\Ex^!_*}}", "\sim"', from=2-3, to=2-2]
        \end{tikzcd}
        \end{equation}
    \end{enumerate}
\end{lemma}

\begin{proof}
    The first two statements are straightforward consequences of \eqref{eq-Lagrangian functoriality of perverse pullback} and \eqref{eq-perverse-pullback-symplectic-finite-base-change-compatibility}. The third statement follows from the compatibility relations between $\alpha$ and $\beta$ in \cite[Theorem 5.23]{KKPS1}.
\end{proof}

\begin{lemma}\label{lem-fs-fiber-functorial}
    Let $B$ be an lft higher Artin stack and $f \colon E' \to E$ a morphism of perfect complexes over $B$.
    We let $f^{\vee} \colon E^{\vee} \to E'^{\vee}$ denote the dual of $f$.
    \begin{enumerate}
        \item If $f$ is smooth (equivalently, its fiber is of Tor-amplitude $\le 0$), then $f^{\vee}$ is a closed immersion and there is a natural isomorphism
        \begin{equation}\label{eq-FS-fiber-pullback}
        \mathfrak{F}_{E'} f^{\dagger} \cong f^{\vee}_* \mathfrak{F}_{E}.
        \end{equation}
        which is functorial for compositions in $f$.
    
        \item If $f$ is a closed immersion (equivalently, its fiber is of Tor-amplitude $> 0$), then $f^{\vee}$ is smooth and there is a natural isomorphism
        \begin{equation}\label{eq-FS-fiber-pushforward}
        \mathfrak{F}_{E} f_* \cong  f^{\vee,\dagger} \mathfrak{F}_{E'} 
        \end{equation}
        which is functorial for compositions in $f$.

        \item Assume that we are given a morphism $e \colon \tilde{E} \to E$ and form the following Cartesian square:
        \[\begin{tikzcd}
        	{\tilde{E}'} & {\tilde{E}} \\
        	{E'} & {E.}
        	\arrow["{\tilde{f}}", from=1-1, to=1-2]
        	\arrow["{e'}", from=1-1, to=2-1]
        	\arrow["{e}", from=1-2, to=2-2]
        	\arrow["f", from=2-1, to=2-2]
        \end{tikzcd}\]
        Assume further that $f$ is smooth and $e_2$ is a closed immersion. Then the following diagram commutes:
        \begin{equation}\label{eq-FS-fibre-fibre-smooth-finite}
        \begin{tikzcd}
        	{\mathfrak{F}_{E'}  e_{ *} \tilde{f}^{\dagger}} & {\mathfrak{F}_{E'} f^{\dagger} e_{*}} & {   f^{\vee}_*  \mathfrak{F}_{E} e_{*}} \\
        	{ (e'^{\vee})^{\dagger} \mathfrak{F}_{\tilde{E}'}  \tilde{f}^{\dagger}} & { (e'^{\vee})^{\dagger}   \tilde{f}^{\vee}_* \mathfrak{F}_{\tilde{E}}} & {   f^{\vee}_*  (e^{\vee})^{\dagger} \mathfrak{F}_{\tilde{E}}.}
        	\arrow["{{(\ref{eq-FS-fiber-pushforward})}}"', from=1-1, to=2-1]
        	\arrow["{\sim}", "\Ex^!_*"', from=1-2, to=1-1]
        	\arrow["{{(\ref{eq-FS-fiber-pullback})}}", from=1-2, to=1-3]
        	\arrow["{{(\ref{eq-FS-fiber-pushforward})}}", from=1-3, to=2-3]
        	\arrow["{{(\ref{eq-FS-fiber-pullback})}}"', from=2-1, to=2-2]
        	\arrow["{\sim}", "\Ex^!_*"', from=2-2, to=2-3]
        \end{tikzcd}
        \end{equation}

        \item Assume that we are given a morphism $g \colon B' \to B$ and form the following Cartesian square:
        \begin{equation}
        \begin{tikzcd}
        	{E' |_{B'}} & {E |_{B'}} \\
        	{E'} & {E.}
        	\arrow["{f'}", from=1-1, to=1-2]
        	\arrow["{g_{E'}}"', from=1-1, to=2-1]
        	\arrow["{g_{E}}", from=1-2, to=2-2]
        	\arrow["f", from=2-1, to=2-2]
        \end{tikzcd}
        \end{equation}
                If $f$ and $g$ are smooth, then the following diagram commutes:
        \begin{equation}\label{eq-FS-base-fibre-smooth-smooth}
        \begin{tikzcd}
        	{ \mathfrak{F}_{E' |_{B'}}  f'^{\dagger}  g_{E}^{\dagger}} & { \mathfrak{F}_{E' |_{B'}} g_{E'}^{\dagger} f^{\dagger} } & {g_{E'^{\vee}}^{\dagger} \mathfrak{F}_{E'} f^{\dagger}} \\
        	{  f'^{\vee}_*  \mathfrak{F}_{E|_{B'}} g_{E}^{\dagger}} & {  f'^{\vee}_* g_{E^{\vee}}^{\dagger} \mathfrak{F}_{E}} & {g_{E'^{\vee}}^{\dagger}  f^{\vee}_* \mathfrak{F}_{E}.}
        	\arrow["\sim", from=1-1, to=1-2]
        	\arrow["{{(\ref{eq-FS-fiber-pullback})}}"', from=1-1, to=2-1]
        	\arrow["{{(\ref{eq-FS-base-pullback})}}", from=1-2, to=1-3]
        	\arrow["{{(\ref{eq-FS-fiber-pullback})}}", from=1-3, to=2-3]
        	\arrow["{{(\ref{eq-FS-base-pullback})}}"', from=2-1, to=2-2]
        	\arrow["{{\Ex^!_*}}", "\sim"', from=2-3, to=2-2]
        \end{tikzcd}
        \end{equation}
        If $f$ is smooth and $g$ is finite, then the following diagram commutes:
        \begin{equation}\label{eq-FS-base-fibre-finite-smooth}
        \begin{tikzcd}
        	{\mathfrak{F}_{E'} f^\dagger g_{E, *}} & {\mathfrak{F}_{E'}  g_{E', *} f'^\dagger} & { g_{E'^{\vee}, *} \mathfrak{F}_{E' |_{B'}}  f'^\dagger} \\
        	{f^{\vee}_* \mathfrak{F}_{E}  g_{E, *}} & {f^{\vee}_* g_{E^{\vee}, *} \mathfrak{F}_{E |_{B'}}  } & {g_{E'^{\vee}, *} f'^{\vee}_*  \mathfrak{F}_{E |_{B'}}.  }
        	\arrow["{\Ex^!_*}", from=1-1, to=1-2]
        	\arrow["{(\ref{eq-FS-fiber-pullback})}"', from=1-1, to=2-1]
        	\arrow["{(\ref{eq-FS-base-pushforward})}", from=1-2, to=1-3]
        	\arrow["{(\ref{eq-FS-fiber-pullback})}", from=1-3, to=2-3]
        	\arrow["{(\ref{eq-FS-base-pushforward})}"', from=2-1, to=2-2]
        	\arrow["\sim"', from=2-2, to=2-3]
        \end{tikzcd}
        \end{equation}
        If $f$ is a closed immersion and $g$ is smooth, then the following diagram commutes:
        \begin{equation}\label{eq-base-fibre-smooth-finite}
        \begin{tikzcd}
        	{ \mathfrak{F}_{E|_{B'}} f'_* g_{E'}^{\dagger}} & { \mathfrak{F}_{E|_{B'}} g_{E}^{\dagger} f_*} & {g_{E^{\vee}}^{\dagger} \mathfrak{F}_{E} f_*} \\
        	{ (f'^{\vee})^{\dagger} \mathfrak{F}_{E'|_{B'}}  g_{E'}^{\dagger} } & {(f'^{\vee})^{\dagger} g_{E'^{\vee}}^{\dagger} \mathfrak{F}_{E'} } & {g_{E^{\vee}}^{\dagger}(f^{\vee})^{\dagger} \mathfrak{F}_{E'}.}
        	\arrow["{{{(\ref{eq-FS-fiber-pushforward})}}}"', from=1-1, to=2-1]
        	\arrow["\sim", "\Ex^!_*"', from=1-2, to=1-1]
        	\arrow["{{{(\ref{eq-FS-base-pullback})}}}", from=1-2, to=1-3]
        	\arrow["{{{(\ref{eq-FS-fiber-pushforward})}}}", from=1-3, to=2-3]
        	\arrow["{{{(\ref{eq-FS-base-pullback})}}}"', from=2-1, to=2-2]
        	\arrow["\sim"', from=2-2, to=2-3]
        \end{tikzcd}
        \end{equation}
        If $f$ is a closed immersion and $g$ is finite, then the following diagrams commutes:
        \begin{equation}\label{eq-FS-base-fibre-finite-finite}
        \begin{tikzcd}
        	{\mathfrak{F}_{E} f_* g_{E', *}} & {\mathfrak{F}_{E}  g_{E, *} f'_*} & { g_{E^{\vee}, *} \mathfrak{F}_{E|_{B'}}  f'_*} \\
        	{(f^{\vee})^{\dagger} \mathfrak{F}_{E'} g_{E', *}} & {(f^{\vee})^{\dagger} g_{E'^{\vee}, *} \mathfrak{F}_{E'|_{B'}} } & { g_{E^{\vee}, *} (f'^{\vee})^{\dagger} \mathfrak{F}_{E'|_{B'}}.}
        	\arrow["\sim", from=1-1, to=1-2]
        	\arrow["{{(\ref{eq-FS-fiber-pushforward})}}"', from=1-1, to=2-1]
        	\arrow["{{(\ref{eq-FS-base-pushforward})}}", from=1-2, to=1-3]
        	\arrow["{{(\ref{eq-FS-fiber-pushforward})}}", from=1-3, to=2-3]
        	\arrow["{{(\ref{eq-FS-base-pushforward})}}"', from=2-1, to=2-2]
        	\arrow["{{\Ex^!_*}}"', from=2-2, to=2-3]
        \end{tikzcd}
        \end{equation} 
    \end{enumerate}
\end{lemma}

\begin{proof}
    In view of the identifications $\T^*(B/E')[-1] \cong E'^\vee$ and $\T^*(B/E)[-1] \cong E^\vee$, the functoriality of the symplectic pushforward implies:
    \begin{enumerate}
        \item[$(\ast)$]
        The morphism $u_{E^\vee/E} \colon E^\vee \to E$, together with its relative $(-1)$-shifted symplectic structure, is identified with the symplectic pushforward of $u_{E'^\vee/E'} \colon E'^\vee \to E'$ along $f \colon E' \to E$.
    \end{enumerate}
    In other words, we have a symplectic pushforward square
    \[\begin{tikzcd}
        {E^{\vee}} \ar{r}{f^\vee}\ar{d}{u_{E^\vee/E}} & {E'^{\vee}} \ar{d}{u_{E'^\vee/E'}} \\
        {E}  & {E'.} \ar{l}{f}
    \end{tikzcd}\]
    The first and second claims then follow by applying \eqref{eq-perverse-pullback-symplectic-finite-base-change-compatibility} and \eqref{eq-symp-push-perverse-pullback}. The third claim is a special case of \cite[Proposition 6.10(4)]{KKPS1}. The fourth claim follows form the compatibility relations of $\alpha$, $\beta$ and $\gamma$ in \cite[Theorem 5.23, Proposition 6.10(2,3)]{KKPS1} respectively.
\end{proof}

\begin{proposition}\label{prop-FS-product}
    Let $E_i \to B_i$ be perfect complexes over lft higher Artin stacks for $i = 1, 2$.
    Then there is a natural isomorphism of functors
    \begin{equation}\label{eq-FS-TS}
    \TS\colon \mathfrak{F}_{E_1}(-) \boxtimes \mathfrak{F}_{E_2}(-) \cong \mathfrak{F}_{E_1 \boxplus E_2}(- \boxtimes -).
    \end{equation}
    Further, it is compatible with the isomorphisms \eqref{eq-FS-base-pullback}, \eqref{eq-FS-base-pushforward}, \eqref{eq-FS-fiber-pullback} and \eqref{eq-FS-fiber-pushforward} in a natural manner.
\end{proposition}

\begin{proof}
    The isomorphism \eqref{eq-FS-TS} is a special case of the isomorphism \eqref{eq:TS}. The compatibility relations are proved in \cite[Theorem 5.23(5) and Proposition 6.10(6)]{KKPS1}.
\end{proof}

Now consider the case when $E$ is a vector bundle (i.e., a  perfect complex of Tor-amplitude $[0, 0]$) over an lft higher Artin stack.
Let $q_E \colon E \times_B E^{\vee} \to \bA^1$ be the pairing function.
It follows from \cref{lem-crit-locus-vs-shifted-cotangent} that there exists a canonical equivalence
\[
E^{\vee} \cong \R\Crit_{E \times_B E^{\vee} / E} (q_E)
\]
preserving the exact $(-1)$-shifted symplectic structure over $E$ and the canonical orientation.
In particular, there exists a natural equivalence
\begin{equation}\label{eq-FS-vect-bundle}
(E^{\vee} \hookrightarrow E \times_B E^{\vee})_* \mathfrak{F}_E \cong   \varphi_{q_E} (E \times_B E^{\vee} \to E)^\dagger.
\end{equation}

In \cite[Section 5.1.1]{KKPS1} we have defined the stabilization isomorphism
\[\stab^E_{E\oplus E^\vee}\colon 0_{E\oplus E^\vee, *}\xrightarrow{\sim} \varphi_{q_E} p_{E\oplus E^\vee}^\dag.\]

By construction,
the following diagrams commute:
\begin{equation}\label{eq-FS-bundle-push}
\begin{tikzcd}
	{\mathfrak{F}_E 0_{E, *}} & {p_{E^{\vee}}^{\dagger}} \\
	{(E^{\vee} \hookrightarrow E \times_B E^{\vee})^* \varphi_{q_E} (E \times_B E^{\vee} \to E)^{\dagger} 0_{E, *}} & {(E^{\vee} \hookrightarrow E \times_B E^{\vee})^* (E \times_B E^{\vee} \to E)^{\dagger} 0_{E, *},}
	\arrow["{(\ref{eq-FS-fiber-pushforward})}", from=1-1, to=1-2]
	\arrow["\cong"', "(\ref{eq-FS-vect-bundle})", from=1-1, to=2-1]
	\arrow["\cong", "(\ref{eq-FS-vect-bundle})"', from=1-2, to=2-2]
	\arrow["\cong"', from=2-1, to=2-2]
\end{tikzcd}
\end{equation}
\begin{equation}\label{eq-FS-bundle-pull}
\begin{tikzcd}
	{\mathfrak{F}_E p_E^{\dagger}} &[40pt] {0_{E^{\vee}, *}} \\
	{(E \hookrightarrow E \oplus E^{\vee})^! \varphi_{q_E} p_{E \oplus E^{\vee}}^{\dagger}} & {0_{E^{\vee}, *}.}
	\arrow["{(\ref{eq-FS-fiber-pullback})}", "\cong"', from=1-1, to=1-2]
	\arrow[from=1-1, to=2-1, "\cong"']
	\arrow[equals, from=1-2, to=2-2]
	\arrow["{\stab^{E}_{E \oplus E^{\vee}}}", "\cong"', from=2-2, to=2-1]
\end{tikzcd}
\end{equation}



\begin{proof}[Proof of \cref{prop:dim red Fourier-Sato}]

    We first prove the result for special cases and then deduce the general case from these.

    \textit{Case 1: $E$ is of Tor-amplitude $\leq 0$.}
    
    In this case, the zero section $0_E \colon B \hookrightarrow E$ is a closed immersion, so the claim follows from the isomorphism
    \begin{equation}\label{eq-dim-red-climm-tor>0}
      p_{E^{\vee},*} \mathfrak{F}_E \simeq 0_{E}^![r]
    \end{equation}
    provided by \cref{cor-dim-red-cl-imm}.

    \textit{Case 2: $B$ admits the resolution property, e.g. $B$ is a quasi-projective scheme.}
    
    We next consider the case when $B$ admits the resolution property.
    We choose a global resolution for $E$, i.e., a morphism of perfect complexes $f \colon F \to E$ such that
    \begin{itemize}
        \item $f$ is a smooth morphism.
        \item The projection $p_F \colon F \to B$ is representable.
    \end{itemize}
    We define an isomorphism
    \[
         \red^{F}_E \colon (p_{E^{\vee}})_*\mathfrak{F}_{E} p_{E}^! \simeq (-)[r]
    \]
    as the composite
    \[
    (p_{E^{\vee}})_*\mathfrak{F}_{E} p_{E}^! \simeq 
    (p_{F^{\vee}})_* f^{\vee}_* \mathfrak{F}_{E} p_{E}^! \xleftarrow[\simeq]{\eqref{eq-FS-fiber-pullback}}
    (p_{F^{\vee}})_*  \mathfrak{F}_{F} f^! p_{E}^! [- \dim f] \xrightarrow[\simeq]{\eqref{eq-dim-red-climm-tor>0}} (-)[r].
    \]

    We claim that $\red^{F}_E$ does not depend on the choice of the resolution $f \colon F \to E$,
    as long as the relative dimension of $f$ is even.
    Assuming this, we define the dimensional reduction isomorphism by
    \[
    \red_E \coloneqq \red^{F}_E
    \]
    for arbitrarily chosen resolution $F \to E$ with even relative dimension.
    To prove the independence of the resolution, take another resolution $f' \colon F' \to E$ such that $\dim f'$ is even.
    We let $F''$ be a resolution of $\fib(F \oplus F' \to E)$,
    such that the relative dimension over $E$ is even.
    To prove $\red^{F}_E \simeq \red^{F'}_E$, it is enough to prove 
    $\red^{F}_E \simeq \red^{F''}_E$ and
    $\red^{F'}_E \simeq \red^{F''}_E$.
    In particular, we may assume that there is a smooth morphism of resolutions $g \colon (F' \to E) \to  (F \to E)$.
    Now consider the following diagram:
\[\begin{tikzcd}
	& {(p_{F^{\vee}})_* (f^{\vee})_* \mathfrak{F}_E p_{E}^!} & {(p_{F^{\vee}})_* \mathfrak{F}_F   p_{F}^![- \dim f]} & {\id_X[-r]} \\
	{(p_{E^{\vee}})_* \mathfrak{F}_E p_{E}^!} & {(p_{F'^{\vee}})_* (g^{\vee})_* (f^{\vee})_* \mathfrak{F}_E p_{E}^!} & {(p_{F'^{\vee}})_* (g^{\vee})_* \mathfrak{F}_F   p_{F}^! [- \dim f]} \\
	& {(p_{F'^{\vee}})_* (f'^{\vee})_* \mathfrak{F}_E p_{E}^!} & {(p_{F'})_* \mathfrak{F}_{F'}   p_{F'}^! [- \dim f']} & {\id_X[-r].}
	\arrow["\cong"', from=1-2, to=2-2]
	\arrow["{{\eqref{eq-FS-fiber-pullback}}}"', from=1-3, to=1-2]
	\arrow[""{name=0, anchor=center, inner sep=0}, "\cong", from=1-3, to=1-2]
	\arrow["{{\mathrm{red}_F}}", from=1-3, to=1-4]
	\arrow[""{name=1, anchor=center, inner sep=0}, "\cong"', from=1-3, to=1-4]
	\arrow["\cong"', from=1-3, to=2-3]
	\arrow[equals, from=1-4, to=3-4]
	\arrow[""{name=2, anchor=center, inner sep=0}, "\cong", from=2-1, to=1-2]
	\arrow["{{\mathrm{red}^F_E}}", curve={height=-80pt}, from=2-1, to=1-4]
	\arrow[""{name=3, anchor=center, inner sep=0}, "\cong"', from=2-1, to=3-2]
	\arrow["{{\mathrm{red}_E^{F'}}}"', curve={height=80pt}, from=2-1, to=3-4]
	\arrow["\cong"', from=2-2, to=3-2]
	\arrow["{{\eqref{eq-FS-fiber-pullback}}}"', from=2-3, to=2-2]
	\arrow[""{name=4, anchor=center, inner sep=0}, "\cong", from=2-3, to=2-2]
	\arrow["{{\eqref{eq-FS-fiber-pullback}}}"', "\cong", from=3-3, to=2-3]
	\arrow["{{\eqref{eq-FS-fiber-pullback}}}"', from=3-3, to=3-2]
	\arrow[""{name=5, anchor=center, inner sep=0}, "\cong", from=3-3, to=3-2]
	\arrow["\cong"', from=3-3, to=3-4]
	\arrow[""{name=6, anchor=center, inner sep=0}, "{{\mathrm{red}_{F'}}}", from=3-3, to=3-4]
	\arrow["{{(B)}}"{description}, draw=none, from=0, to=4]
	\arrow["{{(D)}}"{description}, shift left=3, draw=none, from=1, to=6]
	\arrow["{{(A)}}"{description}, shift right=5, draw=none, from=2, to=3]
	\arrow["{{(C)}}"{description}, draw=none, from=4, to=5]
\end{tikzcd}\]
    It is enough to show that the outer square commutes.
    The commutativity of the diagrams (A) and (B) are obvious.
    The commutativity of the diagram (C) follows from the associativity of \eqref{eq-FS-fiber-pullback} proved in \cref{lem-fs-fiber-functorial}.
    Therefore we are reduced to proving the commutativity of the diagram (D).
    
    Set $K \coloneqq \fib(F' \xrightarrow[]{g} F)$.
    Since $g$ is a smooth representable morphism, $K$ is a vector bundle on $X$.
    Further, $\rank (K)$ is even by assumption.
    We let $\iota \colon K \hookrightarrow F'$ be the natural inclusion.
    To prove the commutativity of the diagram (D) above,
    it is enough to show that the commutativity of the outer square of the following diagram:
\[\begin{tikzcd}
	{(g^{\vee})_* \mathfrak{F}_F p_{F}^!} & {(g^{\vee})_* \mathfrak{F}_F (0_{F})_!} &&& {(g^{\vee})_* p_{F^{\vee}}^{\dagger}} \\
	{ \mathfrak{F}_{F'} g^{\dagger} p_{F}^!} & { \mathfrak{F}_{F'} g^{\dagger} (0_{F})_!} & { \mathfrak{F}_{F'} \iota_! p_K^{\dagger}} & {(\iota^{\vee})^{\dagger} \mathfrak{F}_{K}  p_K^{\dagger}} & {(\iota^{\vee})^{\dagger} (0_{K^{\vee}})_* } \\
	&& { \mathfrak{F}_{F'} \iota_! (0_K)_![-m]} & {(\iota^{\vee})^{\dagger} \mathfrak{F}_{K} (0_K)_![-m]} \\
	{ \mathfrak{F}_{F'}  p_{F'}^![-m]} & { \mathfrak{F}_{F'}  (0_{F'})_![-m]} &&& {p_{F'^{\vee}}^{\dagger}[-m].}
	\arrow[""{name=0, anchor=center, inner sep=0}, "{\mathrm{counit}}"', from=1-2, to=1-1]
	\arrow[""{name=1, anchor=center, inner sep=0}, "{\eqref{eq-FS-fiber-pushforward}}", from=1-2, to=1-5]
	\arrow["\cong"{marking, allow upside down}, shift right=3, draw=none, from=1-2, to=1-5]
	\arrow["{\eqref{eq-FS-fiber-pullback}}", from=2-1, to=1-1]
	\arrow["\cong"{description}, shift right=3, draw=none, from=2-1, to=1-1]
	\arrow["{\eqref{eq-FS-fiber-pullback}}", from=2-2, to=1-2]
	\arrow["\cong"', from=2-2, to=1-2]
	\arrow[""{name=2, anchor=center, inner sep=0}, "{\mathrm{counit}}"', from=2-2, to=2-1]
	\arrow[""{name=3, anchor=center, inner sep=0}, "\cong", "\Ex^*_!"', from=2-3, to=2-2]
	\arrow[""{name=4, anchor=center, inner sep=0}, "{\eqref{eq-FS-fiber-pushforward}}", from=2-3, to=2-4]
	\arrow["\cong"{description}, shift right=3, draw=none, from=2-3, to=2-4]
	\arrow[""{name=5, anchor=center, inner sep=0}, "{\eqref{eq-FS-fiber-pullback}}", from=2-4, to=2-5]
	\arrow["\cong"{description}, shift right=3, draw=none, from=2-4, to=2-5]
	\arrow["\cong"', "\Ex^!_*", from=2-5, to=1-5]
	\arrow["{\mathrm{counit}}", from=3-3, to=2-3]
	\arrow[""{name=6, anchor=center, inner sep=0}, "\cong"{description}, shift right=3, draw=none, from=3-3, to=3-4]
	\arrow["{\eqref{eq-FS-fiber-pushforward}}", from=3-3, to=3-4]
	\arrow["{\mathrm{counit}}", from=3-4, to=2-4]
	\arrow[""{name=7, anchor=center, inner sep=0}, "{\eqref{eq-FS-fiber-pushforward}}", from=3-4, to=4-5]
	\arrow["\cong"{description}, shift right=3, draw=none, from=3-4, to=4-5]
	\arrow["\cong", from=4-1, to=2-1]
	\arrow["{\mathrm{counit}}", from=4-2, to=2-2]
	\arrow[""{name=8, anchor=center, inner sep=0}, "\cong", from=4-2, to=3-3]
	\arrow[""{name=9, anchor=center, inner sep=0}, "{\mathrm{counit}}"', from=4-2, to=4-1]
	\arrow["\cong"', shift right=3, draw=none, from=4-2, to=4-5]
	\arrow["{\mathrm{unit}}"', from=4-5, to=2-5]
	\arrow[""{name=10, anchor=center, inner sep=0}, "{\eqref{eq-FS-fiber-pushforward}}"', from=4-5, to=4-2]
	\arrow["{(A)}"{description}, draw=none, from=0, to=2]
	\arrow["{(B)}"{description}, draw=none, from=1, to=4]
	\arrow["{(C)}"{description}, draw=none, from=2, to=9]
	\arrow["{(D)}"{description}, draw=none, from=3, to=8]
	\arrow[""{name=11, anchor=center, inner sep=0}, "{(E)}"{description}, draw=none, from=4, to=6]
	\arrow["{(F)}"{description}, draw=none, from=5, to=7]
	\arrow["{(G)}"{description, pos=0.4}, draw=none, from=10, to=11]
\end{tikzcd}\]
    Here, we set $m \coloneqq \dim(g) = \rank(K)$.
    The commutativity of the diagrams (A), (C), (D), (E) are obvious.
    The commutativity of the diagram (B) follows from \cite[Proposition 6.10(4)]{KKPS1} and the commutativity of the diagram (G) follows from the associativity of the map \eqref{eq-FS-fiber-pushforward} proved in \cref{lem-fs-fiber-functorial}.
    Therefore it is enough to prove the commutativity of the diagram (F).
    In particular, we are reduced to proving the commutativity of the following diagram:
    \[\begin{tikzcd}
	{\mathfrak{F}_K 0_{K, !}} & {p_{K^{\vee}}^{\dagger}} \\
	{\mathfrak{F}_K p_K^!} & {0_{K, *}[\rank (K)].}
	\arrow["\cong"', "\eqref{eq-FS-fiber-pullback}",  from=1-1, to=1-2]
	\arrow["\counit", from=1-1, to=2-1]
	\arrow["\unit", from=1-2, to=2-2]
	\arrow["\eqref{eq-FS-fiber-pullback}", "\cong"', from=2-1, to=2-2]
    \end{tikzcd}\]
    Using the identification \eqref{eq-FS-vect-bundle} and the commutativity of the diagrams \eqref{eq-FS-bundle-push} and \eqref{eq-FS-bundle-pull},
    it is enough to prove the following diagram commutes:
\[\begin{tikzcd}
	{0_{K \oplus K^{\vee}}^* \varphi_{q_K}(K \times _B K^{\vee} \to K)^{\dagger} 0_{K, !}} & {0_{K \oplus K^{\vee}}^* (K \times _B K^{\vee} \to K)^{\dagger} 0_{K, !}} & {\id[\rank K]} \\
	{0_{K \oplus K^{\vee}}^* \varphi_{q_K}(K \times _B K^{\vee} \to K)^{\dagger} p_{K}^!} & {0_{K \oplus K^{\vee}}^* \varphi_{q_K} p_{K \oplus K^{\vee}}^{\dagger}[\rank K]} & {\id[\rank K].}
	\arrow["\cong", from=1-1, to=1-2]
	\arrow["\counit", from=1-1, to=2-1]
	\arrow["\cong", from=1-2, to=1-3]
	\arrow[equals, from=1-3, to=2-3]
	\arrow["\cong"', from=2-1, to=2-2]
	\arrow["{\stab^{K}_{K \oplus K^{\vee}}}"', from=2-3, to=2-2]
	\arrow["\cong"{description}, shift left=3, draw=none, from=2-3, to=2-2]
\end{tikzcd}\]
    Since $\rank K$ is even,
    we have $\stab^{K}_{K \oplus K^{\vee}} = \stab^{K^{\vee}}_{K \oplus K^{\vee}}$.
    Then the claim immediately follows from the construction of the stabilization isomorphism.
    
    \textit{Case 3: The general case.}

    We now construct the map $\red_E$ for general $B$ by gluing.
    To see this, take smooth morphisms from finite type affine schemes $h_1 \colon B_1 \to B$ and $h_2 \colon B_2 \to B$ along with a smooth morphism $g \colon B_1 \to B_2$ such that $h_1 \cong h_2 \circ g$.
    Consider the following diagram:
    \begin{equation}\label{eq-sm-compat-FS-red}
    \begin{tikzcd}
	{g^\dag (p_{h_2^*{E}^{\vee}})_* \mathfrak{F}_{h_2^*{E}} p_{h_2^*{E}}^!} &[40pt] {g^\dag[r]} \\
	{ (p_{h_1^*{E}^{\vee}})_* \mathfrak{F}_{h_1^*{E}} p_{h_1^*{E}}^! g^\dag} & {g^\dag[r]},
	\arrow["{g^\dag\red_{h_2^*{E}}}", from=1-1, to=1-2]
	\arrow["\cong"', from=1-1, to=2-1]
	\arrow[equals, from=1-2, to=2-2]
	\arrow["{\red_{h_1^*{E}}}"', from=2-1, to=2-2]
    \end{tikzcd}
    \end{equation}
    where the left-hand vertical isomorphism is induced by \eqref{eq-FS-fiber-pullback} (along with the exchange isomorphisms for smooth pullbacks with respect to $*$-pushforward and $!$-pullback).
    By the construction in Case~2, this square commutes and is moreover compatible with compositions in $g$ up to homotopy.
    
    In other words, consider the $\infty$-category whose objects are affine derived schemes $B'$ smooth of finite type over $B$, and whose morphisms are smooth morphisms over $B$.
    We have a natural transformation
    \[
    F \colon \Dbc(-) \to \Dbc(-)
    \]
    assigning to each $h \colon B' \to B$ the functor $p_{h^*E, *} \mathfrak{F}_{h^* E} p_{h^*E}^! \colon \Dbc(B') \to \Dbc(B')$.
    Note that the isomorphism $\red_{h^*E}$ shows in particular that this functor is perverse t-exact up to a shift.
    Furthermore, the homotopy associativity of the commutative squares \eqref{eq-sm-compat-FS-red} yields an equivalence $F|_{\Perv(-)} \cong \id[r]$ on perverse sheaves.
    Using \cite[Proposition 2.17]{KKPS1}, this equivalence extends to the entire derived category.
    Finally, passing to the limit over $B'$ yields the desired equivalence \eqref{eq-dimred-FS}.
\end{proof}

For later use, we collect some properties of the dimensional reduction isomorphism \eqref{eq-dimred-FS} for the Fourier--Sato transform. We adopt the notation from \cref{prop:dim red Fourier-Sato}:

\begin{proposition}\label{prop-property-FS-dimred}
   Let $g \colon B' \to B$ be a morphism between lft higher Artin stacks and denote by $g_{E} \colon E |_{B'} \to E$ and $g_{E^{\vee}} \colon E^{\vee} |_{B'} \to E^{\vee}$ the base changes.
   \begin{enumerate}
        \item \label{item-FS-dimred-base-smooth}
        If $g$ is smooth, then the following diagram commutes:

        \begin{equation}\label{eq-FS-dimred-base-smooth}
            \begin{tikzcd}
        	{ p_{E|_{B'}^{\vee}, *} \mathfrak{F}_{E|_{B'}} p_{E|_{B'}}^! g^{\dagger} } & { p_{E|_{B'}^{\vee}, *} \mathfrak{F}_{E|_{B'}} g_{E}^{\dagger} p_{E}^!} & { p_{E|_{B'}^{\vee}, *} g_{E^{\vee}}^{\dagger} \mathfrak{F}_{E} p_{E}^!} & {g^{\dagger}  p_{E^{\vee}, *} \mathfrak{F}_{E} p_{E}^!} \\
        	{g^![r]} &&& {g^![r].}
        	\arrow["\sim", from=1-1, to=1-2]
        	\arrow["{{\red_{E|_{B'}}}}"', from=1-1, to=2-1]
        	\arrow["{{{(\ref{eq-FS-base-pullback})}}}", from=1-2, to=1-3]
        	\arrow["{{\Ex^!_*}}", from=1-3, to=1-4]
        	\arrow["{{\red_{E}}}", from=1-4, to=2-4]
        	\arrow[equals, from=2-1, to=2-4]
        \end{tikzcd}
        \end{equation}

       \item \label{item-FS-dimred-base-finite} If $g$ is finite, then the following diagram commutes:
    \begin{equation}\label{eq-FS-dimred-base-finite}
    \begin{tikzcd}
    	{p_{E^{\vee}, *} \mathfrak{F}_{E} p_{E}^! g_*} & {p_{E^{\vee}, *} \mathfrak{F}_{E}  g_{E, *} p_{E|_{B'}}^!} & {p_{E^{\vee}, *}   g_{E^{\vee}, *} \mathfrak{F}_{E |_{B'}} p_{E|_{B'}}^!} & {g_* p_{E|_{B'}^{\vee}, *}   \mathfrak{F}_{E |_{B'}} p_{E|_{B'}}^!} \\
    	{g_* [r]} &&& {g_* [r].}
    	\arrow["{\Ex^!_*}", from=1-1, to=1-2]
    	\arrow["{{{\red_{E}}}}"', from=1-1, to=2-1]
    	\arrow["{{{(\ref{eq-FS-base-pushforward})}}}", from=1-2, to=1-3]
    	\arrow["\sim", from=1-3, to=1-4]
    	\arrow["{{{\red_{E|_{B'}}}}}", from=1-4, to=2-4]
    	\arrow[equals, from=2-1, to=2-4]
    \end{tikzcd}
    \end{equation}
    \end{enumerate}

    Now take a morphism of perfect complexes $f \colon E' \to E$.
    \begin{enumerate}[resume]
      \item \label{item-FS-dimred-fibre-smooth} If $f$ is a smooth morphism, then the following diagram commutes:
    \begin{equation}\label{eq-FS-dimred-fibre-smooth}
    \begin{tikzcd}
    	{(p_{E'^{\vee}})_*  \mathfrak{F}_{E'}  p_{E'}^! [- \dim f]} & {(p_{E'^{\vee}})_*  \mathfrak{F}_{E'} f^{\dagger} p_{E}^! } & {(p_{E'^{\vee}})_* f^{\vee}_* \mathfrak{F}_{E} p_{E}^! } & {p_{E^{\vee}, *} \mathfrak{F}_{E} p_{E}^! } \\
    	{(-)[r]} &&& {(-)[r].}
    	\arrow["\sim", from=1-1, to=1-2]
    	\arrow["{\red_{E'}}"', from=1-1, to=2-1]
    	\arrow["{(\ref{eq-FS-fiber-pullback})}", from=1-2, to=1-3]
    	\arrow["\sim", from=1-3, to=1-4]
    	\arrow["{\red_E}", from=1-4, to=2-4]
    	\arrow[equals, from=2-1, to=2-4]
    \end{tikzcd}
    \end{equation}
  \item \label{item-FS-dimred-fibre-finite}
        If $f$ is a closed immersion, then the following diagram commutes:
        \begin{equation}\label{eq-FS-dimred-fibre-finite}
        \adjustbox{scale=0.85, center}
        {\begin{tikzcd}
        	{p_{E^{\vee}, *} \mathfrak{F}_{E}  p_{E}^! } & {p_{E^{\vee}, *} \mathfrak{F}_{E} f_! f^! p_{E}^! } & {p_{E^{\vee}, *} (f^{\vee})^{\dagger} \mathfrak{F}_{E'}  f^! p_{E}^! } & {0_{E^{\vee}}^* (f^{\vee})^{\dagger} \mathfrak{F}_{E'}  f^! p_{E}^! } & {p_{E'^{\vee}, *} \mathfrak{F}_{E'} p_{E'}^! [- \dim f]} \\
        	{\id[r]} &&&& {\id[r].}
        	\arrow["{{\red_{E}}}"', from=1-1, to=2-1]
        	\arrow["\counit"', from=1-2, to=1-1]
        	\arrow["{{(\ref{eq-FS-fiber-pushforward})}}", from=1-2, to=1-3]
        	\arrow["\unit", from=1-3, to=1-4]
        	\arrow["\unit"', "\sim", from=1-5, to=1-4]
        	\arrow["{{\red_{E'}}}", from=1-5, to=2-5]
        	\arrow[equals, from=2-1, to=2-5]
        \end{tikzcd}}
        \end{equation}
    Here, the rightmost upper horizontal map is invertible by \cref{prop:contraction}. 
    \end{enumerate}

    Finally, let $E_i \to B_i$ be perfect complexes over lft higher Artin stacks for $i = 1, 2$.
    \begin{enumerate}[resume]
    \item \label{item-FS-dimred-product}
    The following diagram commutes:
    \[
    \adjustbox{scale=0.82, center}
    {\begin{tikzcd}
	{p_{E_1^{\vee}, *} \mathfrak{F}_{E_1} p_{E_1}^! (-) \boxtimes p_{E_2^{\vee}, *} \mathfrak{F}_{E_2} p_{E_2}^!(-)} & {(p_{E_1^{\vee}} \times p_{E_2^{\vee}})_* (\mathfrak{F}_{E_1} \boxtimes \mathfrak{F}_{E_2}) (p_{E_1} \times p_{E_2})^!(- \boxtimes -)} & {(p_{(E_1 \oplus E_2)^{\vee}})_* \mathfrak{F}_{E_1 \oplus E_2} p_{E_1 \oplus E_2}^!(- \boxtimes -)} \\
	{\id[r_1] \boxtimes \id[r_2]} && {\id [r_1 + r_2].}
	\arrow["\TS", from=1-1, to=1-2]
	\arrow["{\red_{E_1} \boxtimes \red_{E_2}}"', from=1-1, to=2-1]
	\arrow["\sim", from=1-2, to=1-3]
	\arrow["{\red_{E_1 \oplus E_2}}", from=1-3, to=2-3]
	\arrow["\sim"', from=2-1, to=2-3]
\end{tikzcd}}\]
    \end{enumerate}
   
\end{proposition}

\begin{proof}

    We first deal with the case when $E$ has tor-amplitude $\leq 0$.
    In this, case, the zero section $0_E \colon B \to E$ is a closed immersion and the map $\red_E$ is constructed as the composition
    \[
    p_{E^{\vee}, *} \mathfrak{F}_{E} p_{E}^{!}  \xleftarrow[\sim ]{\unit} p_{E^{\vee}, *} \mathfrak{F}_{E} 0_{E, !} \xrightarrow[]{\eqref{eq-FS-fiber-pushforward}}  p_{E^{\vee}, *} p_{E^{\vee}}^{\dagger} \mathfrak{F}_{E} \xrightarrow[]{\sim} p_{E^{\vee}, *} p_{E^{\vee}}^* \mathfrak{F}_{E} [r] \xleftarrow[\sim]{\unit} \id[r].
    \]
    Here, the invertibility of the first map follows from \eqref{lem-pre-dim-red-cl-imm} and the invertibility of the last map follows from \cref{prop:contraction}.
    Then the commutativity of \eqref{eq-FS-dimred-base-smooth}, \eqref{eq-FS-dimred-base-finite}, \eqref{eq-FS-dimred-fibre-smooth}, \eqref{eq-FS-dimred-fibre-finite} follows from the commutativity of diagrams \eqref{eq-base-fibre-smooth-finite}, \eqref{eq-FS-base-fibre-finite-finite}, \eqref{eq-FS-fibre-fibre-smooth-finite} and the associativity of \eqref{eq-FS-fiber-pushforward}.

    Now we prove the general cases. For (\ref{item-FS-dimred-base-smooth}), by the construction of $\red_E$ in \eqref{eq-dimred-FS}, 
        we may assume that $B$ is a quasi-projective scheme and $E$ admits a global resolution $F \to E$ of even relative dimension. Then \eqref{eq-FS-base-fibre-smooth-smooth} implies that we may replace $E$ by $F$, for which we have already proved the commutativity in the last paragraph. Other three statements can be proved in analogous manners.

    Regarding (\ref{item-FS-dimred-product}), if $E$ has Tor-amplitude $\leq 0$, then the claim follows from the compatibility between the Thom--Sebastiani isomorphism \eqref{eq-FS-TS} and \eqref{eq-FS-fiber-pushforward} as established in \cref{prop-FS-product}.
    The general case can be reduced to this case using the compatibility between the isomorphism \eqref{eq-FS-TS} and \eqref{eq-FS-base-pullback} and \eqref{eq-FS-fiber-pullback} proved in \cref{prop-FS-product}.
\end{proof}

\subsection{The dimensional reduction theorem}

We will now prove the dimensional reduction theorem in general,
making use of the deformation to the normal bundle to reduce the general case to the linear one considered in \cref{prop:dim red Fourier-Sato}.

The \defterm{deformation to the normal bundle} is an $\bA^1$-family of morphisms of derived stacks $Y \times \bA^1 \to \mathfrak{B}$ which degenerates the morphism $g \colon Y \to B$ to the zero section of the normal bundle $\N_{Y/B}$.
In this context, the normal bundle is
\[ \N_{Y/B} \coloneqq \T[1](Y/B) \coloneqq \Tot_{Y}(\bL_{Y/B}^\vee[1]), \]
or equivalently the dual of the $(-1)$-shifted cotangent bundle $\T^*[-1](Y/B)$, and the derived stack $\mathfrak{B}$ is defined as the mapping stack
\[ \mathfrak{B} \coloneqq \Map_{B\times\bA^1}(B\times\{0\}, Y\times\bA^1). \]
This fits in the Cartesian diagram of derived stacks
\[\begin{tikzcd}
    Y \ar{r}{0_{\bA^1}}\ar{d}{0_{N_{Y/B}}}
    & Y \times \bA^1 \ar{d}{\hat{g}}\ar[leftarrow]{r}{j_{\bA^1}}
    & Y \times \Gm \ar{d}{g \times \id_{\Gm}}
    \\
    \N_{Y/B} \ar{r}{0_{\mathfrak{B}}}\ar{d}{t_0}
    & \mathfrak{B} \ar[leftarrow]{r}{j_{\mathfrak{B}}}\ar{d}{t}
    & B \times \Gm \ar[equals]{d}
    \\
    B \ar{r}{0_{\bA^1}}
    & B\times \bA^1 \ar[leftarrow]{r}{j_{\bA^1}}
    & B\times \Gm,
\end{tikzcd}\]
where $0_{\bA^1}$ is the inclusion of the zero section, $j_{\bA^1}$ is its complement, $t_0$ is the composite $g\circ p_{\N_{Y/B}} \colon \N_{Y/B} \to Y \to B$, and $t \circ \hat{g} \simeq g \times \id_{\bA^1}$.
See \cite{HekkingKhanRydh,CalaqueSafronov}. For later use, we recall the following statement from \cite{HekkingKhanRydh}:

\begin{proposition}\label{prop-property-def-space-map}
    Consider the following commutative diagram
    \[\begin{tikzcd}
    	{Y'} & Y \\
    	{B'} & {B.}
    	\arrow["{p_Y}", from=1-1, to=1-2]
    	\arrow["{g'}"', from=1-1, to=2-1]
    	\arrow["g", from=1-2, to=2-2]
    	\arrow["{p_B}"', from=2-1, to=2-2]
    \end{tikzcd}\]
    We let $\mathfrak{B}$ and $\mathfrak{B}'$ denote the deformation space for $g \colon Y \to B$ and $g' \colon Y' \to B'$ respectively, and $p_{\mathfrak{B}} \colon \mathfrak{B}' \to \mathfrak{B}$ be the induced map.
    \begin{enumerate}
        \item If $p_B$, $p_{Y}$ are smooth, $\eta \colon Y' \to Y \times_{B} B'$ is quasi-smooth and the map $p_{\mathfrak{B}}$ is smooth.
        \item If $p_{B}$ is a closed immersion, the map $\eta \colon Y' \to Y \times_{B} B'$ induces an isomorphism on classical truncations and $\mathbb{L}_{Y' / Y \times_{B} B'}$ has Tor-amplitude $\leq -2$, the map $p_{\mathfrak{B}}$ is a closed immersion.
    \end{enumerate}
\end{proposition}

We make use of this construction as follows.
Denote by
\[\hat{\tau} \colon \mathfrak{X} \coloneqq \T^*[-1](Y\times\bA^1/\mathfrak{B}) \to Y \times \bA^1\]
the relative $(-1)$-shifted cotangent bundle of $\hat{g}$.
The projection $\hat{\pi} \colon \mathfrak{X} \to \mathfrak{B}$ admits a canonical relative exact $(-1)$-shifted symplectic structure and orientation, both stable under base change by construction.
We have the following Cartesian diagram:
\[\begin{tikzcd}
    N_{Y/B}^\vee \ar{r}{0_{\mathfrak{X}}}\ar{d}{p_{N_{Y/B}^\vee}}
    & \mathfrak{X} \ar{d}{\hat{\tau}}\ar[leftarrow]{r}{j_{\mathfrak{X}}}
    & \T^*[-1](Y/B) \times \Gm \ar{d}{\tau \times \id_{\Gm}}
    \\
    Y \ar{d}{0_{N_{Y/B}}}\ar{r}{0_{\bA^1}}
    & Y \times \bA^1 \ar{d}{\hat{g}}\ar[leftarrow]{r}{j_{\bA^1}}
    & Y \times \Gm \ar{d}{g \times \id_{\Gm}}
    \\
    \N_{Y/B} \ar{r}{0_{\mathfrak{B}}}
    & \mathfrak{B} \ar[leftarrow]{r}{j_{\mathfrak{B}}}
    & B \times \Gm,
\end{tikzcd}\]
where the vertical composites are $\pi \times \id_{\Gm}$, $\hat{\pi}$, and $u_{\N_{Y/B}^\vee}$ \eqref{eq:uEveeE} from right to left.
Informally speaking, we may regard $\hat{\pi} \colon \mathfrak{X} \to \mathfrak{B}$ as a family degenerating the oriented $(-1)$-shifted symplectic fibration $\pi \colon \T^*[-1](Y/B) \to B$ to the oriented $(-1)$-shifted symplectic fibration $\pi_0 \coloneqq u_{\N_{Y/B}^\vee} \colon \N_{Y/B}^\vee \to \N_{Y/B}$.
To prove \cref{thm:dimensionalreduction}, the idea is to use the perverse pullback functor
\[
    \hat{\pi}^\varphi \colon \Perv(\mathfrak{B}) \to \Perv(\mathfrak{X})
\]
to reduce the dimensional reduction theorem for $\pi^\varphi$ to the statement in the linear case $\pi_0^\varphi = \mathfrak{F}_{N_{Y/B}}$ proven in \cref{prop:dim red Fourier-Sato}.
Indeed, we may regard the functor
\begin{equation*}
    \hat{\rho}_{Y/B} \colon
    \Perv(B \times \bA^1) \xrightarrow{t^!}
    \Perv(\mathfrak{B}) \xrightarrow{\hat{\pi}^\varphi}
    \Perv(\mathfrak{X}) \xrightarrow{\hat{\tau}_*}
    \Perv(Y \times \bA^1)
\end{equation*}
as interpolating between
\[
    \rho_{Y\times\Gm/B\times\Gm} \coloneqq (\tau \times \id_{\Gm})_* (\pi\times\id_{\Gm})^\varphi(-),
    \qquad
    \rho_{0,Y/B} \coloneqq p_{N^\vee_{Y/B},*} \pi_0^\varphi t_0^!(-) \cong p_{N^\vee_{Y/B},*} \pi_0^\varphi p_{N^\vee_{Y/B}}^! g^!(-),
\]
where $\rho_{0,Y/B} \cong g^!(-)[-\dim(Y/B)]$ by \cref{prop:dim red Fourier-Sato}.
The following lemma states, informally speaking, that (1) $\hat{\rho}_{Y/B}$ commutes with $j_{\bA^1}^*$, (2) it commutes with $0_{\bA^1,*}$, (3) it is smooth-local on $Y$ and $B$, (4) it preserves constancy along $\bA^1$, and (5) its fiber over $0 \in \bA^1$ only depends on the $0$-fiber of its argument.

\begin{lemma}\label{lem:rhohat}
Let $g \colon Y \to B$ be an lfp geometric morphism of derived stacks with $B$ an lft higher Artin stack. Then we have:
\begin{enumerate}
    \item\label{item:rhohat/generic} There is a natural isomorphism $j_{\bA^1}^* \hat{\rho}_{Y/B}(-) \cong \rho_{Y\times\Gm/B\times\Gm}(j_{\bA^1}^*(-))$.

    \item\label{item:rhohat/zeropush} There is a natural isomorphism $\hat{\rho}_{Y/B} 0_{\bA^1,*} (-) \cong 0_{\bA^1,*} \rho_{0,Y/B}$.

    \item\label{item:rhohat/local} Given a commutative square
    \[\begin{tikzcd}
        Y' \ar{r}{p_Y}\ar{d}{g'}
        & Y \ar{d}{g}
        \\
        B' \ar{r}{p_B}
        & B,
    \end{tikzcd}\]
    where $p_B$ and $p_Y$ are smooth, there is a canonical isomorphism $(p_Y \times \id_{\bA^1})^\dag \hat{\rho}_{Y/B} \cong \hat{\rho}_{Y'/B'} (p_B \times \id_{\bA^1})^\dag$.

    \item\label{item:rhohat/constant} The natural map $p_{\bA^1}^* p_{\bA^1,*} (\hat{\rho}_{Y/B} p_{\bA^1}^*(-)) \xrightarrow{\mathrm{counit}} \hat{\rho}_{Y/B} p_{\bA^1}^*(-)$ is invertible.
    \item\label{item:rhohat/zeropull} The natural map $0_{\bA^1}^* \hat{\rho}_{Y/B}(-) \xrightarrow{\mathrm{unit}}0_{\bA^1}^* \hat{\rho}_{Y/B}(0_{\bA^1,*} 0_{\bA^1}^* (-))$ is invertible.
\end{enumerate}
\end{lemma}
\begin{proof}
    Claims~(\ref{item:rhohat/generic}) and (\ref{item:rhohat/zeropush}) follow from the isomorphisms $\alpha_{j_{\mathfrak{B}}, (j_{\mathfrak{X}}, \id)} \colon j_{\mathfrak{X}}^* \hat{\pi}^\varphi \cong (\pi\times\id_{\Gm})^\varphi j_{\mathfrak{B}}^*$ of \eqref{eq-Lagrangian functoriality of perverse pullback} and $\beta_{0_{\mathfrak{B}}} \colon \hat{\pi}^\varphi 0_{\mathfrak{B},*} \cong 0_{\mathfrak{X},*} \pi_0^\varphi$ of \eqref{eq-perverse-pullback-symplectic-finite-base-change-compatibility}.

    For (\ref{item:rhohat/local}), there is a commutative diagram
    \[\begin{tikzcd}
        Y' \times \bA^1 \ar{d}{p_Y\times\id}\ar{r}{\hat{g}'}
        & \mathfrak{B}' \ar{d}{Dp}\ar{r}{t'}
        & B' \times \bA^1 \ar{d}{p_B\times\id}
        \\
        Y \times \bA^1 \ar{r}{\hat{g}}
        & \mathfrak{B} \ar{r}{t}
        & B \times \bA^1,
    \end{tikzcd}\]
    where the upper row is the deformation to the normal bundle of $g' \colon Y' \to B'$.
    By \cref{prop-property-def-space-map}, the morphism $Dp \colon \mathfrak{B}' \to \mathfrak{B}$ is smooth, with fiber $p_{B} \colon B' \to B$ over $1 \in \bA^1$ and fiber $Np \colon N_{Y'/B'} \to N_{Y/B} \times_B B' \to N_{Y/B}$ over $0 \in \bA^1$.
    Unravelling the definitions of $\hat{\rho}_{Y/B}$ and $\hat{\rho}_{Y'/B'}$, it will suffice to construct canonical isomorphisms
    \begin{align}
        (Dp)^\dag t^! &\cong t'^!  (p_B \times \id_{\bA^1})^\dag \label{eq:rhohatlocal1}
        \\
        (p_Y \times \id_{\bA^1})^\dag \hat{\tau}_* \hat{\pi}^\varphi &\cong \hat{\tau}_*' \hat{\pi}'^\varphi (Dp)^\dag \label{eq:rhohatlocal2}
    \end{align}
    where $\hat{\pi}' \colon \mathfrak{X}' \coloneqq \T^*[-1](Y' \times \bA^1 /\mathfrak{B}') \to \mathfrak{B}'$ and $\tau' \colon \mathfrak{X}' \to Y' \times \mathbb{A}^1$.
    Since $Dp$ is smooth of the same (virtual) dimension as $p$, \eqref{eq:rhohatlocal1} follows from functoriality.
    For \eqref{eq:rhohatlocal2}, we apply \cref{lem:Lagrangian funct}(\ref{item:Lagrangian funct/2}) to the square
    \[\begin{tikzcd}
        Y' \times \bA^1 \ar{r}{h}\ar{d}{p_Y\times\id}
        & \mathfrak{B} \ar[equals]{d}
        \\
        Y \times \bA^1 \ar{r}{\hat{g}}
        & \mathfrak{B},
    \end{tikzcd}\]
    where  $h \coloneqq \hat{g} \circ (p_Y\times\id)$, to obtain a canonical isomorphism
    \[\tau_{h,*} \pi_{h}^\varphi \cong (p_Y\times\id)^\dag \hat{\tau}_* \hat{\pi}^\varphi\]
    where $\tau_h \colon \T^*[-1](h) \to Y' \times \bA^1$ and $\pi_h \coloneqq h \circ \tau_h \colon \T^*[-1](Y' \times \bA^1 / \mathfrak{B}) \to \mathfrak{B}$ are the projections.
    To compute the left-hand side, note that  $\pi_h$ is the symplectic pushforward of $\hat{\pi}'$ along the smooth morphism $Dp \colon \mathfrak{B'} \to \mathfrak{B}$.
    Hence we have
    \[
    \tau_{h,*} \pi_{h}^\varphi \xrightarrow[\sim]{} \hat{\tau}'_* i_* \pi_h^\varphi  \xleftarrow[\sim]{\eqref{eq-symp-push-perverse-pullback-smooth}} \hat{\tau}'_* \hat{\pi}'^\varphi (Dp)^\dag
    \]
    where $i \colon \T^*[-1](Y' \times \bA^1 / \mathfrak{B}) \hookrightarrow \T^*[-1](Y' \times \bA^1 / \mathfrak{B}')$ is the natural map.
    Combining the previous two displayed isomorphisms yields \eqref{eq:rhohatlocal2} and hence concludes the proof of (\ref{item:rhohat/local}).

    To prove (\ref{item:rhohat/constant}) and (\ref{item:rhohat/zeropull}), we may use (\ref{item:rhohat/local}) to work smooth-locally on $Y$ and $B$.
    In particular, we may assume that $g \colon Y \to B$ is a closed immersion.
    In this case, \cref{prop-property-def-space-map} implies that $\hat{g} \colon Y \times \bA^1 \to \mathfrak{B}$ is also a closed immersion, so \cref{cor-dim-red-cl-imm} yields $\hat{\rho}_{Y/B} \cong (g \times \id_{\bA^1})^![-\dim(Y/B)]$.
    It is then obvious that both maps in (\ref{item:rhohat/constant}) and (\ref{item:rhohat/zeropull}) are invertible.
\end{proof}

We are now ready to conclude the proof.

\begin{proof}[Proof of \cref{thm:dimensionalreduction}]
    By \cref{lem:rhohat}, we have natural isomorphisms
    \begin{align*}
        \rho_{Y/B}
        \xrightarrow[\sim]{\eqref{item:rhohat/generic}} 1_{\bA^1}^* \hat{\rho}_{Y/B} p_{\bA^1}^* 
        \xrightarrow[\sim]{\eqref{item:rhohat/constant}} 
         0_{\bA^1}^* \hat{\rho}_{Y/B} p_{\bA^1}^*
        \xrightarrow[\sim]{\eqref{item:rhohat/zeropull}} 0_{\bA^1}^* \hat{\rho}_{Y/B} 0_{\bA^1,*}
        \xrightarrow[\sim]{\eqref{item:rhohat/zeropush}}0_{\bA^1}^* 0_{\bA^1,*} {\rho}_{0,Y/B}
        \xleftarrow[\sim]{} {\rho}_{0,Y/B}.
    \end{align*}
    Finally, we have ${\rho}_{0,Y/B} \cong g^*(-)[-\dim(Y/B)]$  by \cref{prop:dim red Fourier-Sato}.
\end{proof}

\subsection{Properties of the relative dimensional reduction isomorphism}

Here, we prove some properties of the relative dimensional reduction isomorphism \eqref{eq-dimrensional-reudction}. We first show that it is compatible with the isomorphisms defined in \cref{thm:perversepullback}.

\begin{proposition}\label{prop-dimensional-reduction-compatible}
Let $B$ be an lft higher Artin stack and $g \colon Y \to B$ an lfp geometric morphism between derived stacks.
    Assume that we are given the following diagram of derived stacks
    \begin{equation}\label{eq-square-of-stacks}
    \begin{tikzcd}
    	{Y'} & Y \\
    	{B'} & {B.}
    	\arrow["{p_Y}", from=1-1, to=1-2]
    	\arrow["{g'}"', from=1-1, to=2-1]
    	\arrow["g", from=1-2, to=2-2]
    	\arrow["{p_B}"', from=2-1, to=2-2]
    \end{tikzcd}
    \end{equation}
    with $g'\colon Y'\rightarrow B$ an lfp geometric morphism. Form the following diagram:
        \begin{equation}
        \begin{tikzcd}\label{eq-cotangent-square-diagram-pullback}
        	& {\T^*[-1](Y/B) \times_{Y} Y' } \\
        	{\T^*[-1](Y'/B')} & {\T^*[-1](Y/B) \times_{B} B'} & {\T^*[-1](Y/B)} \\
        	{Y'} && Y \\
        	{B'} && {B.}
        	\arrow["{{\partial p}}"', from=1-2, to=2-1]
        	\arrow["{{p_{\tau}}}", from=1-2, to=2-2]
        	\arrow["{{\tilde{p}_Y}}", from=1-2, to=2-3]
        	\arrow["{{\tau_{Y' / B'}}}"', from=2-1, to=3-1]
        	\arrow["{{\pi_{Y'/B'}}}"', curve={height=45pt}, from=2-1, to=4-1]
        	\arrow["{{\tilde{p}_B}}"', from=2-2, to=2-3]
        	\arrow["{{\tau_{Y / B}}}", from=2-3, to=3-3]
        	\arrow["{{\pi_{Y/B}}}", curve={height=-45pt}, from=2-3, to=4-3]
            \arrow["{\pi_{Y / B}'}"', {pos=0.3}, from=2-2, to=4-1]
        	\arrow["{{p_{Y}}}", crossing over, from=3-1, to=3-3]
        	\arrow["{{g'}}"', from=3-1, to=4-1]
        	\arrow["g", from=3-3, to=4-3]
        	\arrow["{{p_B}}"', from=4-1, to=4-3]
        \end{tikzcd}
        \end{equation}
        
\begin{enumerate}
    \item \label{item-dimred-compatible-alpha} Assume that $p_B$ and $\eta \colon Y' \to Y \times_{B} B'$ are smooth. Then the following diagram commutes:
    \[\begin{tikzcd}
    	{\tau_{Y'/ B', *} \pi_{Y' / B'}^{\varphi} p_B^{\dagger}} & {\tau_{Y'/ B', *} (\partial p)_* \tilde{p}_Y^{\dagger} \pi_{Y / B}^{\varphi} } & {p_{Y}^{\dagger}\tau_{Y/B, *} \pi_{Y / B}^{\varphi} } \\
    	{(g')^!p_{B}^{\dagger} [- \dim (Y'/B')]} && {p_Y^{\dagger} g^! [- \dim (Y/B)].}
    	\arrow["{\alpha_{p_B, (\tilde{p}_Y, \partial p)}}", from=1-1, to=1-2]
    	\arrow["{\red_{Y' / B'}}"', from=1-1, to=2-1]
    	\arrow["{\Ex^!_*}", from=1-2, to=1-3]
    	\arrow["{\red_{Y/B}}", from=1-3, to=2-3]
    	\arrow["\sim"', from=2-1, to=2-3]
    \end{tikzcd}\]
    \item \label{item-dimred-compatible-beta}  Assume that $p_B$ is a closed immersion, $p_{Y}$ induces an isomorphism on classical truncations and $\mathbb{L}_{Y' / Y \times_{B} B'}$ has Tor-amplitude $\leq -2$. Then the following diagram commutes:
        \[\adjustbox{scale=0.9, center}{
        \begin{tikzcd}
        	&[-10pt]& { p_{Y, *} \tau_{Y'/B', *}  (\partial p)_* (\partial p)^{\dagger}\pi_{Y' / B'}^{\varphi}} & {p_{Y, *} \tau_{Y'/B', *}  \pi_{Y' / B'}^{\varphi} [\dim (\partial p)]} \\
        	{\tau_{Y/B, *} \pi_{Y / B}^{\varphi} p_{B, *}} & {\tau_{Y/B, *}  \tilde{p}_{B, *} (\pi_{Y / B}')^{\varphi}} & {\tau_{Y/B, *}  \tilde{p}_{B, *} p_{\tau, *} (\partial p)^{\dagger}\pi_{Y' / B'}^{\varphi}} \\
        	{g^!p_{B, *}[- \dim (Y/B)]} &&& {p_{Y, *}  g'^! [- \dim (Y' / B')].}
        	\arrow["\unit"', "\sim", from=1-4, to=1-3]
        	\arrow["{\red_{Y' / B'}}", from=1-4, to=3-4]
        	\arrow["{\beta_{p_B}}", from=2-1, to=2-2]
        	\arrow["{\red_{Y/B}}"', from=2-1, to=3-1]
        	\arrow["{\alpha_{(\partial p, p_{\tau})}}", from=2-2, to=2-3]
        	\arrow["\sim"', from=2-3, to=1-3]
        	\arrow["{\Ex^!_*}"', from=3-1, to=3-4]
        \end{tikzcd}
        }\]
    
\end{enumerate}

Now consider the following diagram:

        \begin{equation}\label{eq-cotangent-square-diagram-pushforward}
        \begin{tikzcd}
        	& {\T^*[-1](Y'/B) } & {\T^*[-1](Y/B) \times_{Y} Y'} \\
        	{\T^*[-1](Y'/B') } && {\T^*[-1](Y/B) } \\
        	{Y'} && Y \\
        	{B'} && {B.}
        	\arrow["r"', from=1-2, to=2-1]
        	\arrow["{\partial p_Y}"', from=1-3, to=1-2]
        	\arrow["{\tilde{p}_Y}", from=1-3, to=2-3]
        	\arrow["{\pi_{Y/B}}", curve={height=-50pt}, from=2-3, to=4-3]
        	\arrow["{\tau_{Y'/B'}}"', from=2-1, to=3-1]
        	\arrow["{\pi_{Y'/B'}}"', curve={height=60pt}, from=2-1, to=4-1]
        	\arrow["{\tau_{Y/B}}", from=2-3, to=3-3]
        	\arrow["{g'}"', from=3-1, to=4-1]
        	\arrow["g", from=3-3, to=4-3]
        	\arrow["{p_B}"', from=4-1, to=4-3]
        	\arrow["{\pi_{Y'/B}}"'{pos=0.37}, from=1-2, to=4-3]
            \arrow["{p_Y}", crossing over, from=3-1, to=3-3]
            \arrow["{\tau_{Y/B}'}"'{pos=0.6}, crossing over, from=1-3, to=3-1]
        \end{tikzcd}
        \end{equation}

        \begin{enumerate}[resume]
            \item \label{item-dimred-compatible-gamma}  If $p_B$ and $p_Y$ are smooth, the following diagram commutes:
            \[\begin{tikzcd}
            	&& { \tau_{Y/B, *}' \tilde{p}_Y^{\dagger} \pi_{Y/B}^{\varphi}} & {p_Y^{\dagger} \tau_{Y/B, *}\pi_{Y/B}^{\varphi}} \\
            	{ \tau_{Y'/B', *}  \pi_{Y'/B'}^{\varphi}  p_B^!} & { \tau_{Y'/B', *} r_* \pi_{Y'/B}^{\varphi}} & { \tau_{Y'/B', *} r_* (\partial p_Y)_* \tilde{p}_Y^{\dagger} \pi_{Y/B}^{\varphi}} \\
            	{ (g')^! p_B^{\dagger}[- \dim(Y' / B')]} &&& {p_Y^{\dagger} g^![- \dim(Y / B)].}
            	\arrow["{\Ex^!_*}"', "\sim", from=1-4, to=1-3]
            	\arrow["{{\red_{Y  /B}}}"', from=1-4, to=3-4]
            	\arrow["{{\gamma_{p_B}}}", from=2-1, to=2-2]
            	\arrow["{{\red_{Y' / B'}}}"', from=2-1, to=3-1]
            	\arrow["{{\alpha_{(\tilde{p}_Y, \partial p_Y)}}}", from=2-2, to=2-3]
            	\arrow["\sim", from=2-3, to=1-3]
            	\arrow["\sim"', from=3-1, to=3-4]
            \end{tikzcd}\]
            
        \item \label{item-dimred-compatible-epsilon} If $p_B$ is a closed immersion and $p_Y$ induces an isomorphism on classical truncations and $\mathbb{L}_{Y'/Y}$ has Tor-amplitude $\leq -2$, the following diagram commutes:

        \[\begin{tikzcd}
        	{\tau_{Y/B, *} \pi_{Y/B}^{\varphi} p_{B, *}} & {\tau_{Y/B, *} \tilde{p}_{Y, *} (\partial p_Y)^{\dagger}\pi_{Y'/B}^{\varphi} p_{B, *}} & {\tau_{Y/B, *} \tilde{p}_{Y, *} (\partial p_Y)^{\dagger} r^{\dagger} \pi_{Y'/B'}^{\varphi} } \\
        	&& {p_{Y, *}\tau_{Y'/B', *} r_* (\partial p_Y)_*  (\partial p_Y)^{\dagger} r^{\dagger} \pi_{Y'/B'}^{\varphi} } \\
        	&& {p_{Y, *} \tau_{Y'/B', *}  \pi_{Y'/B'}^{\varphi}[\dim (Y' / Y \times_B B')] } \\
        	{g^! p_{B, *}[- \dim(Y/B)]} && {p_{Y, *} g'^! [- \dim(Y/B)].}
        	\arrow["{\alpha_{(\partial p_Y, \tilde{p}_Y)}}", from=1-1, to=1-2]
        	\arrow["{\red_{Y/B}}"', from=1-1, to=4-1]
        	\arrow["{\varepsilon_{p_B}}", from=1-2, to=1-3]
        	\arrow["\sim", from=1-3, to=2-3]
        	\arrow["\unit"', "\sim", from=3-3, to=2-3]
        	\arrow["{\red_{Y' / B'}}", from=3-3, to=4-3]
        	\arrow["{\Ex^!_*}"', from=4-1, to=4-3]
        \end{tikzcd}\]
    \end{enumerate}

    Now assume that $B_i$ for $i=1,2$ are lft higher Artin stacks and $g_i\colon Y_i\rightarrow B_i$ are lfp geometric morphisms of derived stacks.
    \begin{enumerate}[resume]
        \item \label{item-dimred-compatible-tau}
        We let $\tau_i \colon \T^*[-1](Y_i / B_i) \to Y_i$ and $\pi_i \colon \T^*[-1](Y_i / B_i) \to B_i$ be the natural projection. Then the following diagram commutes:
        \[\begin{tikzcd}
        	{\tau_{1, *} \pi^{\varphi}_1 (-) \boxtimes \tau_{2, *} \pi^{\varphi}_2 (-)} & {(\tau_{1} \times \tau_{2})_*( \pi^{\varphi}_1 (-) \boxtimes  \pi^{\varphi}_2 (-))} & {(\tau_{1} \times \tau_{2})_* ( \pi^{\varphi}_1  \boxtimes  \pi^{\varphi}_2 ) (- \boxtimes  -)} \\
        	{g_1^! (-) \boxtimes g_2^!(-)} && {(g_1 \times g_2)^!(- \boxtimes -).}
        	\arrow["\sim", from=1-1, to=1-2]
        	\arrow["{\red_{Y_1 / B_1} \boxtimes \red_{Y_2 / B_2}}"', from=1-1, to=2-1]
        	\arrow["\TS", from=1-2, to=1-3]
        	\arrow["{\red_{Y_1 \times Y_2 / B_1 \times B_2}}", from=1-3, to=2-3]
        	\arrow["\sim"', from=2-1, to=2-3]
        \end{tikzcd}\]
        \end{enumerate}

\end{proposition}

\begin{proof}

    First, to prove the first four statements, we may replace the commutative square \eqref{eq-square-of-stacks} by the square
    \[\begin{tikzcd}
    	{Y' \times \mathbb{A}^1} & {Y \times \mathbb{A}^1} \\
    	{\mathfrak{B}'} & {\mathfrak{B},}
    	\arrow["{p_Y \times\id_{\mathbb{A}^1}}", from=1-1, to=1-2]
    	\arrow["{\hat{g}'}"', from=1-1, to=2-1]
    	\arrow["\hat{g}", from=1-2, to=2-2]
    	\arrow["Dp"', from=2-1, to=2-2]
    \end{tikzcd}\]
    where $\mathfrak{B}$ and ${\mathfrak{B}'}$ are the deformation spaces for $g \colon Y \to B$ and $g' \colon Y' \to B'$ respectively,
    and the complex we consider is of the form $(\mathfrak{B} \to B \times \mathbb{A}^1)^{!}(-)$ for (1), (3) and $(\mathfrak{B'} \to B' \times \mathbb{A}^1)^{!}(-)$ for (2) and (4).
    Indeed, the above diagram is defined over $\mathbb{A}^1$, and the original diagram \eqref{eq-square-of-stacks} multiplied by $\mathbb{G}_{\mathrm{m}}$ is recovered as the restriction over $\mathbb{G}_{\mathrm{m}}$. Also, the property for the map $p_{B}$ persists to $Dp$ by \cref{prop-property-def-space-map} and a straightforward computation shows that the cotangent complex $\mathbb{L}_{Y' \times \mathbb{A}^1 / (Y \times \mathbb{A}^1) \times_{\mathfrak{B}} \mathfrak{B}'}$ has Tor-amplitude $\leq -2$ as long as $\mathbb{L}_{Y' / Y \times_{B} B'}$ does.
    Note that all complexes appearing in the diagram of the statements are constant along the $\mathbb{A}^1$-direction. Therefore it is enough to prove the statement after restricting to the fibre over $\{ 0 \} \in \mathbb{A}^1$. Further, using \cref{lem:rhohat} (\ref{item:rhohat/zeropull}) and the commutativity of the map $\beta$ with the other exchange isomorphisms for the perverse pullback established in \cite[Theorem 5.23]{KKPS1}, it is enough to prove the statement for the following commutative square
    \[\begin{tikzcd}
    	{Y'} && Y \\
    	{E'} & {E \times_{Y} Y'} & E
    	\arrow["g", from=1-1, to=1-3]
    	\arrow["{0_{E'}}"', from=1-1, to=2-1]
    	\arrow["{0_E}", from=1-3, to=2-3]
    	\arrow["f"', from=2-1, to=2-2]
    	\arrow["{g_E}"', from=2-2, to=2-3]
    \end{tikzcd}\]
    where $E$ and $E'$ are perfect complexes over $B$ and $B'$, $g_E$ is the base change of $g$, $f$ is a morphism of perfect complexes and the vertical maps are the zero sections.
    Further, we may assume that the object in the source is of the form $(E \to Y)^!(-)$ for (1), (3) and $(E' \to Y')^!(-)$ for (2), (4).
    We now present the proof of each statement in this setting.
    
    For (\ref{item-dimred-compatible-alpha}), by construction, it is enough to deal with the case when $E' = E \times_{Y} Y'$ and $Y = Y'$. The former case follows from \cref{prop-property-FS-dimred} (\ref{item-FS-dimred-base-smooth}). We now prove the latter case, where the diagram \eqref{eq-cotangent-square-diagram-pullback} is identified with the following:
        \[\begin{tikzcd}
        	& {E^{\vee}} \\
        	{E'^{\vee}} & {E^{\vee} \times_Y F} & {E^{\vee}} \\
        	Y && Y \\
        	{E'} && E
        	\arrow["{f^{\vee}}"', from=1-2, to=2-1]
        	\arrow["h"', from=1-2, to=2-2]
        	\arrow[equals, from=1-2, to=2-3]
        	\arrow[from=2-1, to=3-1]
        	\arrow["{u_{E'}}"', curve={height=24pt}, from=2-1, to=4-1]
        	\arrow["{\tilde{f}}"', from=2-2, to=2-3]
        	\arrow["{u_E'}"'{pos=0.2}, from=2-2, to=4-1]
        	\arrow[from=2-3, to=3-3]
        	\arrow["{u_{E}}", curve={height=-24pt}, from=2-3, to=4-3]
        	\arrow[equals, crossing over, from=3-1, to=3-3]
        	\arrow[from=3-1, to=4-1]
        	\arrow[from=3-3, to=4-3]
        	\arrow["f"', from=4-1, to=4-3]
        \end{tikzcd}\]
        Here, $F = \fib(E' \to E)$ and $i$ is induced by the zero section of $F$. 
        We now apply the stacky version of \cite[Proposition 5.20(2)]{KKPS1} as established in \cite[Theorem 5.23]{KKPS1} for the following square
        \[\begin{tikzcd}
        	E' & E' \\
        	{E'} & E.
        	\arrow[equals, from=1-1, to=1-2]
        	\arrow[equals,  from=1-1, to=2-1]
        	\arrow["f", from=1-2, to=2-2]
        	\arrow["f"', from=2-1, to=2-2]
        \end{tikzcd}\]
        It yields the commutativity of the following diagram
        \[\begin{tikzcd}
        	{u_{E'}^{\varphi} f^{\dagger}} & {f^{\vee}_* h^{\dagger}(u_{E}')^{\varphi} f^{\dagger}} \\
        	{f^{\vee}_* u_{E}^{\varphi}} & {f^{\vee}_* h^{\dagger} \tilde{f}^{\dagger} u_{E}^{\varphi}. }
        	\arrow["{\alpha_{(h, f^{\vee})}}", from=1-1, to=1-2]
        	\arrow["{\gamma_f}"', from=1-1, to=2-1]
        	\arrow["{\alpha_{f}}", from=1-2, to=2-2]
        	\arrow["\sim"', from=2-1, to=2-2]
        \end{tikzcd}\]
        Then the claim follows from  \cref{prop-property-FS-dimred} (\ref{item-FS-dimred-fibre-smooth}).

        For (\ref{item-dimred-compatible-beta}), repeating the above discussion, we may assume $Y' = Y$. Then the statement follows from \cref{prop-property-FS-dimred} (\ref{item-FS-dimred-fibre-finite}) together with the construction of the map \eqref{eq-symp-push-perverse-pullback}.

        For (\ref{item-dimred-compatible-gamma}), using the compatibility relation of $\alpha$ and $\gamma$, it is enough to prove the statement when $Y = Y'$ and $E' = E|_{Y'}$.
        In the former case, the statement follows from \cref{prop-property-FS-dimred} (\ref{item-FS-dimred-fibre-smooth}). In the latter case, the diagram \eqref{eq-cotangent-square-diagram-pushforward} is identified with the following diagram:
        \[\begin{tikzcd}
        	& {E^{\vee} |_{Y'} \times_{Y'} \T^*[-1](Y' / Y)} & {E^{\vee} |_{Y'}} \\
        	{E^{\vee} |_{Y'}} && {E^{\vee}} \\
        	{Y'} && Y \\
        	{E|_{Y'}} && {E.}
            \arrow["{u_{E}'}", from=1-2, to=4-3]
        	\arrow["k"', from=1-2, to=2-1]
        	\arrow["i"', from=1-3, to=1-2]
        	\arrow["{g_{E^{\vee}}}", from=1-3, to=2-3]
        	\arrow[crossing over, from=1-3, to=3-1]
        	\arrow[from=2-1, to=3-1]
        	\arrow["{u_{E'}}"', curve={height=24pt}, from=2-1, to=4-1]
        	\arrow[from=2-3, to=3-3]
        	\arrow["{u_E}", curve={height=-24pt}, from=2-3, to=4-3]
        	\arrow[crossing over, "g", from=3-1, to=3-3]
        	\arrow[from=3-1, to=4-1]
        	\arrow[from=3-3, to=4-3]
        	\arrow["{g_E}"', from=4-1, to=4-3]
        \end{tikzcd}\]
        We now apply the stacky version of \cite[Proposition 5.20(2)]{KKPS1} as established in \cite[Theorem 5.23]{KKPS1} for the following square
        \[\begin{tikzcd}
        	E |_{Y'} & E \\
        	{E} & E
        	\arrow["g_E", from=1-1, to=1-2]
        	\arrow["g_E"', from=1-1, to=2-1]
        	\arrow[equals,  from=1-2, to=2-2]
        	\arrow[equals,  from=2-1, to=2-2]
        \end{tikzcd}\]
        which yields the commutativity of the following diagram
        \[\begin{tikzcd}
        	{u_{E'}^{\varphi} g_E^{\dagger}} & {k_* i_*  u_{E'}^{\varphi} g_{E}^{\dagger}} \\
        	{k_* u_{E}'^{\varphi}} & {k_* i_* g_{E^{\vee}}^{\dagger} u_{E}^{\varphi}.}
        	\arrow["\sim", from=1-1, to=1-2]
        	\arrow["{\gamma_{g_E}}"', from=1-1, to=2-1]
        	\arrow["{\alpha_{g_E}}", from=1-2, to=2-2]
        	\arrow["{\alpha_{(g_E^{\vee}, i)}}"', from=2-1, to=2-2]
        \end{tikzcd}\]
        Then the claim follows from  \cref{prop-property-FS-dimred} (\ref{item-FS-dimred-base-smooth}).

        For (\ref{item-dimred-compatible-epsilon}), repeating the above discussion, we may assume $E' = E|_{Y'}$. Then the statement follows from \cref{prop-property-FS-dimred} (\ref{item-FS-dimred-base-finite}) together with the construction of the map \eqref{eq-symp-push-perverse-pullback}.

        For (\ref{item-dimred-compatible-tau}), repeating the above argument, it is enough to prove the statement when $g_i \colon Y_i \to B_i$ is given by the zero section of a vector bundle $0_{E_1} \colon Y_i \to E_i$. In this case, the statement is proved in \cref{prop-property-FS-dimred} (\ref{item-FS-dimred-product}).
\end{proof}

Next, we show that the general construction of the dimensional reduction isomorphism coincides with the one previously constructed in \cref{cor-dim-red-cl-imm} in the special case of closed immersions.

\begin{proposition}\label{prop-dimred-cl-imm-comparison}
    In the notation of \cref{cor-dim-red-cl-imm}, the isomorphism $\tau_* \pi^{\varphi} \cong g^! [- \dim (Y / B)]$ in \eqref{eq-dim-red-cl-imm} coincides with \eqref{eq-dimrensional-reudction}.
\end{proposition}

\begin{proof}
    By repeating the argument in the first paragraph of the proof of \cref{prop-dimensional-reduction-compatible}, we may assume that $g$ is a zero section for the total space of the perfect complex of Tor-amplitude $\leq 0$. In this case, the statement is tautological, since we constructed the isomorphism \eqref{eq-dim-red-climm-tor>0} using \cref{cor-dim-red-cl-imm}.
\end{proof}

We now compare the general construction of the dimensional reduction isomorphism and the one for local models given in  \cref{cor:dim red for crit}.

\begin{corollary}\label{cor-dimred-local-model-comparison}
        In the notation of \cref{cor:dim red for crit}, the isomorphism $\tau_* \pi^{\varphi} \cong g^! [- \dim (Y / B)]$ in \eqref{eq-dim-red-local-model} coincides with \eqref{eq-dimrensional-reudction}. 
\end{corollary}

\begin{proof}
    Consider the following commutative diagram:
    \[\begin{tikzcd}
    	{E^{\vee}|_{Z(s)}} & {\T^*[-1](Z(s) / B)} \\
    	{Z(s)} & {Z(s)} \\
    	Y & {B.}
    	\arrow["{\tau_{Z(s)/Y}}", from=1-1, to=2-1]
    	\arrow["{\pi_{Z(s) /Y}}"', curve={height=24pt}, from=1-1, to=3-1]
    	\arrow["r"', from=1-2, to=1-1]
    	\arrow["{\tau_{Z(s)/B}}"', from=1-2, to=2-2]
    	\arrow["{\pi_{Z(s) /B}}", curve={height=-24pt}, from=1-2, to=3-2]
    	\arrow["\iota", from=2-1, to=3-1]
    	\arrow[equals, from=2-2, to=2-1]
    	\arrow["{g\circ \iota}"', from=2-2, to=3-2]
    	\arrow["g", from=3-1, to=3-2]
    \end{tikzcd}\]
    Then \cref{prop-dimensional-reduction-compatible} (\ref{item-dimred-compatible-gamma}) implies the commutativity of the following diagram:
    \[\begin{tikzcd}
    	{\tau_{Z(s ) / Y, *}  \pi_{Z(s) /Y}^{\varphi} g^{\dagger} } & {\tau_{Z(s ) / Y, *} r_* \pi_{Z(s) /B}^{\varphi} } & {\tau_{Z(s ) / B, *} \pi_{Z(s) /B}^{\varphi} } \\
    	{\iota^! g^![- \dim(Z(s) / B)]} && {(g \circ \iota)^![- \dim(Z(s) / B)].}
    	\arrow["{\gamma_g}", from=1-1, to=1-2]
    	\arrow["{\red_{Z(s)/Y}}"', from=1-1, to=2-1]
    	\arrow["\sim", from=1-2, to=1-3]
    	\arrow["{\red_{Z(s)/B}}", from=1-3, to=2-3]
    	\arrow["\sim"', from=2-1, to=2-3]
    \end{tikzcd}\]
    Therefore it is enough to prove the analogous statement for the isomorphism \eqref{eq-dim-red-local-model}. However, this is obvious from the construction.
\end{proof}

We expect that \cref{thm:dimensionalreduction} holds for any coefficient ring, if we restrict to the perverse category on the source. As a consequence of \cref{cor-dimred-local-model-comparison}, we prove the following weaker version of this statement. Namely, for an lft higher Artin stack $B$ and a commutative ring $R$ (which is no longer assumed to be a field), let $\Dbc(B; R)$ be the $\infty$-category of constructible complexes of sheaves of $R$-modules on $B$ and $R_B\in\Dbc(B; R)$ the corresponding constant sheaf.

\begin{corollary}
    We adopt the notation from \cref{thm:dimensionalreduction}. Assume that $B$ is smooth, so that the dualizing complex $\bbD \Z_{B}[- \dim B]\in\Dbc(B; \Z)$ is perverse.
    Then there exists a natural isomorphism
    \[
    \tau_* \pi^{\varphi} \bbD \Z_{B} \cong  \bbD \Z_{Y}[- \dim(Y/B)]
    \]
    in $\Dbc(Y; \Z)$ which recovers \eqref{eq-dimrensional-reudction} after applying the rationalization functor $\Dbc(Y; \Z)\rightarrow \Dbc(Y; \Q)$.
\end{corollary}

\begin{proof}
    It is enough to prove the isomorphism $\tau_* \pi^{\varphi}(\bbD \Q_{B}) \cong  \bbD \Q_{Y}[- \dim(Y/B)]$ lifts to an integral isomorphism after pulling back to a smooth cover.
    Indeed, this statement implies that the Hom complex $\mathcal{L} \coloneqq\mathcal{H}\mathrm{om}(\tau_* \pi^{\varphi} \bbD \Z_{B}, \bbD \Z_{Y}[- \dim(Y/B)])$ is a local system of rank one and $\mathcal{L}_{\Q}$ admits a trivialization which admits an integral lift over a smooth cover. This implies that the trivialization $\mathcal{L}_{\Q}$ lifts to an integral isomorphism as desired.

    Now consider the following commutative square
    \[\begin{tikzcd}
	{Y'} & Y \\
	{B'} & B
	\arrow["{p_{Y}}", from=1-1, to=1-2]
	\arrow["{g'}"', from=1-1, to=2-1]
	\arrow["g", from=1-2, to=2-2]
	\arrow["{p_{B}}"', from=2-1, to=2-2]
    \end{tikzcd}\]
    such that $p_B$ and $p_Y$ are smooth and surjective map from schemes and $g'$ is a closed immersion.
    Applying \cref{prop-dimensional-reduction-compatible} (\ref{item-dimred-compatible-gamma}), we conclude that we may assume that $g$ is a closed immersion between schemes.

    Now, possibly shrinking $Y$ and $B$, we may choose the following commutative diagram:
    \[\begin{tikzcd}
    	Y & {\tilde{Y}} \\
    	B & B
    	\arrow["\iota", from=1-1, to=1-2]
    	\arrow["g"', from=1-1, to=2-1]
    	\arrow["{\tilde{g}}", from=1-2, to=2-2]
    	\arrow[equals, from=2-1, to=2-2]
    \end{tikzcd}\]
    where $\iota$ induces an isomorphism on classical truncations and $\mathbb{L}_{\tilde{Y} / B}$ has Tor-amplitude $\leq -1$.
    Using \cref{prop-dimensional-reduction-compatible} (\ref{item-dimred-compatible-epsilon}), we may assume that $\mathbb{L}_{\tilde{Y} / B}$ has Tor-amplitude $\leq -1$.
    After shrinking $Y$ and $B$ if necessary, we may assume that there exists a vector bundle $E$ on $B$ and section $s \in \Gamma(B, E)$ such that $Y = Z(s)$. 
    In this case, we have shown in \cref{cor-dimred-local-model-comparison} that the dimensional reduction isomorphism in \eqref{eq-dimrensional-reudction} coincides with \eqref{eq-dim-red-local-model}. On the other hand, since \eqref{eq-dim-red-local-model} holds integrally, we obtain the desired statement.
\end{proof}

\begin{remark}
    By comparing the perverse pullback for $(-1)$-shifted cotangent stacks with the derived microlocalization developed by Khan and Kinjo \cite{KhanKinjo}, and by gluing the isomorphisms for local models constructed by Kinjo \cite[Theorem 4.1]{KinjoVirtual} and independently by Schefers \cite[Theorem 5.1]{Schefers}, we expect that it is possible to lift \cref{thm:dimensionalreduction} to an isomorphism of functors on $\Perv(B; R)$ for any coefficient ring $R$.
    We do not pursue this further in the present paper.
\end{remark}



\subsection{Application: the Fourier--Kashiwara isomorphism}

Let $p \colon B \to S$ be an lfp morphism between higher Artin stacks and $E$ be a perfect complex over $B$.
Here we prove the following statement, which we refer to as the Fourier--Kashiwara isomorphism following \cite{mirkovic2010linear} (which describes its categorification in terms of linear Koszul duality):

\begin{theorem}\label{thm-Kashiwara-Fourier}
    Let $\pr_E \colon E \to B$ and $\pr_{E^{\vee}[-1]} \colon E^{\vee}[-1] \to B$ be the projections. Assume that the coefficient ring $R$ is a field.
    Then there exists a natural equivalence of functors
    \[
    \pr_{E, *} (p \circ \pr_E)^! \simeq \pr_{E^{\vee}[-1], *} (p \circ \pr_{E^{\vee}[-1]} )^! [2\rank(E)]
    \]
    from $\Dbc(S; R)$ to $\Dbc(B; R)$. In particular, there is a natural isomorphism between the relative Borel--Moore homology groups
    \begin{equation}
    \label{eq-Kashiwara-Fourier-absolute}
    \HBM_\bullet(E / S) \cong \HBM_{\bullet - 2\rank(E)}(E^{\vee}[-1] / S).
    \end{equation}
\end{theorem}

\begin{proof}
Applying \cref{lem-crit-locus-vs-shifted-cotangent} for the zero section $0 \in \Gamma(B, E[1])$, we obtain an isomorphism of oriented relative exact $(-1)$-shifted symplectic stacks
\[
 \T^*[-1](E / B) \cong  \T^*[-1](E^{\vee}[-1] / B).
\]
Then using \cref{thm:dimensionalreduction} twice, we obtain the desired isomorphism. 
\end{proof}

\begin{remark}
A homogeneous version of the Fourier--Kashiwara isomorphism was also previously established in \cite[Corollary 1.33]{KhanFourier}.
That version follows from \cref{thm-Kashiwara-Fourier} applied to a perfect complex on $B \times \mathrm{B} \mathbb{G}_{\mathrm{m}}$ for which $\mathbb{G}_{\mathrm{m}}$ acts by weight one.
\end{remark}

\begin{example}
When $E$ is a vector bundle, the isomorphism \eqref{eq-Kashiwara-Fourier-absolute} is nothing but the Thom isomorphism $\HBM_\bullet(E/S)\cong \HBM_{\bullet-2\rank(E)}(B/S)$. Therefore, we can regard the above theorem as a generalization of the Thom isomorphism theorem to singular vector bundles.
\end{example}

\begin{example}
Let $G$ be an algebraic group acting on a smooth variety $X$. Then we get an isomorphism
\[\HBM_\bullet(\T^*[X/G])\cong \HBM_{\bullet-2(\dim(X)-\dim(G))}(\T[-1][X/G]).\]
Here $\T^*[X/G]\cong [\mu^{-1}(0)/G]$, where $\mu\colon \T^*X\rightarrow \g^*$ is the moment map, and $\T[-1][X/G]\cong [\{(a,x)\in \g\times X\mid \rho(a)_x =0\} / G]$, where $\rho\colon \g\rightarrow \T_X$ is the action map. This is the case considered in \cite{KashiwaraCharacter}.
\end{example}


We provide an application of \cref{thm-Kashiwara-Fourier} to the study of the moduli space of Higgs bundles. Let $\Sigma$ be a smooth projective curve and $G$ a reductive algebraic group. For a line bundle $\mathcal{L}$,  we let $\mathrm{Higgs}_G^{\mathcal{L}}(\Sigma)$ be the moduli stack of pairs $(P, \phi)$ of a principal $G$-bundle $P$ and a section $\phi \in \Gamma(\Sigma, (P \times^{G} \mathfrak{g}) \otimes \mathcal{L})$ where  $\mathfrak{g}$ denotes the Lie algebra associated with $G$. For instance, for $K_\Sigma$ the canonical bundle $\mathrm{Higgs}_G^{K_\Sigma}(\Sigma)$ is the usual moduli stack of Higgs bundles. Then we have the following:

\begin{corollary}
    Set $n \coloneqq (g(\Sigma)-1 ) \cdot \dim(G)$. There is a natural isomorphism
    \begin{equation}\label{eq-Higgs-Fourier}
    \HBM_\bullet(\mathrm{Higgs}_G^{\mathcal{O}_\Sigma}(\Sigma)) \cong \HBM_{\bullet + 2n}(\mathrm{Higgs}_G^{K_\Sigma}(\Sigma)).
    \end{equation}
\end{corollary}

\begin{proof}
    Let $\Bun_G(\Sigma)$ be the moduli stack of $G$-bundles on $\Sigma$.
    Then we have
    \[
    \mathrm{Higgs}_G^{\mathcal{O}_\Sigma}(\Sigma) \cong \T[-1]\Bun_G(\Sigma), \quad \mathrm{Higgs}_G^{K_\Sigma}(\Sigma) \cong \T^*\Bun_G(\Sigma)
    \]
    and $\dim(\Bun_G(\Sigma)) = n$. Hence the desired isomorphism follows from \cref{thm-Kashiwara-Fourier}.
\end{proof}

\begin{remark}
    The Borel--Moore homology group $\HBM_\bullet(\mathrm{Higgs}_G^{K_\Sigma}(\Sigma))$ has been intensively studied in the literature, especially in its connection with the topological mirror symmetry and the P=W conjecture (see \cite{davison2023nonabelian, davison2024purity, bu2025cohomology}). We expect that the isomorphism \eqref{eq-Higgs-Fourier} together with the Hitchin fibration $\mathrm{Higgs}_G^{\mathcal{\cO}_\Sigma}(\Sigma) \to \mathfrak{g} /\!/ G$ provides an interesting algebraic structure on $\HBM_\bullet(\mathrm{Higgs}_G^{K_\Sigma}(\Sigma))$. We note that a decategorified version of \eqref{eq-Higgs-Fourier} has been already observed by Mozgovoy and Schiffmann \cite[Proposition 4.1]{mozgovoy2014counting}.
\end{remark}

\section{Period sheaves for Hamiltonian spaces}

In this section we define period sheaves for general Hamiltonian spaces as envisaged in \cite{BZSV}. In this section $k=\C$ and $R$ is an algebraically closed field of characteristic zero.

\subsection{Setup}\label{sect:periodsetup}

Throughout this section we consider the following objects over $k$:
\begin{itemize}
    \item $G_1, G_2$ are reductive groups. We denote $G=G_1\times G_2$.
    \item $(M, \omega)$ is a quasi-separated smooth symplectic scheme equipped with a $G\times \Gm$-action, such that the $G$-action preserves $\omega$ and the $\Gm$-action scales it with weight $2$.
    \item $\Sigma$ is a connected smooth projective algebraic curve over $k$ equipped with a choice of $K_\Sigma^{1/2}$.
\end{itemize}

The assumptions on $M$ imply the following:
\begin{itemize}
    \item The symplectic form $\omega$ is canonically exact. Namely, $\omega = d\alpha$, where $\alpha = \iota_v \omega$ with $v$ the vector field generating the $\Gm$-action.
    \item The symplectic volume form $\vol_M = \frac{\omega^{\dim M/2}}{(\dim M/2)!}$ is $G$-equivariant and has weight $\dim M$ under the $\Gm$-action. Thus, the symplectic volume form provides a trivialization
    \begin{equation}\label{eq:symplecticvolumeform}
    \vol_M\colon \cO_{[M/(G\times\Gm)]}\xrightarrow{\sim} K_M\otimes \pi^*_{\B\Gm}\cO(-\dim M),
    \end{equation}
    where $\pi_{\B\Gm}\colon [M/(G\times \Gm)]\rightarrow \B\Gm$ and $\cO(n)$ is the line bundle on $\B\Gm$ corresponding to the one-dimensional $\Gm$-representation of weight $n$.
\end{itemize}

Let
\[\Bun_G(\Sigma) = \Map(\Sigma, \B G)\]
be the moduli stack of $G$-bundles, which is a smooth 1-Artin stack. Let $\Bun^M_G(\Sigma)$ be the derived stack given by the fiber product
\begin{equation}\label{eq:BunGM}
\xymatrix{
\Bun^M_G(\Sigma) \ar[r] \ar^{\pi}[d] & \Map(\Sigma, [M/(G\times\Gm)]) \ar[d] \\
\Bun_G(\Sigma) \ar^{\id\times K^{-1/2}_\Sigma}[r] & \Bun_{G\times\Gm}(\Sigma).
}
\end{equation}

\begin{example}
Let $\Sigma=E$ be an elliptic curve. Then $\Bun_G^M(E)\cong \Map(E, [M/G])$.
\end{example}

By \cite[Theorem 1.2]{HallRydh} the classical truncation of $\Bun_G^M(\Sigma)$ is a 1-Artin stack locally of finite type. An important observation that will allow us to apply the formalism of perverse pullbacks is that $\Bun^M_G(\Sigma)$ carries a relative exact $(-1)$-shifted symplectic structure.

\begin{proposition}\label{prop:BunGMsymplectic}
The projection $\pi\colon\Bun^M_G(\Sigma)\rightarrow \Bun_G(\Sigma)$ carries a natural relative exact $(-1)$-shifted symplectic structure.
\end{proposition}
\begin{proof}
Consider the diagram
\[
\xymatrix{
\ev^*[M/(G\times\Gm)] \ar[d] \ar[r] & [M/(G\times\Gm)] \ar[d] \\
\Bun_G(\Sigma)\times \Sigma \ar^{\ev}[r] \ar^{\pi_1}[d] & \B(G\times\Gm) \\
\Bun_G(\Sigma),
}
\]
where the square is Cartesian and $\ev\colon \Bun_G(\Sigma)\times\Sigma\rightarrow\B(G\times\Gm)$ is given by the composite
\[\Bun_G(\Sigma)\times\Sigma\xrightarrow{\id\times K^{-1/2}_\Sigma}\Bun_{G\times\Gm}(\Sigma)\times \Sigma\longrightarrow \B(G\times\Gm),\]
where the last morphism is the natural evaluation. By assumption on $M$ the derived stack $[M/(G\times \Gm)]$ over $\B(G\times\Gm)$ carries an exact $\cO(-2)$-twisted symplectic structure. Therefore, the base change $\ev^*[M/(G\times\Gm)]\rightarrow \Bun_G(\Sigma)\times\Sigma$ carries an exact $K_\Sigma$-twisted symplectic structure. By \cite[Proposition 2.20]{CalaqueSafronov} $\Sigma$ carries a $K_\Sigma[1]$-orientation. Therefore, using the AKSZ construction as in \cite[Theorem 2.34]{CalaqueSafronov}, we get that the Weil restriction $\Res_{\Bun_G(\Sigma)\times \Sigma/\Bun_G(\Sigma)}(\ev^*[M/(G\times\Gm)])\cong \Bun^M_G(\Sigma)$ carries a relative exact $(-1)$-shifted symplectic structure.
\end{proof}

\begin{remark}\label{rmk:GaiottoLagrangian}
The moment map associated to the relative exact $(-1)$-shifted symplectic structure on $\Bun_G^M(\Sigma)\rightarrow \Bun_G(\Sigma)$ provides an exact $0$-shifted Lagrangian structure on $\Bun_G^M(\Sigma)\rightarrow \T^* \Bun_G(\Sigma)$, where on the right we consider the moduli stack of $G$-Higgs bundles. This is exactly the $0$-shifted Lagrangian morphism constructed in \cite[Theorem 1.3]{GinzburgRozenblyum}.
\end{remark}

\subsection{Orientations and anomaly}

Our next goal is to describe orientations of the relative exact $(-1)$-shifted symplectic structure on $\Bun_G^M(\Sigma)$. We begin with a computation of the relative cotangent complex which shows, in particular, that $\Bun_G^M(\Sigma)$ is quasi-smooth.

\begin{lemma}\label{lm:BunGMcotangent}
For an affine scheme $S$ together with a morphism $f\colon S\rightarrow \Bun_G^M(\Sigma)$ represented by $\overline{f}\colon S\times \Sigma\rightarrow [M/(G\times\Gm)]$ we have an isomorphism
\[f^*\bL_{\Bun_G^M(\Sigma)/\Bun_G(\Sigma)}\cong p_*((\cO_S\boxtimes K_\Sigma)\otimes \overline{f}^*\Omega^1_M)[1]\]
functorially in $S$, where $p\colon S\times \Sigma\rightarrow S$ is the projection.
\end{lemma}
\begin{proof}
Since the diagram \eqref{eq:BunGM} is Cartesian, it is enough to compute the relative cotangent complex of $\Map(\Sigma, [M/(G\times\Gm)])\rightarrow \Bun_{G\times\Gm}(\Sigma)$. By Serre duality the functor $p^*\colon \QCoh(S)\rightarrow \QCoh(S\times \Sigma)$ admits a left adjoint $p_\sharp(-) = p_*((\cO_S\boxtimes K_\Sigma)\otimes (-))[1]$. The claim then follows from \cite[Proposition B.3.5]{Rozenblyum}.
\end{proof}

The description of orientations will be given in terms of motivic cohomology which we briefly recall. For a stack $X$ let $\bK(X)$ be the nonconnective $K$-theory spectrum of $\Perf(X)$ and let $\rK(X)$ be its underlying space, so that for a perfect complex $V\in\Perf(X)$ we have the associated class $[V]\in \rK(X)$ and the determinant map factors as
\[\det\colon \Perf(X)^{\sim}\longrightarrow \rK(X)\longrightarrow \Pic(X)\]
naturally in $X$.

Let $\SH(k)$ be the stable motivic $\infty$-category; we denote by $\bMap_{\SH(k)}(-, -)$ the mapping spectrum in $\SH(k)$. $K$-theory is representable on smooth schemes, i.e. there is a motivic spectrum $\bKGL\in\SH(k)$ such that
\[\bK(X)\cong \bMap_{\SH(k)}(\Sigma^\infty_T X_+, \bKGL)\]
naturally in smooth schemes $X$.
Moreover, there is a motivic filtration on $\bK(X)$ for $X$ a smooth scheme which can be realized as follows. Let $f_j \bKGL\in\SH(k)$ be the slice filtration \cite{Voevodsky,Levine} with associated graded $s_j \bKGL\cong \bMZ(j)[2j]$, where $\bMZ$ is the motivic Eilenberg--MacLane spectrum. Then
\[\Fil^\bullet \bK(X) = \bMap_{\SH(k)}(\Sigma^\infty_T X_+, f_\bullet \bKGL)\]
with associated graded
\[\gr^j \bK(X)\cong \R\Gamma(X, \Z(j))[2j],\]
the Zariski cohomology of the motivic complexes of weight $j$ as in \cite{MVW}. The motivic filtration on $\bK(-)$, as well as the motivic complexes $\Z(j)$, can be extended to general Noetherian schemes in two equivalent ways: either by left Kan extending from smooth schemes and then applying pro-cdh sheafification \cite{KellySaito} or by using trace methods \cite{ElmantoMorrow}.

Taking motivic cohomology mod 2 we get morphisms
\[\R\Gamma(X, \Z(j))\longrightarrow \R\Gamma_{\et}(X, \mu_2^{\otimes j}),\]
where on the right we consider \'etale cohomology. For a vector bundle $E\rightarrow X$ we denote by
\[c_j(E)\in |\R\Gamma_{\et}(X, \mu_2^{\otimes j})[2j]|\]
its \'etale Chern classes (see \cite[Expos\'e~VII, \S 3]{SGA5} or \cite{Soule}). In case $E$ is a $G$-equivariant vector bundle, we denote by
\[c^G_j(E) \in |\R\Gamma_{\et}([X/G], \mu_2^{\otimes j})[2j]|\]
its equivariant Chern classes. The obvious projection
\[\ch_n\colon \Fil^n\rK(X)\longrightarrow |\R\Gamma(X, \Z(n))[2n]|,\]
has the following explicit description in low weights:
\begin{enumerate}
\setcounter{enumi}{-1}
    \item $\R\Gamma(X, \Z(0))\cong \R\Gamma(X, \Z)$ is the Zariski cohomology of the constant sheaf. For $V\in\Perf(X)$ we have $\ch_0([V]) = \rank(V)$, the rank of $V$.
    \item $\R\Gamma(X, \Z(1))[2]\cong \Pic(X)$. For $V\in\Perf(X)$ we have $\ch_1([V]-\rank(V)[\cO]) = \det(V)$. Its image in $\R\Gamma_{\et}(X, \mu_2)[2]$ coincides with $c_1(V)$; this is the class of the gerbe parametrizing square roots of $\det(V)$.
    \item For $V\in\Perf(X)$ with $\det(V)\cong \cO$ we have $[V]-\rank(V)[\cO]\in\Fil^2\rK(X)$. In this case the image of $\ch_2([V]-\rank(V)[\cO])$ in $\R\Gamma_{\et}(X, \mu_2^{\otimes 2})[4]$ coincides with $-c_2(V)$.
\end{enumerate}

We have the following functoriality properties of the motivic filtration.

\begin{proposition}\label{prop:motfiltrationfunctoriality}
Let $X$ be a scheme of finite type.
\begin{enumerate}
    \item Let $E$ be a vector bundle of rank $r+1$ over $X$. Let $\pi\colon \bP(E)\rightarrow X$ be the projection and $\cO(-1)$ the tautological line bundle on $\bP(E)$. Then for every $n$ the morphism
    \[\sum_{i=0}^r ([\cO]-[\cO(-1)])^i \pi^*\colon \bigoplus_{i=0}^r \Fil^{n-i} \bK(X)\longrightarrow \Fil^n\bK(\bP(E))\]
    is an isomorphism.
    \item In the notation of the previous part, $\pi_*\colon \bK(\bP(E))\rightarrow \bK(X)$ refines to a morphism
    \[\pi_*\colon \Fil^n\bK(\bP(E))\longrightarrow \Fil^{n-r}\bK(X)\]
    functorially in $X$.
    \item Let $i\colon Z\rightarrow X$ be a regular closed immersion of codimension $d$. Then the pushforward $i_*\colon \bK(Z)/2\rightarrow \bK(X)/2$ refines to a morphism
    \[i_*\colon \Fil^n\bK(Z)/2\longrightarrow \Fil^{n+d}\bK(X)/2\]
    functorially in $X$ with respect to base change.
\end{enumerate}
\end{proposition}
\begin{proof}$ $
\begin{enumerate}
    \item As the motivic filtration is bounded \cite[Theorem 4.12]{ElmantoMorrow}, it is enough to show that the morphism on the associated graded level is an isomorphism. Passing to the associated graded we get the composite
    \begin{align*}
    \bigoplus_{i=0}^r \R\Gamma(X, \Z(n-i))[2n-2i]&\xrightarrow{\pi^*\otimes [\cO(1)]^i} \bigoplus_{i=0}^r \R\Gamma(\bP(E), \Z(n-i))[2n-2i]\otimes (\R\Gamma(\bP(E), \Z(1))[2])^{\otimes i}\\
    &\longrightarrow \R\Gamma(\bP(E), \Z(n))[2n]
    \end{align*}
    which is an isomorphism by the projective bundle formula \cite[Theorem 5.24(2)]{ElmantoMorrow}.
    \item We use the description of the filtered pieces $\Fil^n\bK(\bP(E))$ from the previous part. Using the projection formula as well as the vanishing of cohomology of $\cO(-i)$ for $1\leq i\leq r$ on $\bP(E)$ we get that the composite
    \[\bigoplus_{i=0}^r \Fil^{n-i}\bK(X)\xrightarrow{\sim} \Fil^n\bK(\bP(E))\xrightarrow{\pi_*} \bK(X)\]
    is given by the obvious inclusions $\Fil^{n-i}\bK(X)\rightarrow \bK(X)$ for each $i$. In particular, $\pi_*\colon \Fil^n\bK(\bP(E))\rightarrow \bK(X)$ factors through $\Fil^{n-r}\bK(X)$.
    \item Using the localization fiber sequence $\bK(X~\textrm{on}~Z) \to \bK(X) \to \bK(X\setminus Z)$ we see that the pushforward $i_*\colon \bK(Z)\rightarrow \bK(X)$ lifts to $i_*\colon \bK(Z)\rightarrow \bK(X~\textrm{on}~Z)$. Our goal is to show that mod 2 we can lift this morphism to a morphism
    \[i_*\colon \Fil^n\bK(Z)/2\longrightarrow \Fil^{n+d}\bK(X~\textrm{on}~Z)/2\]
    functorially in $X$. Using deformation to the normal cone as well as the $\bA^1$-homotopy invariance of $\bK(-)/2$ (see \cite[Proposition 1.6]{Weibel} where we recall our standing assumption that the ground field $k$ has characteristic not $2$) we are reduced to showing the claim for $i$ the zero section $Z\rightarrow E$ of a vector bundle $\pi\colon E\rightarrow Z$ of rank $d$. That is, we have to show that the morphism $i_*\colon \bK(Z)/2\rightarrow \bK(E~\textrm{on}~Z)/2$ lifts to a morphism
    \[\Fil^n\bK(Z)/2\longrightarrow \Fil^{n+d} \bK(E~\textrm{on}~Z)/2\]
    functorially in $Z$. The projection formula shows that it is compatible with the $\bK(Z)/2$-module structure on both sides (where on the right the action is given via $\pi^*$). Thus, we are reduced to showing that the Thom class $\Th(E)\in \rK(E~\textrm{on}~Z)/2$ (i.e. the class of $i_*\cO_Z$) lifts canonically along
    \[\Fil^d \bK(E~\textrm{on}~Z)/2\rightarrow \bK(E~\textrm{on}~Z)/2.\]
    We claim that this map is, in fact, an isomorphism. The claim is local, so we may assume $E$ is trivial. In this case the cofiber has a finite filtration with associated graded pieces $\R\Gamma(Z, \Z(i-d)/2)[2i-2d]$ for $i < d$. But these motivic cohomology complexes have negative weights, so they are zero.
\end{enumerate}
\end{proof}

\begin{remark}
    \cref{prop:motfiltrationfunctoriality} holds for derived schemes, and (3) holds for quasi-smooth closed immersions, with the same proofs.
\end{remark}

The following statement is a homotopy-coherent version of the Grothendieck--Riemann--Roch formula for curves in the style of Deligne.

\begin{corollary}\label{cor:GRR}
Let $S$ be an affine scheme of finite type and let $p\colon S\times \Sigma\rightarrow S$ be the projection. Let $L$ be a line bundle on $\Sigma$. Then there is a commutative diagram
\[
\xymatrix@C=2cm{
\Fil^2 \bK(S\times \Sigma)/2\ar^-{p_*((\cO_S\boxtimes L)\otimes(-))}[r] \ar^{-c_2}[d] & \bK(S)/2 \ar^{c_1}[d] \\
\R\Gamma(S\times \Sigma, \mu_2^{\otimes 2})[4] \ar^-{\int_\Sigma}[r] & \R\Gamma(S, \mu_2)[2],
}
\]
functorial in $S$.
\end{corollary}
\begin{proof}
Let us first prove the claim for $L=\cO_\Sigma$. Fix a closed immersion $i\colon \Sigma\hookrightarrow \bP^n$ and let $\overline{p}\colon \bP^n\rightarrow \pt$. By \cref{prop:motfiltrationfunctoriality} we have commutative diagrams
\[
\xymatrix{
\Fil^2 \bK(S\times \Sigma)/2 \ar^-{(\id\times i)_*}[r] \ar^{\ch_2}[d] & \Fil^{n+1} \bK(S\times \bP^n)/2 \ar^-{(\id\times\overline{p})_*}[r] \ar^{\ch_{n+1}}[d] & \Fil^1 \bK(S)/2 \ar^{\ch_1}[d] \\
\gr^2 \bK(S\times \Sigma)/2 \ar^-{(\id\times i)_*}[r] & \gr^{n+1} \bK(S\times \bP^n)/2 \ar^-{(\id\times\overline{p})_*}[r] & \gr^1 \bK(S)/2
}
\]
Identifying the bottom terms with motivic cohomology groups and passing to underlying spaces, we get the result for $L=\cO_\Sigma$.

Consider the composite
\[\bK(\Sigma)\otimes \Fil^2\bK(S\times \Sigma)\longrightarrow \Fil^2\bK(S\times \Sigma)\longrightarrow \gr^2\bK(S\times \Sigma),\]
where the first map is induced by the tensor product $V, W\mapsto (\cO_S\boxtimes V)\otimes W$. Using the multiplicativity of the motivic filtration we get that this morphism factors through
\[\gr^0\bK(\Sigma)\otimes \gr^2\bK(S\times \Sigma)\longrightarrow \gr^2\bK(S\times \Sigma).\]
Since the images of $[L]$ and $[\cO_\Sigma]$ are canonically homotopic in $\gr^0\rK(\Sigma)$, the commutative diagram for an arbitrary line bundle $L$ is induced by the commutative diagram for $\cO_\Sigma$.
\end{proof}

To have a hands-on description of orientations, we will need a description of the cohomology of quotient stacks by $\Gm$.

\begin{lemma}\label{lm:Gmcohomology}
Let $X$ be a smooth stack with a $\Gm$-action and $\pi\colon X\rightarrow [X/\Gm]$ the projection. For any $n\geq 0$ there is a fiber sequence
\[\R\Gamma_{\et}([X/\Gm], \mu_2^{\otimes n})[-2]\xrightarrow{(-)\cdot c_1(\cO(1))} \R\Gamma_{\et}([X/\Gm], \mu_2^{\otimes (n+1)})\xrightarrow{\pi^*} \R\Gamma_{\et}(X, \mu_2^{\otimes (n+1)}).\]
\end{lemma}
\begin{proof}
By smooth base change \cite[Corollary 6.2.10]{LiuZheng} applied to the Cartesian diagram
\[
\xymatrix{
X \ar[r] \ar[d] & \pt \ar[d] \\
[X/\Gm] \ar[r] & \B\Gm
}
\]
the claim is reduced to the case $X=\pt$ in which case it comes from the localization sequence associated to the closed immersion $\B\Gm\rightarrow [\bA^1/\Gm]$ and its open complement $\pt$.
\end{proof}

Consider the \'etale first Chern class $c_1(K_\Sigma^{1/2})\in|\R\Gamma_{\et}(\Sigma, \mu_2)[2]|$. We will be interested in a trivialization of its square
\[c_1(K_\Sigma^{1/2})^2\in|\R\Gamma_{\et}(\Sigma, \mu_2^{\otimes 2})[4]|.\]
Note that since $k$ is algebraically closed, $\Sigma$ has cohomological dimension $\leq 2$ \cite[Chapter VI, Theorem 1.1]{Milne}, so there is a unique trivialization of $c_1(K_\Sigma^{1/2})^2$. In fact, the space of trivializations is non-empty and connected, but generally non-contractible.

\begin{theorem}\label{thm:periodanomalyfree}
Let $(G_1\times G_2, \Sigma, M)$ be as in \cref{sect:periodsetup}. Consider the following data:
\begin{enumerate}
    \item There is a nullhomotopy of the $G\times\Gm$-equivariant \'etale second Chern class $c^{G\times\Gm}_2(\T_M\otimes\cO(1))\in |\R\Gamma_{\et}([M/(G\times \Gm)], \mu_2^{\otimes 2})[4]|$.
    \item
    \begin{enumerate}
        \item A choice of $K_\Sigma^{1/4}$ and
        \item a nullhomotopy of the $G$-equivariant \'etale second Chern class $c^G_2(\T_M)\in|\R\Gamma_{\et}([M/G], \mu_2^{\otimes 2})[4]|$.
    \end{enumerate}
    \item
    \begin{enumerate}
        \item A trivialization of $c_1(K_\Sigma^{1/2})^2$,
        \item a nullhomotopy of the $G$-equivariant \'etale second Chern class $c^G_2(\T_M)$, so that by \cref{lm:Gmcohomology} there is a canonical class $\alpha\in|\R\Gamma_{\et}([M/(G\times\Gm)], \mu_2)[2]|$ such that $c^{G\times\Gm}_2(\T_M\otimes\cO(1)) = \alpha\cdot c_1(\cO(1))$ and
        \item a trivialization of the image of $\alpha$ in $|\R\Gamma_{\et}([M/G], \mu_2)[2]|$.
    \end{enumerate} 
\end{enumerate}
Then there is a canonical orientation of the relative exact $(-1)$-shifted symplectic structure on $\Bun_G^M(\Sigma)\rightarrow \Bun_G(\Sigma)$.
\end{theorem}
\begin{proof}
As $\Bun_G^M(\Sigma)$ is locally of finite presentation, we have to construct a trivialization of the orientation gerbe $f^*\ori_{\Bun_G^M(\Sigma)/\Bun_G(\Sigma)}$ for every morphism $f\colon S\rightarrow \Bun_G^M(\Sigma)$ from a finite type affine scheme, functorially in $S$. By 
\cref{lm:BunGMcotangent} the orientation gerbe for the relative exact $(-1)$-shifted symplectic structure on $\Bun_G^M(\Sigma)\rightarrow \Bun_G(M)$ coincides with the \'etale first Chern class of the relative cotangent complex of $\Bun_G^M(\Sigma)\rightarrow \Bun_G(\Sigma)$.

If $f$ is represented by $\overline{f}\colon S\times\Sigma\rightarrow [M/(G\times\Gm)]$ and $p\colon S\times\Sigma\rightarrow S$ is the projection, by \cref{lm:BunGMcotangent} as well as using the isomorphism $p_*((\cO_S\boxtimes K_\Sigma)\otimes \overline{f}^*\Omega^1_M)\cong p_*((\cO_S\boxtimes K^{1/2}_\Sigma)\otimes \overline{f}^*(\Omega^1_M\otimes \cO(-1)))$ we have to construct a trivialization of
\[c_1(p_*((\cO_S\boxtimes K^{1/2}_\Sigma)\otimes \overline{f}^*(\Omega^1_M\otimes \cO(-1))))\in|\R\Gamma_{\et}(S, \mu_2)[2]|.\]
Using the symplectic volume form \eqref{eq:symplecticvolumeform} we have a $G\times\Gm$-equivariant trivialization of the determinant line bundle $\det(\Omega^1_M\otimes \cO(-1))$ on $M$. Thus,
\[\overline{f}^*[\Omega^1_M\otimes \cO(-1)] - \dim(M) [\cO_{S\times\Sigma}]\in \Fil^2\rK(S\times\Sigma).\]
By base change the line bundle $\det(p_*(\cO_S\otimes K^{1/2}_\Sigma))$ on $S$ is pulled back from the corresponding line bundle on $S=\Spec k$. The latter line bundle has a canonical square root given by $\det\rH^0(\Sigma, K^{1/2}_\Sigma)$. Thus, by \cref{cor:GRR} we get that
\[c_1(p_*((\cO_S\boxtimes K^{1/2}_\Sigma)\otimes \overline{f}^*(\Omega^1_M\otimes \cO(-1))))\sim -\int_\Sigma \overline{f}^* c^{G\times\Gm}_2(\Omega^1_M\otimes \cO(-1))\in|\R\Gamma_{\et}(S, \mu_2)[2]|,\]
where
\[c^{G\times\Gm}_2(\Omega^1_M\otimes \cO(-1))\sim c^{G\times\Gm}_2(\T_M\otimes\cO(1)).\]

We will now show that the choice of the data (1), (2) or (3) trivializes this expression. The case (1) is clear. In the cases (2) or (3) by \cref{lm:Gmcohomology} we have $c^{G\times\Gm}_2(\T_M\otimes\cO(1)) = \alpha\cdot c_1(\cO(1))$, so that
\[\overline{f}^* c^{G\times\Gm}_2(\Omega^1_M\otimes \cO(-1)) \sim (\overline{f}^* \alpha)\cdot c_1(\cO_S\boxtimes K_\Sigma^{1/2}).\]
\begin{enumerate}
\setcounter{enumi}{1}
    \item A choice of $K_\Sigma^{1/4}$ trivializes $c_1(K_\Sigma^{1/2})$.
    \item Given a trivialization of the image of $\alpha$ in $|\R\Gamma_{\et}([M/G], \mu_2)[2]|$, by \cref{lm:Gmcohomology} we have
    \[\overline{f}^* c^{G\times\Gm}_2(\Omega^1_M\otimes \cO(-1)) = (\overline{f}^* \beta) \cdot c_1(\cO_S\boxtimes K_\Sigma^{1/2})^2.\]
    Given a trivialization of $c_1(K_\Sigma^{1/2})^2$, this expression is trivialized.
\end{enumerate}
\end{proof}

We refer to the data as in \cref{thm:periodanomalyfree} an \defterm{anomaly trivialization} for the triple $(G_1\times G_2, \Sigma, M)$.

\begin{remark}
Denote by $\mu$ the original $\Gm$-action on $M$. Consider a $\Gm$-action $\sqrt{\mu}$ on $M$ commuting with the $G$-action whose square gives $\mu$. Given a trivialization of $c_2^G(\T_M)$, by \cref{lm:Gmcohomology} we have
\[c_2^{G\times\Gm, \sqrt{\mu}}(\T_M) = \overline{\alpha}\cdot c_1(\cO(1)),\]
where on the left we consider the equivariant second Chern class with respect to the $\Gm$-action $\sqrt{\mu}$. Thus,
\[c_2^{G\times\Gm, \mu}(\T_M) = \overline{\alpha}\cdot c_1(\cO(2)),\]
which is canonically trivial (the trivialization is given by the line bundle $\cO(1)$ which is a square root of $\cO(2)$). Therefore, if we have a square root of the $\Gm$-action on $M$, we have a trivialization of $\alpha$ relevant for condition (3c) in \cref{thm:periodanomalyfree}.
\end{remark}

\subsection{Induction functors}

For a reductive group $G$ we may identify $\T^* \Bun_G(\Sigma)$ with the moduli stack of $G$-Higgs bundles on $\Sigma$. Denote by $\Nilp\subset \T^* \Bun_G(\Sigma)$ the global nilpotent cone, i.e. the subset where the Higgs field is nilpotent. We will be interested in the full subcategory
\[\bD_{\Nilp}(\Bun_G(\Sigma))\subset \bD(\Bun_G(\Sigma))\]
of ind-constructible sheaves with nilpotent singular support. By \cite[Theorem 16.4.10 and Proposition 17.2.3]{AGKRRV1} the inclusion admits a colimit-preserving right adjoint
\[\sfP\colon \bD(\Bun_G(\Sigma))\longrightarrow \bD_{\Nilp}(\Bun_G(\Sigma)).\]
The following statement is closely related to the miraculous duality of $\Bun_G(\Sigma)$.
\begin{theorem}[{\cite[Theorem 3.2.2]{AGKRRV2}}]\label{thm:miraculousduality}
For a reductive group $G$ the functor
\[\R\Gamma_c(\Bun_G(\Sigma), -\otimes -)\colon \bD_{\Nilp}(\Bun_G(\Sigma))\otimes \bD_{\Nilp}(\Bun_G(\Sigma))\longrightarrow \Mod_R\]
is the evaluation of a self-duality of $\bD_{\Nilp}(\Bun_G(\Sigma))\in\PrSt_R$.
\end{theorem}

By \cref{thm:miraculousduality} for any $\cC\in\PrSt_R$ the data of a colimit-preserving functor $F\colon \cC\rightarrow \bD_{\Nilp}(\Bun_G(\Sigma))$ is the same as the data of a colimit-preserving functor $\tilde{F}\colon \cC\otimes\bD_{\Nilp}(\Bun_G(\Sigma))\rightarrow \Mod_R$, where the two are related by
\[\R\Gamma_c(\Bun_G(\Sigma), F(-)\otimes -)\cong \tilde{F}(-,-).\]

Consider now the triple $(G_1\times G_2, \Sigma, M)$ as in \cref{sect:periodsetup} and denote by
\[
\xymatrix{
& \Bun_G^M(\Sigma) \ar_{\pi_1}[dl] \ar^{\pi}[d] \ar^{\pi_2}[dr] & \\
\Bun_{G_1}(\Sigma) & \Bun_G(\Sigma) & \Bun_{G_2}(\Sigma).
}
\]
the natural projections.

\begin{definition}\label{def:induction}
Consider a triple $(G_1\times G_2, \Sigma, M)$ together with a choice of anomaly trivialization. The \defterm{induction functor}
\[\ind^{G_1\rightarrow G_2}_M\colon \bD_{\Nilp}(\Bun_{G_1}(\Sigma))\longrightarrow \bD_{\Nilp}(\Bun_{G_2}(\Sigma))\]
is the unique functor with a natural equivalence
\[\R\Gamma_c(\Bun_{G_2}(\Sigma), \ind^{G_1\rightarrow G_2}_M(-)\otimes -)\cong \R\Gamma_c(\Bun_G^M(\Sigma), \pi^\varphi(-\boxtimes -))[\dim \Bun_{G_2}(\Sigma)].\]
\end{definition}

\subsection{Cotangent case}\label{sect:cotangentperiods}

Let $X$ be a quasi-separated smooth scheme equipped with a $G\times\Gm$-action. Set $M=\T^* X$ with its natural Hamiltonian $G$-action. As the $\Gm$-action on $M$ we take the product of the Hamiltonian $\Gm$-action induced by the $\Gm$-action on $X$ and the weight $2$ action along the fibers of $\T^* X$. By construction we have an exact sequence
\[0\longrightarrow \Omega^1_X|_M\otimes \cO(-2)\longrightarrow \T_M\longrightarrow \T_X|_M\longrightarrow 0\]
of $G\times\Gm$-equivariant vector bundles on $M$. We get
\[c_2^{G\times\Gm}(\T_M) \sim c^{G\times\Gm}_2(\T_X|_M) + c^{G\times\Gm}_2(\Omega^1_X|_M\otimes \cO(-2)).\]
Using the square root of $\cO(-2)$ given by $\cO(-1)$ we obtain homotopies $c^{G\times\Gm}_2(\Omega^1_X|_M\otimes \cO(-2)) \sim c^{G\times\Gm}_2(\Omega^1_X|_M) \sim c^{G\times\Gm}_2(\T_X|_M)$. Thus, we get a canonical trivialization of $c_2^{G\times\Gm}(\T_M)$, i.e. an anomaly trivialization for the triple $(G_1\times G_2, \Sigma, \T^* X)$.

We define the derived stack $\Bun_G^X(\Sigma)$ analogously to $\Bun_G^M(\Sigma)$ by a Cartesian diagram
\[
\xymatrix{
\Bun_G^X(\Sigma) \ar[r] \ar[d] & \Map(\Sigma, [X/(G\times\Gm)]) \ar[d] \\
\Bun_G(\Sigma) \ar^-{\id\times K^{-1/2}_\Sigma}[r] & \Bun_{G\times\Gm}(\Sigma).
}
\]
\begin{remark}
In contrast to $\Bun_G^M(\Sigma)$, $\Bun_G^X(\Sigma)$ carries no natural relative $(-1)$-shifted symplectic structure.    
\end{remark}

Consider the correspondence
\begin{equation}\label{eq:BunXcorrespondence}
\xymatrix{
& \Bun_G^X(\Sigma) \ar_{\pi^X_1}[dl] \ar^{\pi^X_2}[dr] & \\
\Bun_{G_1}(\Sigma) && \Bun_{G_2}(\Sigma).
}
\end{equation}

\begin{proposition}\label{prop:cotangentperiod}
There is a natural isomorphism
\[\ind^{G_1\rightarrow G_2}_M(-)\cong \sfP((\pi^X_2)_!(\pi^X_1)^*(-))[\dim(\Bun_G^X(\Sigma)/\Bun_{G_1}(\Sigma))].\]
\end{proposition}
\begin{proof}
Our first goal is to establish an isomorphism $\Bun_G^M(\Sigma)\cong \T^*[-1](\Bun_G^X(\Sigma)/\Bun_G(\Sigma))$ compatibly with the relative exact $(-1)$-shifted symplectic structures over $\Bun_G(\Sigma)$. For this, recall that $\Bun_G^M(\Sigma)$ is obtained from the exact $\cO(-2)$-twisted symplectic structure on $[M/(G\times\Gm)]\rightarrow \B(G\times\Gm)$ by first pulling it back along $\ev\colon \Bun_G(\Sigma)\times\Sigma\rightarrow \B(G\times\Gm)$ and then applying a Weil restriction along $p_1\colon \Bun_G(\Sigma)\times\Sigma\rightarrow \Bun_G(\Sigma)$.

By definition $[M/(G\times\Gm)]\cong \T^*_{\cO(-2)} ([X/(G\times\Gm)]/\B(G\times\Gm))$, where we consider the twisted cotangent bundle by a line bundle as in \cite[Definition 3.1]{CalaqueSafronov}. Therefore,
\[\ev^*[M/(G\times\Gm)]\cong \T^*_{\cO_{\Bun_G(\Sigma)}\boxtimes K_\Sigma}(\ev^*[X/(G\times\Gm)]/\B(G\times\Gm)).\]
Since $p_1\colon \Bun_G(\Sigma)\times \Sigma\rightarrow \Bun_G(\Sigma)$ carries a $K_\Sigma[1]$-orientation by \cite[Proposition 2.20]{CalaqueSafronov}, by \cite[Proposition 3.9]{CalaqueSafronov} we get an isomorphism
\begin{align*}
\Res_{p_1}\T^*_{\cO_{\Bun_G(\Sigma)}\boxtimes K_\Sigma}(\ev^*[X/(G\times\Gm)]/\B(G\times\Gm))&\cong \T^*[-1](\Res_{p_1} \ev^*[X/(G\times\Gm)]/\Bun_G(\Sigma))\\
&\cong \T^*[-1](\Bun_G^X(\Sigma)/\Bun_G(\Sigma)).
\end{align*}

Let $\pi^X\colon \Bun_G^M(\Sigma)\rightarrow \Bun_G^X(\Sigma)$ and $\cF_i\in\bD_{\Nilp}(\Bun_{G_i}(\Sigma))$. By the dimensional reduction theorem, \cref{thm:dimensionalreduction}, we have
\[\pi^X_! \pi^\varphi(\cF_1\boxtimes \cF_2)\cong (\pi^X_1\times \pi^X_2)^*(\cF_1\boxtimes \cF_2)[\dim(\Bun_G^X(\Sigma)/\Bun_G(\Sigma))].\]
Therefore,
\begin{align*}
\R\Gamma_c(\Bun_G^M(\Sigma), \pi^\varphi(\cF_1\boxtimes \cF_2))&\cong \R\Gamma_c(\Bun_G^X(\Sigma), (\pi^X_1)^*\cF_1\otimes (\pi^X_2)^*\cF_2)[\dim(\Bun_G^X(\Sigma)/\Bun_G(\Sigma))] \\
&\cong \R\Gamma_c(\Bun_{G_2}(\Sigma), ((\pi^X_2)_! (\pi^X_1)^*\cF_1)\otimes \cF_2)[\dim(\Bun_G^X(\Sigma)/\Bun_G(\Sigma))] \\
&\cong \R\Gamma_c(\Bun_{G_2}(\Sigma), \sfP((\pi^X_2)_! (\pi^X_1)^*\cF_1)\otimes \cF_2)[\dim(\Bun_G^X(\Sigma)/\Bun_G(\Sigma))],
\end{align*}
where we have used the projection formula in the second line and \cite[Proposition 3.4.6]{AGKRRV2} in the third line (where we note that $\sfP(\cF_2)\cong\cF_2$ since $\cF_2$ has nilpotent singular support. The claim then follows from the definition of the induction functor.
\end{proof}

\begin{example}\label{ex:Eisensteinperiod}
Let $G$ be a connected reductive group, $P\subset G$ a parabolic subgroup, $L$ the Levi factor and $U\subset P$ the unipotent radical. Consider the trivial $\Gm$-action on $G$ and $L$. Consider the $G\times L$-space $X=G/U$ equipped with the trivial $\Gm$-action. In this case $\Bun_G^X(\Sigma) \cong \Bun_P(\Sigma)$ is the moduli stack of $P$-bundles. The correspondence \eqref{eq:BunXcorrespondence} in this case is
\[
\xymatrix{
& \Bun_P(\Sigma) \ar_{\pi_L}[dl] \ar^{\pi_G}[dr] & \\
\Bun_L(\Sigma) && \Bun_G(\Sigma).
}
\]
Applying \cref{prop:cotangentperiod} we obtain that
\[\ind^{L\rightarrow G}_{\T^*(G/U)}(-)=\mathrm{Eis}_!(-)\coloneqq \sfP((\pi_G)_! (\pi_L)^*(-))[\dim(\Bun_P(\Sigma))-\dim(\Bun_L(\Sigma))]\]
is the functor of Eisenstein series and
\[\ind^{G\rightarrow L}_{\T^*(G/U)}=\mathrm{CT}_!(-)\coloneqq \sfP((\pi_L)_! (\pi_G)^*(-))[\dim(\Bun_P(\Sigma))-\dim(\Bun_G(\Sigma))]\]
is the constant term functor, where we note that the projector $\sfP$ can be omitted in both functors, see \cite[Section 1.3.5]{FaergemanHayash} and \cite[Proposition 1.4.2]{GaitsgoryRaskin}.
\end{example}

Let us now relate the induction functor to period sheaves as defined in \cite{BZSV}. Consider the case $G_1=\pt$, so that $G=G_2$ and recall the notion of an eigenmeasure from \cite[Section 3.8]{BZSV}.

\begin{definition}
Consider a scheme $X$ with a $G\times\Gm$-action as before. An \defterm{eigenmeasure} on $X$ is a volume form $\vol_X$ on $X$ which has weight $\gamma\in\Z$ with respect to the $\Gm$-action and such that $G$ acts on $\vol_X$ via a character $\eta\colon G\rightarrow \Gm$.
\end{definition}

An eigenmeasure $\vol_X$ identifies
\begin{equation}\label{eq:volX}
\vol_X\colon \cO^G(\eta)\otimes \cO^{\Gm}(\gamma)\xrightarrow{\sim} K_X
\end{equation}
as $G\times\Gm$-equivariant line bundles on $X$, where $\cO^G(\eta)$ is the trivial line bundle corresponding to the one-dimensional $G$-representation $\eta$ and $\cO^{\Gm}(\gamma)$ is the trivial line bundle corresponding to the one-dimensional $\Gm$-representation with weight $\gamma$.

\begin{proposition}\label{prop:BZSVcomparison}
Consider the pair $(X, G)$ as above together with an eigenmeasure $\vol_X$. The induction functor satisfies
\[\ind^{\pt\rightarrow G}_{\T^* X}(\bu)\cong \cP^{\norm}_X,\]
where on the right we have the normalized period sheaf as defined in \cite[Section 10.4]{BZSV}.
\end{proposition}
\begin{proof}
Consider the degree map
\[\deg_\eta\colon \Bun_G(\Sigma)\xrightarrow{\eta} \Bun_{\Gm}(\Sigma)\xrightarrow{\deg}\Z,\]
where the first morphism sends a $G$-torsor $P\rightarrow \Sigma$ to the line bundle $P\times^{\Gm, \eta} \bA^1$. If we denote by $\ev\colon \Bun_G^X(\Sigma)\times \Sigma\rightarrow [X/(G\times\Gm)]$ the tautological evaluation morphism and by $p\colon \Bun_G^X(\Sigma)\times \Sigma\rightarrow \Bun_G^X(\Sigma)$ the projection on the first factor, by \cref{lm:BunGMcotangent} we have
\[\bL_{\Bun_G^X(\Sigma)/\Bun_G(\Sigma)}\cong p_*((\cO_{\Bun_G^X(\Sigma)}\boxtimes K_\Sigma)\otimes \ev^* \Omega^1_X)[1].\]
Using \eqref{eq:volX} the Grothendieck--Riemann--Roch formula computes the relative dimension of $\Bun_G^X(\Sigma)\rightarrow \Bun_G(\Sigma)$ as
\begin{align*}
-&\int_\Sigma (1 + c_1(\Sigma)/2 + \dots)(\dim(X) + c_1(\Sigma)(\gamma/2-\dim(X)) + c_1(P\times^{\Gm, \eta} \bA^1)+\dots) \\
&= (g-1)(\gamma-\dim(X)) - \deg_\eta.
\end{align*}
An analogous computation gives
\[\dim(\Bun_G(\Sigma)) = (g-1)\dim(G),\]
so we conclude that
\[\dim(\Bun_G^X(\Sigma)) = (g-1)(\dim(G) + \gamma - \dim(X)) - \deg_\eta.\]
Combined with \cref{prop:cotangentperiod} we get that
\[\ind^{\pt\rightarrow G}_{\T^* X}(\bu) \cong \sfP((\pi_2^X)_!\bu_{\Bun_G^X(\Sigma)})[(g-1)(\dim(G) + \gamma - \dim(X)) - \deg_\eta],\]
which coincides with the definition of the normalized period sheaf.
\end{proof}

\subsection{Whittaker case}

In this section we analyze a particular example of induction functor giving rise to the Whittaker functor. Let $G$ be a connected reductive group with Lie algebra $\g$, $B\subset G$ a Borel subgroup, $N\subset B$ its unipotent radical and $T\subset B$ the maximal torus. Let $\psi\colon N\rightarrow \Ga$ be a nondegenerate character; we denote by the same letter the induced character of its Lie algebra. Let $2\rho^\vee\colon \Gm\rightarrow T$ be the cocharacter given by the sum of the positive coroots. There is a natural $T$-action on $N$. We refer to the induced $\Gm$-action on $N$ via $2\rho^\vee$ as the canonical $\Gm$-action on $N$. With respect to that action $\psi\colon N\rightarrow \Ga$ has weight $2$.

Consider $G$ with its natural left $G$-action and a right $N$-action. We get that $\T^* G$ is a $G\times N$-Hamiltonian space. Let $M=\T^*G /\!/_\psi N$ be the $N$-Hamiltonian reduction of $\T^* G$ at the moment map value specified by the character $\psi$. Consider the $\Gm$-action on $\T^* G$ defined as the product of the Hamiltonian $\Gm$-action induced by the right $\Gm$-action on $G$ via $2\rho^\vee$ and the weight $2$ action along the fibers of $\T^* G$. It descends to a $\Gm$-action on $M$ which scales the symplectic structure with weight $2$. $M$ also carries a residual Hamiltonian $G$-action, so we are in the setting of \cref{sect:periodsetup}.

\begin{remark}
Consider the moment map $M\rightarrow \g^*$. Then $[M/G]\rightarrow [\g^*/G]$ identifies with the Kostant section.
\end{remark}

\begin{remark}
The underlying classical stack of $\Bun_G^M(\Sigma)$ is isomorphic to the base of the Hitchin fibration. The corresponding $0$-shifted Lagrangian morphism $\mu\colon \Bun_G^M(\Sigma)\rightarrow \T^* \Bun_G(\Sigma)$ to the moduli space of Higgs bundles from \cref{rmk:GaiottoLagrangian} identifies with the Hitchin section.
\end{remark}

Consider the projection $M\rightarrow X=G/N$ which identifies $M$ with the twisted cotangent bundle of $X$. Using the exact sequence
\[0\longrightarrow \Omega^1_X|_M\otimes \cO(-2)\longrightarrow \T_M\longrightarrow \T_X|_M\longrightarrow 0\]
of $G\times \Gm$-equivariant vector bundles on $M$ we construct an anomaly trivialization for the triple $(G, \Sigma, M)$ as in \cref{sect:cotangentperiods}.

Define the moduli stack of $K^{1/2}_\Sigma$-twisted $N$-bundles on $\Sigma$ by the Cartesian diagram
\[
\xymatrix{
\Bun_N^{K^{1/2}}(\Sigma) \ar[r] \ar[d] & \Bun_{N\rtimes \Gm}(\Sigma) \ar[d] \\
\pt \ar^{K^{-1/2}_\Sigma}[r] & \Bun_{\Gm}(\Sigma)
}
\]
Equivalently, it is the moduli stack of $B$-bundles whose underlying $T$-bundle is $\rho^\vee(K_\Sigma)$, the $T$-bundle induced from $K^{1/2}_\Sigma$ via the homomorphism $2\rho^\vee\colon\Gm\rightarrow T$.

\begin{example}
Suppose that $G$ has semisimple rank $1$, so that $\psi\colon N\rightarrow \Ga$ is an isomorphism. In this case $\Bun_N^{K^{1/2}}(\Sigma)$ identifies with the moduli stack of $K_\Sigma$-torsors, where we consider $K_\Sigma\rightarrow \Sigma$ as a group scheme with respect to the addition along the fibers.
\end{example}

The morphism $\psi$ in general allows us to forget a $K^{1/2}_\Sigma$-twisted $N$-bundle to a $K_\Sigma$-torsor. In particular, we get a natural morphism
\[\psi\colon \Bun^{K^{1/2}}_N(\Sigma)\longrightarrow \rH^1(\Sigma, K_\Sigma)\cong \bA^1\]
which, by abuse of notation, we also denote by $\psi$. Let
\[g\colon \Bun^{K^{1/2}}_N(\Sigma)\longrightarrow \Bun_G(\Sigma)\]
be the morphism induced by the homomorphism $N\rightarrow G$, where we note that the $\Gm$-action given by $2\rho^\vee$ on $G$ is inner, so it is trivial on $\B G$. Denote by
\[d = \dim(\Bun^{K^{1/2}}_N(\Sigma))_{\rho^\vee(K_\Sigma)} - \dim(\Bun_G(\Sigma))_{\rho^\vee(K_\Sigma)}\]
the relative dimension of $g$ at $K^{1/2}_\Sigma\in \Bun^{K^{1/2}}_N(\Sigma)$.

For a $\Gm$-action on $\bA^1$ of weight $n$ denote by $[\bA^1/_n\Gm]$ the corresponding quotient stack. For the $\Gm$-action on $\Bun^{K^{1/2}}_N(\Sigma)$ induced by the natural automorphisms of $K^{1/2}_\Sigma$ the morphism $\psi$ has weight $2$ and $g$ is $\Gm$-equivariant. Therefore, we have a correspondence
\[
\xymatrix{
[\bA^1/_2\Gm] & [\Bun^{K^{1/2}}_N(\Sigma)/\Gm] \ar_-{\overline{\psi}}[l] \ar^-{\overline{g}}[r] & \Bun_G(\Sigma).
}
\]

We will now identify the induction functor for $M$ with the Whittaker functor defined as in \cite{NadlerTaylor} and the Whittaker sheaf defined as in \cite{BZSV}.

\begin{theorem}\label{thm:Whittakerperiod}
There are natural isomorphisms
\begin{align*}
\ind^{G\rightarrow \pt}_M(-)&\cong \R\Gamma_c(\Bun_N^{K^{1/2}}(\Sigma), \phi_\psi g^*(-))[d]\\
&\cong \R\Gamma_c([\Bun_N^{K^{1/2}}(\Sigma)/\Gm], \overline{g}^*(-)\otimes \overline{\psi}^*\exp_2)[d]
\end{align*}
and
\[\ind^{\pt\rightarrow G}_M(\bu)\cong \sfP(\overline{g}_! \overline{\psi}^*\exp_2)[\dim(\Bun^{K^{1/2}}_N(\Sigma))_{\rho^\vee(K_\Sigma)}].\]
\end{theorem}
\begin{proof}
We may identify the exact $0$-shifted symplectic structure on $[M/G]\rightarrow \B G$ with the relative twisted cotangent bundle of $\B N\rightarrow \B G$, where the twist is by the differential of $\psi\colon \B N\rightarrow \B\Ga$. Analogously to the proof of \cref{prop:cotangentperiod} this implies that the exact $(-1)$-shifted symplectic structure on $\Bun_G^M(\Sigma)\rightarrow \Bun_G(\Sigma)$ is identified with the one on the relative derived critical locus
\[\R\Crit_{\Bun_N^{K^{1/2}}(\Sigma)/\Bun_G(\Sigma)}(\psi)\longrightarrow \Bun_G(\Sigma).\]

For a natural number $n$ and a point $x\in \Sigma$ consider the following diagram from \cite[(3.11)]{NadlerTaylor}:
\[
\xymatrix{
\Bun_N^{K^{1/2}}(\Sigma, nx) \ar^{i'}[r] \ar^{p'}[d] & \Bun_G^{K^{1/2}}(\Sigma, nx) \ar^{p}[d] \\
\Bun_N^{K^{1/2}}(\Sigma) \ar^{i}[r] \ar_{g}[dr] & \Bun^{K^{1/2}}_{G, N}(\Sigma, nx) \ar^{h}[d] \\
& \Bun_G(\Sigma),
}
\]
where the square is Cartesian and the individual spaces are as follows:
\begin{itemize}
    \item $\Bun^{K^{1/2}}_{G, N}(\Sigma, nx)$ is the moduli stack of $G$-bundles on $\Sigma$ with a $B$-reduction on the $n$-th order neighborhood $D_n(x)$ of $x$ whose underlying $T$-bundle is $2\rho^\vee(K^{1/2}_\Sigma)$.
    \item $\Bun_G^{K^{1/2}}(\Sigma, nx)$ is the moduli stack of $G$-bundles on $\Sigma$ with a reduction of the $G$-bundle on $D_n(x)$ to $K^{1/2}_\Sigma$ via the homomorphism $\Gm\xrightarrow{2\rho^\vee} T\subset G$.
    \item $\Bun_N^{K^{1/2}}(\Sigma, nx)$ is the moduli stack of $B$-bundles on $\Sigma$ together with an isomorphism of the underlying $T$-bundle to $2\rho^\vee(K^{1/2}_\Sigma)$ and a $T$-reduction on $D_n(x)$.
\end{itemize}
The morphism $h$ is smooth and $p$ and $p'$ are smooth affine bundles of dimension $n\dim N$. The whole diagram carries a $\Gm$-action given by the obvious $\Gm$-action on $K^{1/2}_\Sigma$. Let $\psi'\colon \Bun^{K^{1/2}}_N(\Sigma, nx)\rightarrow \bA^1$ be the pullback of $\psi$, which has weight $2$. Moreover, it is shown in \cite[Section 3.1]{NadlerTaylor} that we may choose $n$ large enough with the following properties:
\begin{itemize}
    \item For an open substack $U\subset \Bun_G(\Sigma)$ containing the image of $\Bun_N^{K^{1/2}}(\Sigma)$ the stack $\Bun_G^{K^{1/2}}(\Sigma, nx)|_U$ is a scheme.
    \item $i$ and $i'$ are closed immersions.
\end{itemize}
Let $\pi\colon \Bun_G^M(\Sigma)=\R\Crit_g(\psi)\rightarrow \Bun_G(\Sigma)$, $\nu\colon \R\Crit_i(\psi)\rightarrow \Bun_{G, N}^{K^{1/2}}(\Sigma, nx)$ and $\nu'\colon \R\Crit_{i'}(\psi')\rightarrow \Bun_G^{K^{1/2}}(\Sigma, nx)$, so that we have a Cartesian diagram
\[
\xymatrix{
\R\Crit_{i'}(\psi') \ar^{\overline{p}}[d] \ar^-{\nu'}[r] & \Bun_G^{K^{1/2}}(\Sigma, nx) \ar^{p}[d] \\
\R\Crit_i(\psi) \ar^-{\nu}[r] & \Bun_{G, N}^{K^{1/2}}(\Sigma, nx)
}
\]
Let
\[\tilde{\Nilp} = \Nilp\times_{\Bun_G(\Sigma)} \Bun_{G, N}^{K^{1/2}}(\Sigma, nx)\subset \T^*\Bun_{G, N}^{K^{1/2}}(\Sigma, nx)\]
and
\[\Nilp' = \Nilp\times_{\Bun_G(\Sigma)} \Bun_G^{K^{1/2}}(\Sigma, nx)\subset \T^* \Bun_G^{K^{1/2}}(\Sigma, nx)\]
be the pullbacks of the global nilpotent cones. For $\cF\in\bD_{\Nilp}(\Bun_G(\Sigma))$ we have an isomorphism
\[\gamma_h\colon \R\Gamma_c(\R\Crit_i(\psi), \nu^\varphi h^\dag \cF)\xrightarrow{\sim} \R\Gamma_c(\R\Crit_g(\psi), \pi^\varphi\cF).\]
By \cite[Proposition 3.3.2]{NadlerTaylor} $\R\Crit_i(\psi)$ and $\tilde{\Nilp}$ intersect in $\T^*\Bun^{K^{1/2}}_{G, N}(\Sigma, nx)$ at a single point $c\in\R\Crit_i(\psi)$. Therefore, by \cref{prop:microsupportestimate} we see that $\nu^\varphi h^\dag \cF$ is supported at $c$ and hence
\[\R\Gamma_c(\R\Crit_i(\psi), \nu^\varphi h^\dag \cF)\cong (\nu^\varphi h^\dag \cF)_c.\]
Let $c'\in\R\Crit_{i'}(\psi')$ be any preimage of $c$. Then
\[(\nu^\varphi h^\dag \cF)_c\cong (\overline{p}^\dag \nu^\varphi h^\dag \cF)_{c'}[-n\dim N]\xleftarrow[\sim]{\alpha_{p, \overline{p}}} ((\nu')^\varphi p^\dag h^\dag \cF)_{c'}[-n\dim N].\]
We may find an open subset $U \subset \Bun_G^{K^{1/2}}(\Sigma, nx)$ containing the image of $c'$ together with a vector bundle $E\rightarrow U$ and a section $s$ whose zero locus coincides with $\Bun_N^{K^{1/2}}(\Sigma, nx)|_U\subset U$. Let $s^\vee\colon E^\vee\rightarrow \bA^1$ the corresponding cosection and $r\colon E^\vee\rightarrow U$ the projection. Moreover, we may assume that $\psi'|_U\colon \Bun_N^{K^{1/2}}(\Sigma, nx)|_U\rightarrow \bA^1$ extends to a function $\tilde{\psi}$ on $U$. As in \cref{lem-crit-locus-vs-shifted-cotangent} we identify $\R\Crit_{i'}(\psi'|_U)\cong \R\Crit_{E^\vee/U}(r^\ast\tilde{\psi} + s^\vee)$. Thus, we obtain a critical chart for $\nu'\colon \R\Crit_{i'}(\psi')\rightarrow \Bun_G^{K^{1/2}}(\Sigma, nx)$ near $c'\in\R\Crit_{i'}(\psi')$ given by the smooth morphism $r\colon E^\vee\rightarrow U$ equipped with the function $\tilde{\psi} + s^\vee$. Therefore,
\[((\nu')^\varphi p^\dag h^\dag \cF)_{c'}[-n\dim N]\cong \phi_{r^\ast \tilde{\psi} + s^\vee}(r^\dag (p^\dag h^\dag \cF)|_U)_{c'}[-n\dim N].\]
Let $\cF' = (p^\dag h^\dag \cF)|_U\in\bD_{\Nilp'}(U)$. By \cite[Proposition 3.3.3]{NadlerTaylor} the graph $\Gamma_{d(r^\ast \tilde{\psi} + s^\vee)}$ intersects $\Nilp'$ in $\T^* \Bun_G^{K^{1/2}}(\Sigma, nx)$ cleanly with excess $e=2n\dim N$ (note that this is the \emph{real} dimension of the intersection). By \cref{prop:microstalk} we get that $(\phi_{r^\ast \tilde{\psi} + s^\vee} r^\dag \cF')_{c'}[-n\dim N]$ computes the microstalk of $r^\dag\cF'$. By smooth functoriality of microstalks (\cref{prop:microstalkpullback}) it is naturally isomorphic to the microstalk of $p^\dag h^\dag \cF$ at $c'$. Thus, we obtain a natural isomorphism
\[\mu_{c'}(p^\dag h^\dag\cF)\cong \R\Gamma_c(\R\Crit_g(\psi), \pi^\varphi\cF).\]
By \cite[Theorem 3.3.4]{NadlerTaylor} 
\[\R\Gamma_c(\Bun_N^{K^{1/2}}(\Sigma), \phi_\psi g^*(-))[d]\colon \bD_{\Nilp}(\Bun_G(\Sigma))\longrightarrow \Mod_R\]
is naturally isomorphic to a shifted microstalk functor of $p^\dag h^\dag \cF$. But since both functors commute with Verdier duality, the shifts coincide. Thus, we obtain a natural isomorphism
\[\R\Gamma_c(\Bun_N^{K^{1/2}}(\Sigma), \phi_\psi g^*(-))[d]\cong \R\Gamma_c(\Bun_G^M(\Sigma), \pi^\varphi(-)).\]
This finishes the construction of the first isomorphism.

The $\Gm$-action on $\Bun^{K^{1/2}}_N(\Sigma, nx)$ is contracting and $\psi'\colon \Bun^{K^{1/2}}_N(\Sigma, nx)\rightarrow \bA^1$ has weight $2$. Thus, applying \cref{prop:exponentialvanishing}(2) and descending the isomorphism along $p'$ (which is a smooth affine bundle) we get an isomorphism
\[\R\Gamma_c(\Bun_N^{K^{1/2}}(\Sigma), \phi_\psi g^*(-))[d]\cong \R\Gamma_c([\Bun_N^{K^{1/2}}(\Sigma)/\Gm], \overline{g}^*(-)\otimes \overline{\psi}^*\exp_2)\]
which establishes the second isomorphism.

The third isomorphism follows immediately from above and the definition of the induction functor.
\end{proof}

\begin{remark}
Let us relate the above Whittaker sheaf $\overline{g}_! \overline{\psi}^*\exp_2\in\bD(\Bun_G(\Sigma))$ to the one defined in \cite[Section 2.5]{NadlerYun}. The standing assumption in \cite{NadlerYun} is that $2\rho^\vee\colon\Gm\rightarrow T$ is the square of a cocharacter $\rho^\vee\colon \Gm\rightarrow T$. In that case the original $\Gm$-action on $\Bun^{K^{1/2}}_N(\Sigma)$ admits a square root. Thus, we have a diagram
\[
\xymatrix{
& [\Bun_N^{K^{1/2}}(\Sigma)/_2\Gm] \ar^-{\overline{\psi}}[r] \ar[d] \ar_{\overline{g}}[dl] & [\bA^1/_2\Gm] \ar^{s_2}[d] \\
\Bun_G(\Sigma) & [\Bun_N^{K^{1/2}}(\Sigma)/_1\Gm] \ar^-{\tilde{\psi}}[r] \ar_-{\tilde{g}}[l] & [\bA^1/_1\Gm] \\
}
\]
where the square is Cartesian. Therefore, using the isomorphism $(s_2)_!\exp_2\cong \exp_1$ from \ref{prop:expproperties}(3) and base change we get
\[\overline{g}_! \overline{\psi}^*\exp_2\cong \tilde{g}_! \tilde{\psi}^* \exp_1.\]
\end{remark}

\subsection{Symplectic representations}

In this section we restrict to the case when $G_1=\pt$, so that $G=G_2$ with the trivial $\Gm$-action and $M$ is a symplectic $G$-representation with the $\Gm$-action on $M$ by scaling. As an example, consider $X$ to be a $G$-representation with the $\Gm$-action by scaling. The induced Hamiltonian $\Gm$-action on $M=\T^* X=X\oplus X^\vee$ is given by weight $1$ on $X$ and weight $-1$ on $X^\vee$. Therefore, the natural $\Gm$-action on $M$ described in \cref{sect:cotangentperiods} is given by weight $1$ on both $X$ and $X^\vee$. As our first observation, we show that the symplectic nature of the construction of periods encodes the categorification of the functional equation for $L$-functions.

\begin{proposition}\label{prop:categorifiedfunctionalequation}
Let $X$ be a $G$-representation and let $X^\vee$ be its dual. Consider the scaling $\Gm$-actions on both. Then there is an isomorphism
\[\cP^{\norm}_X\cong \cP^{\norm}_{X^\vee}\in\bD_{\Nilp}(\Bun_G(\Sigma)).\]
\end{proposition}
\begin{proof}
Let $M=\T^* X=X\oplus X^\vee$ and $M'=\T^* X^\vee=X\oplus X^\vee$. We have a natural symplectomorphism $M\cong M'$ given by $(q, p)\mapsto (-p, q)$ which intertwines the $\Gm$-actions. In particular, there is an isomorphism between $\Bun_G^M(\Sigma)\rightarrow \Bun_G(\Sigma)$ and $\Bun_G^{M'}(\Sigma)\rightarrow \Bun_G(\Sigma)$ which preserves the relative d-critical structures over $\Bun_G(\Sigma)$. Therefore, we get a natural isomorphism of the induction functors $\ind_M\cong \ind_{M'}$. The claim then follows from the combination of isomorphisms
\[\cP^{\norm}_X\cong \ind_M(\pt)\cong \ind_{M'}(\pt) \cong \cP^{\norm}_{X^\vee},\]
where the first and the third isomorphisms are established in \cref{prop:BZSVcomparison}.
\end{proof}

\begin{remark}
For $X=\bA^1$ equipped with the scaling $G=\Gm$-action (the Iwasawa--Tate case in the terminology of \cite{BZSV}) the above isomorphism was also established in \cite[Proposition 6.5.1]{FengWang} and \cite[Lemma 7.3.6]{BravermanGaitsgory}.
\end{remark}

\begin{remark}
The above isomorphism is an example of the Fourier--Kashiwara isomorphism from \cref{thm-Kashiwara-Fourier} using that $\Bun_G^X(\Sigma)\rightarrow \Bun_G(\Sigma)$ is a perfect complex and $\Bun_G^{X^\vee}(\Sigma)\rightarrow \Bun_G(\Sigma)$ is a shift of its dual.
\end{remark}

Let us make a comment about about anomaly for symplectic $G$-representations. Let $\pi\colon [M/(G\times\Gm)]\rightarrow \B(G\times\Gm)$ be the projection. By $\bA^1$-homotopy invariance of \'etale cohomology we get that
\[c_2^{G\times\Gm}(\T_M\otimes \cO(1))\sim \pi^* c_2^{G\times\Gm}(\T_{M, 0}\otimes \cO(1)).\]
By assumption $\T_{M, 0}\otimes \cO(1)$ has weight $2$ as a $\Gm$-representation. Since $\cO(2)$ has a canonical square root given by $\cO(1)$, we get
\[c_2^{G\times\Gm}(\T_{M, 0}\otimes \cO(1)) \sim c_2^G(M).\]
In particular, the orientation gerbe of $\Bun_G^M(\Sigma)\rightarrow \Bun_G(\Sigma)$ is pulled back from the gerbe $\int_\Sigma \ev^* c_2^G(M)$ on $\Bun_G(\Sigma)$ (where $\ev\colon \Bun_G(\Sigma)\times \Sigma\rightarrow \B G$ is the evaluation morphism) and the period sheaf $\cP^{\symp}_M=\ind^{\pt\rightarrow G}_M(\bu)$ in this case always makes sense as a \emph{twisted} ind-constructible complex
\[\cP^{\symp}_M\in \bD^{\int_\Sigma \ev^* c_2^G(M)}_{\Nilp}(\Bun_G(\Sigma)).\]
Choose a Borel subgroup $B\subset G$ and let $T=B/[B, B]$. Let $\Lambda^\vee=\Hom(\Gm, T)$ be the coweight lattice and $W$ the Weyl group. By \cite[Theorem 3.2.6]{GaitsgoryLysenko} the maps
\[\rH^4_{\et}(\B G, \mu_2^{\otimes 2})\longrightarrow \rH^4_{\et}(\B B, \mu_2^{\otimes 2})\xleftarrow{\sim} \rH^4_{\et}(\B T, \mu_2^{\otimes 2})\]
determine an inclusion from $\rH^4_{\et}(\B G, \mu_2^{\otimes 2})$ to the group of $W$-invariant quadratic forms $\Lambda^\vee\rightarrow \Z/2\Z$. Thus, we get the following:
\begin{itemize}
    \item There is an anomaly trivialization if, and only if, the quadratic form $\Lambda^\vee\rightarrow \Z/2\Z$ associated to the symplectic $G$-representation $M$ is zero.
    \item The anomaly trivialization is unique if $\rH^3_{\et}(\B G, \mu_2^{\otimes 2})\cong \Ext^1(\pi_1(G), \mu_2) = 0$.
\end{itemize}

For instance, for $G=\mathrm{Sp}_{2n}$ and $M$ the standard $2n$-dimensional symplectic $G$-representation the corresponding gerbe $\int_\Sigma \ev^* c_2^G(M)$ on $\Bun_G(\Sigma)$ is nontrivial; its total space is the moduli stack of metaplectic $G$-bundles on $\Sigma$. We expect the corresponding period sheaf $\cP^{\symp}_M$ to be isomorphic to the theta sheaf as defined in \cite{LysenkoTheta}.

\printbibliography

\end{document}